\documentclass[a4paper,12pt]{article}

\usepackage{authblk}
\usepackage{parskip}
\usepackage{mathtools}
\usepackage{amssymb}
\usepackage{amsmath}
\usepackage{latexsym}
\usepackage{mathrsfs}  
\usepackage{bbm}       
\usepackage{amsthm}    
\usepackage{relsize}   
\usepackage{graphicx}
\usepackage{thmtools}  
\usepackage{mathrsfs}  
\usepackage{paralist}  
\usepackage{lipsum}    
\usepackage[all]{xy}   

\numberwithin{equation}{section}

\usepackage{titlesec}   

\titleformat{\section}[block]{\bfseries\filcenter}
{{\upshape\thesection\enspace}}{.5em}{}

\titleformat{\subsection}[block]{\filcenter}
{{\upshape\thesubsection\enspace}}{.5em}{} 

\titleformat{\subsubsection}[block]{\filcenter}
{{\upshape\thesubsubsection\enspace}}{.5em}{} 

\usepackage{enumitem}  
\setlist{nosep}  
\setitemize[0]{leftmargin=*}
\setenumerate[0]{leftmargin=*}
\setenumerate[1]{label={(\arabic*)}} 

\newcommand{\N}{\mathbb{N}}     
\newcommand{\R}{\mathbb{R}}     
\newcommand{\C}{\mathbb{C}}     
\newcommand{\Prob}{\mathbb{P}}  
\newcommand{\Exp}{\mathbb{E}}   
\newcommand{\goth}[1]{\mathfrak{#1}} 
\newcommand{\ind}[2]{\mathbbm{1}_{#1}\left( #2 \right)}          
\newcommand{\inner}[2]{\left\langle #1 \, , \, #2 \right\rangle} 
\newcommand{\norm}[1]{\left|\left|#1\right|\right|}              
\newcommand{\triplet}[3]{\left( #1, #2, #3 \right) }             
\newcommand{\ProbSpace}{\triplet{\Omega}{\mathscr{F}}{\Prob}}    
\newcommand{\abs}[1]{\left| #1 \right|}                          

\newcommand{\defeq}{\mathrel{\mathop:}=}                         
\newcommand\restr[2]{{
  \left.\kern-\nulldelimiterspace 
  #1 
  \vphantom{\big|} 
  \right|_{#2} 
  }}
  
\makeatletter
\newsavebox{\@brx}
\newcommand{\llangle}[1][]{\savebox{\@brx}{\(\m@th{#1\langle}\)}%
  \mathopen{\copy\@brx\kern-0.5\wd\@brx\usebox{\@brx}}}
\newcommand{\rrangle}[1][]{\savebox{\@brx}{\(\m@th{#1\rangle}\)}%
  \mathclose{\copy\@brx\kern-0.5\wd\@brx\usebox{\@brx}}}
\makeatother

\theoremstyle{plain} 
\newtheorem{theorem}{Theorem}[section]    
\newtheorem{proposition}[theorem]{Proposition} 
\newtheorem{corollary}[theorem]{Corollary}
\newtheorem{lemma}[theorem]{Lemma}
\newtheorem{assumption}[theorem]{Assumption}

\theoremstyle{definition} 
\newtheorem{definition}[theorem]{Definition}
\newtheorem{example}[theorem]{Example}
\newtheorem{remark}[theorem]{Remark}

 \title{Stochastic integration with respect to cylindrical semimartingales}
\author{C. A. Fonseca-Mora}
\affil{  Escuela de Matem\'{a}tica, Universidad de Costa Rica, \\ San Jos\'{e}, 11501-2060, Costa Rica. 

\noindent E-mail:  christianandres.fonseca@ucr.ac.cr }

\date{}

\begin{document}

 \maketitle

\abstract{In this work we introduce a theory of  stochastic integration with respect to general cylindrical semimartingales defined on a locally convex space $\Phi$. Our construction of the stochastic integral is based on the theory of tensor products of topological vector spaces and the property of good integrators of  real-valued semimartingales. This theory is further developed in the case where $\Phi$ is a complete, barrelled, nuclear space, where we obtain a complete description of the class of integrands as $\Phi$-valued locally bounded and weakly predictable processes. Several other properties of the stochastic integral are proven, including a Riemann representation, a stochastic integration by parts formula and a stochastic Fubini theorem. Our theory is then applied to provide sufficient and necessary conditions for existence and uniqueness of solutions to linear stochastic evolution equations driven by  semimartingale noise taking values in the strong dual $\Phi'$ of $\Phi$. In the last part of this article we apply our theory to define stochastic integrals with respect to a sequence of real-valued semimartingales.} 

\smallskip

\emph{MSC2020 subject classifications.}  Primary 60H05; Secondary 60B11, 60G20, 60G48. 


\section{Introduction}\label{sectIntroduction}

Cylindrical semimartingales play an important role as the driving noise for stochastic partial differential equations. For many decades the most widely used model of cylindrical semimartingale in stochastic analysis was the cylindrical Brownian motion (see e.g. \cite{Criens:2019, DaPratoZabczyk, Kalinichenko:2019, vanNeervenVeraarWeis:2008}).  However, in recent years there has been an increasing interest in the usage of other classes of cylindrical semimartingales as such driving noise; we can cite, for example, the cylindrical L\'{e}vy processes \cite{JakubowskiRiedle:2017, KosmalaRiedle:2021, KumarRiedle:2020,  PriolaZabczyk:2011, SunXieXie:2020}, cylindrical martingale-valued measures \cite{FonsecaMora:2018-1} and cylindrical continuous local martingales \cite{MikuleviciusRozovskii:1998, MikuleviciusRozovskii:1999, VeraarYaroslavtsev:2016}. To the extent of our knowledge, none of these works considers stochastic integration with respect to general cylindrical semimartingales in a general locally convex space. 

In this work, our main objective is to introduce a theory of stochastic integration for cylindrical semimartingales in the topological dual $\Phi'$ to a locally convex space $\Phi$. Recall that a cylindrical semimartingale $X$ in $\Phi'$ is a linear operator such that for each $\phi \in \Phi$, $X(\phi)$ is a real-valued semimartingale. In practice, in order to develop a theory of integration one usually ask for some appropriate continuity properties on $X$ as an operator from $\Phi$ into the space of real-valued semimartingales (see Assumption \ref{assuCylSemimartingale}).

One can easily define the stochastic integral for elementary weakly predictable integrands, however, when one wants to consider more general integrands, one faces the difficulties inherent to the absence of estimates due to the generality of $\Phi$ and $X$. Moreover, even for the case when $\Phi$ is a Banach space it is pointed out in (\cite{JakubowskiRiedle:2017}, Remark 2.2) that  general cylindrical semimartingales do not enjoy a semimartingale decomposition in terms of a cylindrical local martingale and another cylindrical process, then making the classical approach to define the stochastic integral in finite dimensions non-viable in our infinite dimensional setting. Consequently, our work requires a different approach which the reader will see, is based on the theory of tensor products on topological vector spaces, and the property of good integrators for real-valued semimartingales. All the necessary properties of locally convex spaces and tensor products, as well as properties of cylindrical processes  will be given in Section \ref{sectPrelimi}.     

The construction of the stochastic integral is the main object of our study in Section \ref{sectionConstruStochInteg}. In the next paragraphs we will try to give a description of the approach used in this construction. 

Consider a probability space $\ProbSpace$ equipped with a filtration  and denote by $b\mathcal{P}$ the Banach space of all the real-valued bounded predictable processes equipped with the norm of uniform convergence on $[0,\infty) \times \Omega$. Let $S^{0}$ denotes the space of real-valued semimartingales equipped with Emery's topology. Given $\phi \in \Phi$ and $h \in b\mathcal{P}$, one can define a bilinear form $J$ on $\Phi \times b\mathcal{P}$  with values in $S^{0}$ by the prescription $J(\phi,h)= h \cdot X(\phi)$, where $h \cdot X(\phi)$ refers to the stochastic integral of $h$ with respect to the real-valued semimartingale  $X(\phi)$. If  
$\Phi \, \widehat{\otimes}_{\nu} \, b\mathcal{P}$ denotes the completed topological linear tensor product of $\Phi$ and $b\mathcal{P}$, by using the good integrator property of real-valued semimartingales and the continuity property of our cylindrical semimartingales one can show that the bilinear form $J$ is jointly continuous, and hence, defines via its linearization a continuous linear operator $I: \Phi \, \widehat{\otimes}_{\nu} \, b\mathcal{P} \rightarrow S^{0}$ satisfying $I(\phi \otimes h)=J(\phi,h)=h \cdot X(\phi)$ (see Theorem  \ref{theoExisWeakStochIntegral}). We consider the operator $I$ as a first step in our construction of the stochastic integral. Indeed, one can show that the operator $I$ behaves nicely in terms of linearity on the integrators and by taking continuous parts and under stopping times. The above construction and properties of this general integral are carried out in Section \ref{subSectionConstStocIntegral}. 

The operator $I$ introduced above has the problem that the integrands are elements of the topological vector space $\Phi \, \widehat{\otimes}_{\nu} \, b\mathcal{P}$, which is not, in general, locally convex and hence one might not be able to find a ``nice'' description for its elements.  To overcome this obstacle we utilize a convexification of topologies argument to show that the mapping $I$ defines a unique continuous linear mapping from $\Phi \, \widehat{\otimes}_{\pi} \, b\mathcal{P}$ into $(S^{0})_{lcx}$ such that for each $(\phi,h) \in \Phi \times b\mathcal{P}$ we have $I(\phi \otimes h)=h \cdot X(\phi)$ (see Theorem \ref{theoContiWeakIntegConvexifSemiSpace}). Here, 
$\Phi \, \widehat{\otimes}_{\pi} \, b\mathcal{P}$ denotes the completed locally convex tensor product of $\Phi$ and 
$b\mathcal{P}$, and $(S^{0})_{lcx}$ denotes the space $S^{0}$ equipped with the convexification of Emery's topology. We then define our stochastic integral as the image of $H \in \Phi \, \widehat{\otimes}_{\pi} \, b\mathcal{P}$ under the mapping $I$. 
The above extension of the stochastic integral to locally convex integrands and the study of some further properties of the stochastic integral is carried out in Section \ref{subSecLocalConvInteg}.   
 
In Section \ref{sectStochIntegNuclear} we apply our general theory of stochastic integration to the construction of stochastic integrals for weakly predictable and weakly bounded stochastic processes taking values in a reflexive nuclear space $\Phi$ and with respect to cylindrical semimartingales in $\Phi'$. Before we describe our theory of integration in the context of nuclear spaces, we will explain our motivations. 

The theory of stochastic integration with respect to semimartingales taking values in the dual of a nuclear space was initiated in a series of papers by A. S. \"{U}st\"{u}nel (see \cite{Ustunel:1982, Ustunel:1982-1, Ustunel:1982-2}). There,  
under the assumption that $\Phi$ is a complete, bornological, reflexive nuclear space whose strong dual $\Phi'$ is complete and nuclear, \"{U}st\"{u}nel introduced the concept of projective system of  semimartingales in $\Phi'$, and developed a theory of stochastic integration using his concept of projective system and the theory of stochastic integration with respect to semimartingales taking values in a separable Hilbert space. The approach of \"{U}st\"{u}nel was proven to be very useful in applications, but the concept of projective system of semimartingales is strongly tied to the assumption that $\Phi'$ is complete and nuclear, therefore making it impossible to be extended to general nuclear spaces, since it is often the case that $\Phi'$ is neither nuclear nor complete (e.g. if $S$ is any uncountable set, $\R^{S}$ is nuclear but $(\R^{S})'$ is not nuclear; see Remark 51.1 in \cite{Treves}).

Motivated by the above, in \cite{FonsecaMora:Semi} the present author initiated a systematic study of semimartingales in duals of general nuclear spaces by considering them as $\Phi'$-valued processes whose induced cylindrical process is a cylindrical semimartingale in $\Phi'$. In \cite{FonsecaMora:Semi} we obtained several criteria to determine when a given cylindrical semimartingale in $\Phi'$ defines a $\Phi'$-valued c\`{a}dl\`{a}g semimartingale (this is known as the technique of \emph{regularization}), and we proved in \cite{FonsecaMora:Semi} that semimartingales in $\Phi'$ do enjoy a canonical semimartingale representation as the sum of a predictable process which is weakly of bounded variation, a continuous process which is weakly a continuous local martingale, and  a c\`{a}dl\`{a}g process which is weakly a purely discontinuous local martingale. 
However, such a decomposition cannot be used to define a stochastic integral with respect to semimartingales in the dual of a nuclear space. The reasons behind this fact is that there does not exist a theory of stochastic integration for a $\Phi'$-valued c\`{a}dl\`{a}g process which is weakly a local martingale nor for a $\Phi'$-valued c\`{a}dl\`{a}g process which is weakly of bounded  variation. Therefore, a different approach has to be used and this is where our general theory of stochastic integration with respect to cylindrical semimartingales in duals of locally convex spaces fits precisely to fulfil our needs.

In Section \ref{sectNuclearSpaceCylSemi} we review some properties of cylindrical semimartingales and semimartingales in duals of nuclear spaces. Then in Section \ref{sectionCharacWeakIntegransNuclear} we show that when $\Phi$ is a complete barrelled nuclear space,  the collection $b\mathcal{P}(\Phi)$ of all $\Phi$-valued weakly predictable and weakly adapted processes is isomorphic to $\Phi \, \widehat{\otimes}_{\pi} \, b\mathcal{P}$, and therefore we will show that the stochastic integral mapping $I$, defines a linear continuous operator from $b\mathcal{P}(\Phi)$ (with some topology of uniform convergence) into $(S^{0})_{lcx}$. This shows the existence of the stochastic integral for every $\Phi$-valued weakly predictable and weakly adapted process. We would like to stress that the above does not imply in general that $I$ is continuous from  
$b\mathcal{P}(\Phi)$ into $S^{0}$. However, if for a cylindrical semimartingale $X$ such a continuity for the mapping $I$ is satisfied, we will refer to such an $X$ as a ``good integrator'' and further properties of the stochastic integral in this particular setting will be explored. It is worth to mention that our concept of good integrator differs from that introduced in \cite{KurtzProtter:1996} for cylindrical semimartingales (there referred as $\mathbbm{H}^{\#}$-semimartingales) on a separable Banach space $\mathbbm{H}$. 

In Section \ref{sectExteStocIntegNuclear} we carry on a further study of the stochastic integral in nuclear spaces by extending the integral mapping to locally bounded integrands and studying its continuity properties for left-continuous integrands. The study of locally bounded integrands and the extension of the stochastic integral to these integrands is carried out in Section \ref{subSectStocIngLocalBound}. We will see that, in particular, our extended integral allows us to define the stochastic integral for a larger class of nuclear spaces than those considered by \"{U}st\"{u}nel. In Section \ref{subSectRiemanRepre} we explore the continuity properties of the integral mapping defined by a $\Phi'$-valued c\`{a}dl\`{a}g semimartingale that is a ``good integrator'', and we show that a Riemann representation and a integration by parts formula hold for these types of integrators (see Theorems \ref{theoRiemannRepresentation} and \ref{theoIntegByPartsFormula}).

In Section \ref{sectStochFubiniSEE} we pursue further extensions of our theory of stochastic integration and applications to linear stochastic evolution equations driven by  semimartingale noise taking values in the strong dual $\Phi'$ of $\Phi$. First, in Section \ref{subsectStochFubiniTheorem} we formulate and prove a stochastic Fubini theorem for semimartingales which are ``good integrators'' (Theorem \ref{theoStochFubini}). We are not aware of any other work in the literature that proved a stochastic Fubini theorem for semimartingales in the dual of a nuclear space. Later, in Section \ref{subsectSEECylSemiNoise} we to formulate sufficient and necessary conditions for the existence and uniqueness of weak solutions to the following class of stochastic evolution equations
$$ dY_{t}=A' Y_{t} +  dX_{t}, t \geq 0, $$
where $X$ is $\Phi'$-valued semimartingale which is a good integrator, and $A$ is the generator of a strongly continuous $C_{0}$-semigroup $(S(t): t \geq 0) \subseteq \mathcal{L}(\Phi,\Phi)$ (Theorems \ref{theoExisSolutionsSEE} and \ref{theoUniqueSoluSEE}). Our stochastic Fubini theorem will play a central role in our arguments.  Sufficient conditions for the existence of a unique solution with continuous paths are also given (Theorem \ref{theoExisUniqCadlagSoluSEE}). 

Later, in Section \ref{sectAppliSequenceSemi} we apply our theory of stochastic integration in nuclear spaces to carry out an alternative construction for the stochastic integral with respect to a sequence of real-valued semimartingales introduced in \cite{DeDonnoPratelli:2006}. Some new properties of the stochastic integral are obtained as part of our construction. 

Finally, in Section \ref{sectRemarkComparisons} we include a literature review of other theories of stochastic integration in locally convex spaces and comparisons with the theory developed in this article.

\section{Preliminaries}\label{sectPrelimi}

\subsection{Locally Convex Spaces}\label{secNotaDefi}

In this section we introduce our notation and review some of the key concepts on locally convex spaces and operators that we will need throughout this paper. For more
information see \cite{Jarchow, Schaefer, Treves}. All vector spaces in this paper are real. 

Let $\Phi$ be a locally convex space. We denote by $\Phi'$ the topological dual of $\Phi$ and by $\inner{f}{\phi}$ the canonical pairing of elements $f \in \Phi'$, $\phi \in \Phi$. Unless otherwise specified, $\Phi'$ will always be consider equipped with its \emph{strong topology}, i.e. the topology on $\Phi'$ generated by the family of semi-norms $( \eta_{B} )$, where for each $B \subseteq \Phi$ bounded we have $\eta_{B}(f)=\sup \{ \abs{\inner{f}{\phi}}: \phi \in B \}$ for all $f \in \Phi'$. Recall that $\Phi$ is called \emph{semi-reflexive} if the canonical (algebraic) embedding of $\Phi$ into $\Phi''$ is onto, and is called \emph{reflexive} if the canonical embedding is indeed an isomorphism (of topological vector spaces). 

A locally convex space is called \emph{ultrabornological} if it is the inductive limit of a family of Banach spaces. A \emph{barreled space} is a locally convex space such that every convex, balanced, absorbing and closed subset is a neighborhood of zero. For equivalent definitions see \cite{Jarchow, NariciBeckenstein}.
  
A continuous seminorm (respectively norm) $p$ on $\Phi$ is called \emph{Hilbertian} if $p(\phi)^{2}=Q(\phi,\phi)$, for all $\phi \in \Phi$, where $Q$ is a symmetric, non-negative bilinear form (respectively inner product) on $\Phi \times \Phi$. For any given continuous seminorm $p$ on $\Phi$ let $\Phi_{p}$ be the Banach space that corresponds to the completion of the normed space $(\Phi / \mbox{ker}(p), \tilde{p})$, where $\tilde{p}(\phi+\mbox{ker}(p))=p(\phi)$ for each $\phi \in \Phi$. We denote by  $\Phi'_{p}$ the Banach space dual to $\Phi_{p}$ and by $p'$ the corresponding dual norm. Observe that if $p$ is Hilbertian then $\Phi_{p}$ and $\Phi'_{p}$ are Hilbert spaces. If $q$ is another continuous seminorm on $\Phi$ for which $p \leq q$, we have that $\mbox{ker}(q) \subseteq \mbox{ker}(p)$ and the inclusion map from $\Phi / \mbox{ker}(q)$ into $\Phi / \mbox{ker}(p)$ has a unique continuous and linear extension that we denote by $i_{p,q}:\Phi_{q} \rightarrow \Phi_{p}$. 

Let $\theta$ be a weaker pseudometrizable topology  on $\Phi$. We denote by $\Phi_{\theta}$ the space $(\Phi,\theta)$ and by $\widehat{\Phi}_{\theta}$ its completion. If $\theta$ is generated by an increasing sequence of separable continuous Hilbertian semi-norms $( p_{n} : n \in \N)$ on $\Phi$, we will say that $\Phi$ a (weaker) \emph{countably Hilbertian topology} on $\Phi$.  In this latter case $\widehat{\Phi}_{\theta}$ is separable and  complete (but not necessarily Hausdorff) and its dual space satisfies $(\widehat{\Phi}_{\theta})'=(\Phi_{\theta})'=\bigcup_{n \in \N} \Phi'_{p_{n}}$ (see \cite{FonsecaMora:2018}, Proposition 2.4).   

For a topological vector space $\Psi$ we denote by $(\Psi)_{lcx}$ its convexification, i.e. the space $\Psi$ equipped with the strongest locally convex topology on $\Psi$ that is weaker than the given topology in $\Psi$. A local basis of neighborhoods of zero in $(\Psi)_{lcx}$ can be defined taking the convex envelopes of the members of a local basis of neighborhoods of zero in $\Psi$.

We denote by $\mathcal{L}(\Phi,\Psi)$ the linear space of all the linear and continuous operators between any two locally convex spaces (or more generally topological vector spaces) $\Phi$ and $\Psi$. For information of the topologies on the space $\mathcal{L}(\Phi,\Psi)$ the reader is referred to e.g. Chapter 32 in \cite{Treves}. If $R \in \mathcal{L}(\Phi,\Psi)$ we denote by $R'$ its \emph{dual operator} and recall that $R' \in \mathcal{L}(\Psi',\Phi')$.

%


\subsection{Tensor Products of Topological Vector Spaces}\label{subsecTensorProdTVS}

In this section we quickly review the definition of the projective topology on the tensor product of two topological vector spaces. 

Let $\Psi$ and $\Phi$ be two vector spaces. We denote by $\Psi \otimes \Phi$ the algebraic tensor product defined as the set of elements of the form $\sum_{i=1}^{n} \psi_{i} \otimes \phi_{i}$, for some $n \in \N$ and $\psi_{i} \in \Psi$, $\phi_{i} \in \Phi$ for $i =1, \dots, n$. The canonical product mapping $\otimes: \Psi \times \Phi \rightarrow \Psi \otimes \Phi$ is bilinear. Recall the `universal property' of tensor products (see e.g. Theorem 39.1 in \cite{Treves}): to every bilinear mapping $B$ of $\Psi \times \Phi$ into a vector space $\Upsilon$, there corresponds a unique linear map $\tilde{B}: \Psi \otimes \Phi \rightarrow \Upsilon$, called its \emph{linearization}, such that $B=\tilde{B} \circ \otimes$.  

Assume that $\Psi$ and $\Phi$ are locally convex spaces. The projective topology on $\Psi \otimes \Phi$ can be constructed via seminorms in the following way. Let $p$ (respectively $q$) be a seminorm on $\Psi$ (respectively $\Phi$). For any $\theta \in \Psi \otimes \Phi$, define 
$$(p \otimes q)(\theta)=\inf \sum_{j} p(\psi_{j}) q(\phi_{j}),$$
where the infimum is taken over all finite set of pairs $(\psi_{j}, \phi_{j})$ such that $\theta= \sum_{j} \psi_{j} \otimes \phi_{j}$. Then, one can show (see \cite{Treves}, Proposition 43.1, p.435) that $p \otimes q$ defines a seminorm on $\Psi \otimes \Phi$ (a norm if both $p$ and $q$ are norms). If $(p_{\alpha})$ (respectively $(q_{\beta})$) are a basis of continuous seminorms on $\Psi$ (respectively $\Phi$), then $(p_{\alpha} \otimes q_{\beta})$ is a basis of seminorms generating a locally convex topology on $\Psi \otimes \Phi$ called the \emph{projective topology}. The space $\Psi \otimes \Phi$ equipped with this topology will be denoted by $\Psi \,  \otimes_{\pi} \,  \Phi$ and its completion will be denoted by $\Psi \,  \widehat{\otimes}_{\pi} \,  \Phi$. Observe that if $\Psi$ and $\Phi$ are normed spaces then $\Psi \,  \widehat{\otimes}_{\pi} \,  \Phi$ is a Banach space. 

The projective topology $\pi$ is the strongest locally convex topology on $\Psi \otimes \Phi$ for which the mapping $\otimes: \Psi \times \Phi \rightarrow \Psi \otimes \Phi$ is continuous. Moreover it is the unique vector topology on  $\Psi \otimes \Phi$ having the property that for every locally convex space $\Upsilon$, a bilinear mapping $B$ of $\Psi \times \Phi$ into  $\Upsilon$ is continuous if and only if its linearization $\tilde{B}$ is continuous of $\Psi \otimes \Phi$ into $\Upsilon$  (see \cite{Treves}, Proposition 43.4, p.438). For further properties of the projective topology of two locally convex spaces the reader is referred to \cite{Jarchow, KotheII, Treves}. 

Suppose that $\Psi$ and $\Phi$ are topological vector spaces. We will also require the  definition of the (not-locally convex) projective tensor topology on $\Psi \otimes \Phi$ defined in \cite{Tomasek:1965, Turpin:1982}. Let $\mathcal{U}$ be a system of neighborhoods of zero in $\Psi$ and $\mathcal{V}$ be a system of neighborhoods of zero in $\Phi$. For any sequence $( U_{i}: i \in \N) \subseteq \mathcal{U}$, and any sequence $( V_{i}: i \in \N) \subseteq \mathcal{V}$, define
$$ \Gamma_{(U_{i}),(V_{i})} = \bigcup_{n \geq 1} \sum_{i=1}^{n} U_{i} \otimes V_{i}, $$    
where $U_{i} \otimes V_{i}= \{ \psi \otimes \phi: \psi \in U_{i}, \phi \in V_{i}\}$. One can show (see \cite{Tomasek:1965}) that the collection of sets of the form $\Gamma_{(U_{i}),(V_{i})}$ defines a vector topology $\nu$ on $\Psi \otimes \Phi$, called the \emph{projective topology}. This topology can be defined equivalently in terms of a generating family of pseudo-seminorms (see \cite{Turpin:1982}). The space $\Psi \otimes \Phi$ equipped with this topology will be denoted by $\Psi \,  \otimes_{\nu} \,  \Phi$ and we denote by $\Psi \,  \widehat{\otimes}_{\nu} \,  \Phi$ its completion. The space $\Psi \,  \otimes_{\nu} \,  \Phi$ is Hausdorff (respectively complete) if both $\Psi$ and $\Phi$ are Hausdorff (respectively complete) (see Th\'{e}or\`{e}me 2.4 in \cite{Turpin:1982}). However, it is worth to mention that unlike the case of the projective topology for locally convex spaces, the topology introduced above in general fails to be associative (see \cite{Glockner:2004}). 

The projective topology $\nu$ is the strongest vector topology on $\Psi \otimes \Phi$ for which the mapping $\otimes: \Psi \times \Phi \rightarrow \Psi \otimes \Phi$ is continuous. Moreover it is the unique vector topology on  $\Psi \otimes \Phi$ having the property that for every topological vector space $\Upsilon$, a bilinear mapping $B$ of $\Psi \times \Phi$ into  $\Upsilon$ is continuous if and only if its linearization $\tilde{B}$ is continuous of $\Psi \otimes \Phi$ into $\Upsilon$ (see \cite{Tomasek:1965}, Theorem 2).

\subsection{Cylindrical and Stochastic Processes} \label{subSectionCylAndStocProcess}

Throughout this work we assume that $\ProbSpace$ is a complete probability space equipped with a filtration $( \mathcal{F}_{t} : t \geq 0)$ that satisfies the \emph{usual conditions}, i.e. it is right continuous and $\mathcal{F}_{0}$ contains all subsets of sets of $\mathcal{F}$ of $\Prob$-measure zero. We denote by $L^{0} \ProbSpace$ the space of equivalence classes of real-valued random variables defined on $\ProbSpace$. We always consider the space $L^{0} \ProbSpace$ equipped with the topology of convergence in probability and in this case it is a complete, metrizable, topological vector space.

Let $\Phi$ be a locally convex space. A \emph{cylindrical random variable}\index{cylindrical random variable} in $\Phi'$ is a linear map $X: \Phi \rightarrow L^{0} \ProbSpace$ (see \cite{FonsecaMora:2018}). If $X$ is a cylindrical random variable in $\Phi'$, we say that $X$ is \emph{$n$-integrable} ($n \in \N$) if $ \Exp \left( \abs{X(\phi)}^{n} \right)< \infty$, $\forall \, \phi \in \Phi$, and has \emph{mean-zero} if $ \Exp \left( X(\phi) \right)=0$, $\forall \phi \in \Phi$. The \emph{Fourier transform} of $X$ is the map from $\Phi$ into $\C$ given by $\phi \mapsto \Exp ( e^{i X(\phi)})$.

Let $X$ be a $\Phi'$-valued random variable, i.e. $X:\Omega \rightarrow \Phi'$ is a $\mathscr{F}/\mathcal{B}(\Phi')$-measurable map (recall that $\Phi'$ is equipped with its strong topology). For each $\phi \in \Phi$ we denote by $\inner{X}{\phi}$ the real-valued random variable defined by $\inner{X}{\phi}(\omega) \defeq \inner{X(\omega)}{\phi}$, for all $\omega \in \Omega$. The linear mapping $\phi \mapsto \inner{X}{\phi}$ is called the \emph{cylindrical random variable induced/defined by $X$}. We will say that a $\Phi'$-valued random variable $X$ is \emph{$n$-integrable} if the cylindrical random variable induced by $X$ is \emph{$n$-integrable}. 
 
Let $J=\R_{+} \defeq [0,\infty)$ or $J=[0,T]$ for  $T>0$. We say that $X=( X_{t}: t \in J)$ is a \emph{cylindrical process} in $\Phi'$ if $X_{t}$ is a cylindrical random variable for each $t \in J$. Clearly, any $\Phi'$-valued stochastic processes $X=( X_{t}: t \in J)$ induces/defines a cylindrical process under the prescription: $\inner{X}{\phi}=( \inner{X_{t}}{\phi}: t \in J)$, for each $\phi \in \Phi$. 

If $X$ is a cylindrical random variable in $\Phi'$, a $\Phi'$-valued random variable $Y$ is called a \emph{version} of $X$ if for every $\phi \in \Phi$, $X(\phi)=\inner{Y}{\phi}$ $\Prob$-a.e. A $\Phi'$-valued process $Y=(Y_{t}:t \in J)$ is said to be a $\Phi'$-valued \emph{version} of the cylindrical process $X=(X_{t}: t \in J)$ on $\Phi'$ if for each $t \in J$, $Y_{t}$ is a $\Phi'$-valued version of $X_{t}$.  

For a $\Phi'$-valued process $X=( X_{t}: t \in J)$ terms like continuous, c\`{a}dl\`{a}g, c\`{a}gl\`{a}d, adapted, predictable, etc. have the usual (obvious) meaning. 

A $\Phi'$-valued random variable $X$ is called \emph{regular} if there exists a weaker countably Hilbertian topology $\theta$ on $\Phi$ such that $\Prob( \omega: X(\omega) \in (\widehat{\Phi}_{\theta})')=1$. Furthermore, a $\Phi'$-valued process $Y=(Y_{t}:t \in J)$ is said to be \emph{regular} if $Y_{t}$ is a regular random variable for each $t \in J$. In that case the law of each $Y_{t}$ is a Radon measure in $\Phi'$ (see Theorem 2.10 in \cite{FonsecaMora:2018}).

Recall that a real-valued adapted c\`{a}dl\`{a}g  process $( x_{t}: t \geq 0) $ is a \emph{semimartingale} if it admit a representation of the form 
$$x_{t}=x_{0} + m_{t}+a_{t}, \quad \forall \, t \geq 0, $$
where $(m_{t}:t \geq 0)$ is a c\`{a}dl\`{a}g local martingale and $(a_{t}: t \geq 0)$ is a c\`{a}dl\`{a}g adapted process of finite variation, and $m_{0}=a_{0}=0$. In the following paragraphs we will review only some of the most important results on the theory of semimartingales that we will need through this article. The reader is referred to \cite{DellacherieMeyer, JacodShiryaev, Protter} for further details. 

We denote by $S^{0}$  the linear space (of equivalence classes) of real-valued semimartingales. We denote by $\mathcal{M}_{loc}$ and $\mathscr{V}$ the subspaces of real-valued local martingales and of finite variation process.  

We will always assume $S^{0}$ is equipped with Emery's topology which is defined as follows: recall that the topology of convergence in probability uniformly on compact intervals of time is defined by the F-seminorm:  
$$z \mapsto \norm{z}_{ucp}\defeq \sum_{n=1}^{\infty} 2^{-n} \Exp \left( 1 \wedge \sup_{0 \leq t \leq n} \abs{z_{t}} \right).$$
Emery's topology on $S^{0}$ is the topology defined by the F-seminorm:
$$z \mapsto \norm{z}_{S^{0}}\defeq \sup\{ \norm{ h \cdot z}_{ucp}: h \in \mathcal{E}_{1}  \},$$
where  $\mathcal{E}_{1}$ is the collection of all the  real-valued predictable processes of the form 
$ \displaystyle{h=\sum_{i=1}^{n-1} a_{i} \mathbbm{1}_{(t_{i}, t_{i+1}] \times \Omega}}$,  
for $0 < t_{1} < t_{2} < \dots t_{n} < \infty$, $a_{i}$ is an $\mathcal{F}_{t_{i}}$-measurable random variable, $\abs{a_{i}} \leq 1$, $i=1, \dots, n-1$, and 
$$ (h \cdot z)_{t}=\int_{0}^{t} h_{s} dz_{s}= \sum_{i=1}^{n-1} a_{i} \left( z_{t_{i+1} \wedge t}-z_{t_{i} \wedge t} \right). $$  
The space $S^{0}$ equipped with Emery's topology is a complete, metrizable, topological vector space (however it is not in general locally convex). See \cite{Memin:1980} for an equivalent topology on $S^{0}$ defined by Memin.  

In this work we will make reference to several spaces of particular classes of semimartingales which we detail as follows. 
By $S^{c}$ we denote the subspace of $S^{0}$ of all the continuous semimartingales and by $\mathcal{M}^{c}_{loc}$ the space of continuous local martingales, both are equipped with the topology of uniform convergence in probability on compact intervals of time. Likewise $\mathcal{A}_{loc}$  denotes the space of all predictable processes of finite variation, with locally integrable variation, equipped with the F-seminorm: $\norm{a}_{\mathcal{A}_{loc}}=\Exp \left( 1 \wedge \int_{0}^{\infty} \abs{d a_{s}} \right)$. The spaces $S^{c}$, $\mathcal{M}^{c}_{loc}$ and $\mathcal{A}_{loc}$ are all closed subspaces of $S^{0}$ and the subspace topology on $S^{c}$, $\mathcal{M}^{c}_{loc}$ and $\mathcal{A}_{loc}$ coincides with their given topology (see \cite{Memin:1980}, Th\'{e}or\`{e}me IV.5 and IV.7). 
 
For every real-valued semimartingale $x=(x_{t}:t \geq 0)$ and each $p \in [1,\infty]$, we denote by $\norm{x}_{\mathcal{H}^{p}_{S}}$ the following quantity: 
$$\norm{x}_{\mathcal{H}^{p}_{S}} = \inf \left\{ \norm{ [m,m]_{\infty}^{1/2} + \int_{0}^{\infty} \abs{d a_{s} } }_{L^{p}\ProbSpace} : x=m+a \right\}, $$
where the infimum is taken over all the decompositions $x=m+a$ as a sum of a local martingale $m$ and a process of finite variation $a$. Recall that  $([m,m]_{t}: t \geq 0)$ denotes the quadratic variation process associated to the local martingale $m$, i.e. $[m,m]_{t}=\llangle  m^{c} , m^{c}  \rrangle_{t}+\sum_{0 \leq s \leq t} (\Delta m_{s})^{2}$, where $m^{c}$ is the (unique) continuous local martingale part of $m$ and $(\llangle  m^{c} , m^{c}  \rrangle_{t}:t \geq 0)$ its angle bracket process (see Section I in \cite{JacodShiryaev}). The set of all semimartingales $x$ for which $\norm{x}_{\mathcal{H}^{p}_{S}}< \infty$ is a Banach space under the norm $\norm{\cdot}_{\mathcal{H}^{p}_{S}}$ and is denoted by $\mathcal{H}^{p}_{S}$ (see Section 16.2 in \cite{CohenElliott}). Furthermore, if $x=m+a$ is a decomposition of $x$ such that $\norm{x}_{\mathcal{H}^{p}_{S}}< \infty$ it is known that in such a case $a$ is of integrable variation (see VII.98(c) in \cite{DellacherieMeyer}). 


For $p \geq 1$, denote by $\mathcal{M}_{\infty}^{p}$ the space of real-valued martingales for which $\norm{m}_{\mathcal{M}_{\infty}^{p}}= \norm{ \sup_{t \geq 0} m_{t} }_{L^{p}\ProbSpace}< \infty$. It is well-known that $\mathcal{M}_{\infty}^{p}$ equipped with the norm $\norm{\cdot}_{\mathcal{M}_{\infty}^{p}}$ is a Banach space. Likewise, we denote by $\mathcal{A}$ the space of all predictable processes of finite variation, with integrable variation. It is well-know that $\mathcal{A}$ is a Banach space when equipped with the norm $\norm{a}_{\mathcal{A}}=\Exp \int_{0}^{t} \abs{d a_{s}} < \infty$. 

If $\Psi$ is a Hilbert space, the above definitions of semimartingale, 
and (of spaces) of particular classes of semimartingales can be defined in a completely analogue way for $\Psi$-valued processes. The corresponding spaces of semimartingales will be denoted by $S^{0}(\Psi)$, $\mathcal{H}^{p}_{S}(\Psi)$, $\mathcal{M}_{\infty}^{p}(\Psi)$, $\mathcal{A}$, etc. The reader is referred to \cite{Metivier} for the basic theory of Hilbert space valued semimartingales.

Let $\Phi$ denotes a locally convex space. 
A \emph{cylindrical semimartingale} in $\Phi'$ is a cylindrical process $X=(X_{t}: t \geq 0)$ in $\Phi'$ such that $\forall \phi \in \Phi$, $X(\phi)$ is a real-valued semimartingale. In general, if $\goth{S}$ denotes any space of a particular class of semimartingales (as described above), then by a $\goth{S}$-cylindrical semimartingale in $\Phi'$ we mean a cylindrical process $X=(X_{t}: t \geq 0)$ in $\Phi'$ such that $\forall \phi \in \Phi$, $X(\phi) \in \goth{S}$.

\section{Stochastic Integration in Locally Convex Spaces} \label{sectionConstruStochInteg}

Let $\Phi$ denotes a locally convex space and let $(X_{t}:t \geq 0)$ be a cylindrical semimartingale in $\Phi'$. Our objective in this section is to define the stochastic integral process $\int_{0}^{t} H dX$, $t \geq 0$,  as a real-valued semimartingale, where $H$ belongs to a sufficiently large class of integrands.  

As described in Section \ref{sectIntroduction}, to have a working theory of stochastic integration minimum continuity conditions on $X$ are required. These conditions are the following: 
 
\begin{assumption}\label{assuCylSemimartingale} From now on and unless otherwise specified, we will assume that to every cylindrical semimartingale $X=( X_{t}: t \geq 0)$ in $\Phi'$  there exists a weaker pseudometrizable linear topology $\theta$ on $\Phi$ such that $X$ extends to a continuous linear mapping $X^{\theta}$ from $\widehat{\Phi}_{\theta}$ into $S^{0}$. 
\end{assumption}

\begin{remark} \label{remaAssumpCylSemiWeakInteg}
If $\Phi$ is a Fr\'{e}chet space and $X$ is a cylindrical semimartingale in $\Phi'$ for which each $X_{t}:\Phi \rightarrow L^{0} \ProbSpace$ is continuous, from an application of the closed graph theorem it is easy to verify that $X$ satisfies  Assumption \ref{assuCylSemimartingale} if we take $\theta$ as the original topology.  We will see in Section \ref{sectNuclearSpaceCylSemi} that if $\Phi$ is a nuclear space then Assumption \ref{assuCylSemimartingale} is equivalent to several conditions which are easier to check.  
\end{remark}

\subsection{Construction of the Stochastic Integral}\label{subSectionConstStocIntegral}

For our construction of the stochastic integral we  will assume the reader is familiar with the theory of stochastic integration with respect to real-valued semimartingales. However, for further developments the following key properties will be needed.

Denote by $b\mathcal{P}$ the linear space of all the bounded predictable processes $h : \R_{+} \times \Omega \rightarrow \R$. It is a Banach space when equipped with the uniform norm $\norm{h}_{u}=\sup_{(r,\omega)} \abs{h(r,\omega)}$. If $h \in b\mathcal{P}$ and $z \in S^{0}$, then $h$ is stochastically integrable with respect to $z$, and its stochastic integral, that we denote by $h \cdot z=(( h \cdot z)_{t}: t \geq 0)$, is an element of $S^{0}$ (see \cite{Protter}, Theorem IV.15). The mapping $(z,h) \mapsto h \cdot z$ from $S^{0} \times b\mathcal{P}$ into $S^{0}$ is bilinear (see \cite{Protter}, Theorem IV.16-7) and separately continuous (see Theorems 12.4.10-13 in \cite{CohenElliott}).

Our construction of the stochastic integral will be given in Theorem \ref{theoExisWeakStochIntegral} below. We won't follow the `standard' procedure of defining the integral first for a class of simple integrands and then to extend it to general integrands by a continuity argument. In fact, as described in Section \ref{sectIntroduction} our idea is to use the continuity properties of the stochastic integral with respect to real-valued semimartingales and the powerful machinery of tensor products on topological vector spaces. Our arguments will produce a continuous linear operator $I$ from the projective tensor product space $\Phi \, \widehat{\otimes}_{\nu} \, b\mathcal{P}$ (our space of integrands) into the space of real-valued semimartingales $S^{0}$.

\begin{theorem}\label{theoExisWeakStochIntegral}
There exists a unique continuous linear map $I: \Phi \, \widehat{\otimes}_{\nu} \, b\mathcal{P} \rightarrow S^{0}$ such that for each $(\phi,h) \in \Phi \times b\mathcal{P}$ we have $I(\phi \otimes h)=h \cdot X(\phi)$.
\end{theorem}
\begin{proof}
First, define the map $J: \Phi \times b\mathcal{P}  \rightarrow S^{0}$ by 
$$ J(\phi,h)=h \cdot X(\phi), \quad \forall \, \phi \in \Phi, \, h \in  b\mathcal{P} .$$ 
The properties of the stochastic integral with respect to real-valued semimartingales show that $J$ is bilinear and that $h \mapsto J^{w}(\phi,h)$ is continuous from $b\mathcal{P}$ into $S^{0}$ for any given $\phi \in \Phi$. 

Let $\theta$ and $X^{\theta}$  as in Assumption \ref{assuCylSemimartingale}. From our assumption $X$ extends to a continuous linear mapping $X^{\theta}$ from $\widehat{\Phi}_{\theta}$ into $S^{0}$, but since the canonical inclusion from $\Phi$ into $\widehat{\Phi}_{\theta}$ is continuous, we therefore have that $X$ is continuous from $\Phi$ into $S^{0}$. The properties of the stochastic integral with respect to real-valued semimartingales show that $\phi \mapsto J(\phi,h)$ is continuous from $\Phi$ into $S^{0}$ for any given $h \in  b\mathcal{P}$. Hence, $J$ is separately continuous. 

Now, by the properties of $X^{\theta}$ and since $S^{0}$ is complete, it is clear that $J$ has a bilinear and separately continuous extension $J^{\theta}: \widehat{\Phi}_{\theta} \times b\mathcal{P} \rightarrow S^{0}$,  defined by $J^{\theta} (\phi,h)=h \cdot X^{\theta}(\phi)$ for all $(\phi,h )  \in \widehat{\Phi}_{\theta} \times b\mathcal{P}$. Moreover, since $\widehat{\Phi}_{\theta}$, $b\mathcal{P}$ and $S^{0}$ are all complete pseudometrizable topological vector spaces, the mapping $J^{\theta}$ is indeed continuous (see \cite{Swartz:1984}, Corollary 8); here the space $ \widehat{\Phi}_{\theta} \times b\mathcal{P} $ being equipped with the product topology.

Let $\tilde{J}^{\theta}: \widehat{\Phi}_{\theta} \,  \otimes \, b\mathcal{P} \rightarrow S^{0}$ be the linearization of the bilinear form $J^{\theta}: \widehat{\Phi}_{\theta} \times b\mathcal{P} \rightarrow S^{0}$, that is, we have 
$\tilde{J}^{\theta}(\phi \otimes h)=J^{\theta}(\phi,h)$ for each $(\phi,h )  \in \widehat{\Phi}_{\theta} \times b\mathcal{P}$. The continuity of $J^{\theta}$ implies that $\tilde{J}^{\theta}$ is continuous from $\widehat{\Phi}_{\theta} \,  \otimes_{\nu} \, b\mathcal{P}$ into $S^{0}$. 

Now, as the canonical inclusion from $\Phi$ into $\widehat{\Phi}_{\theta}$ is continuous, it follows that the canonical inclusion $j$ from $\Phi \,  \otimes_{\nu} \, b\mathcal{P}$ into $\widehat{\Phi}_{\theta} \,  \otimes_{\nu} \, b\mathcal{P}$ is also continuous (see \cite{Tomasek:1965}, Proposition 2). Therefore the map $\widetilde{J}: \Phi \,  \otimes_{\nu} \, b\mathcal{P} \rightarrow S^{0}$ defined by the composition $\widetilde{J}= \widetilde{J}^{\theta} \circ j$, is linear and continuous. But then there exists a continuous linear map $I: \Phi \, \widehat{\otimes}_{\nu} \, b\mathcal{P} \rightarrow S^{0}$ such that $\widetilde{J}= I \circ \iota$, where $\iota$ denotes the canonical inclusion from $\Phi \, \otimes_{\nu} \, b\mathcal{P}$ into its completion $\Phi \, \widehat{\otimes}_{\nu} \, b\mathcal{P}$. Observe that from its construction, for all $\phi \in \Phi$ and $h \in b\mathcal{P}$ we have 
$$ I (\phi \otimes h)= \widetilde{J}(\phi \otimes h)= \widetilde{J}^{\theta} \circ j (\phi \otimes h) = J^{\theta}(\phi,h)=J(\phi,h)= h \cdot X(\phi).$$ 
The above property determines $I$ uniquely due to the isomorphism of the continuous bilinear mappings of $\Phi \times b\mathcal{P}$ into $S^{0}$ onto the space $\mathcal{L}(\Phi \, \widehat{\otimes}_{\nu} \, b\mathcal{P}, S^{0})$.  
\end{proof}
 
\begin{definition}
We will call the map $I: \Phi \, \widehat{\otimes}_{\nu} \, b\mathcal{P} \rightarrow S^{0}$ defined in Theorem \ref{theoExisWeakStochIntegral} the \emph{stochastic integral map} determined by $X$, and for each $H \in \Phi \, \widehat{\otimes}_{\nu} \, b\mathcal{P}$, we call $I(H)$ the \emph{stochastic integral of $H$ with respect to $X$}. We will utilize the notation $\int H \, dX$ for $I(H)$, and for each $t \geq 0$ we denote $\left( \int H \, dX \right)_{t}$ by $\int_{0}^{t} H \, dX$.  We will call $\Phi \, \widehat{\otimes}_{\nu} \, b\mathcal{P}$ the \emph{space of integrands} and each $H \in \Phi \, \otimes \, b\mathcal{P}$ will be referred as an \emph{elementary integrand}.    
\end{definition}

\begin{remark} 
At this point it is useful to stress why we have chosen the space $ b\mathcal{P}$ to define our class of integrands $\Phi \, \widehat{\otimes}_{\nu} \, b\mathcal{P}$. There are three main reasons for this choice. The first is because the stochastic integral with respect to real-valued semimartingales possesses nice continuity properties for integrands in the space $ b\mathcal{P}$. The importance of the above has been made clear in the proof of Theorem \ref{theoExisWeakStochIntegral}. Second, the elements in the class  $ b\mathcal{P}$ are integrable with respect to every real-valued semimartingale, in particular, if  $ h \in b\mathcal{P}$ then $h \cdot X(\phi) \in S^{0}$ for each $\phi \in \Phi$. The above fact is of extreme importance since in our construction of the stochastic integral mapping $I$ we must be able to integrate with respect to every possible realization $X(\phi)$ of $X$. Third, the space $ b\mathcal{P}$ is Banach and hence we will be able to consider also the locally convex tensor product of $\Phi$ and $ b\mathcal{P}$.   
\end{remark}

We now explore other properties of our stochastic integral. If $H \in \Phi \, \otimes \, b\mathcal{P}$ and if $X$ is a cylindrical local martingale (respectively a cylindrical finite variation process), then it is clear that  $\int H \, dX $ is a local martingale (respectively a finite variation process). However, it is not clear that $\int H \, dX $ is a local martingale (respectively a finite variation process) if $H \in \Phi \, \widehat{\otimes}_{\nu} \, b\mathcal{P}$. The main reason for the occurrence of this phenomenon is the fact that the spaces $\mathcal{M}_{loc}$ and $\mathscr{V}$ are not closed subspaces of $S^{0}$ (see \cite{Emery:1979}). 

A different situation occurs with the spaces $S^{c}$, $\mathcal{M}^{c}_{loc}$, $\mathcal{A}_{loc}$ which are all closed subspaces of $S^{0}$. In this case we have the following:

\begin{proposition} \label{propWeakInteContinuSemimar} 
Assume $H \in \Phi \, \widehat{\otimes}_{\nu} \, b\mathcal{P}$. 
\begin{enumerate}
\item If $X$ is a $S^{c}$-cylindrical semimartingale in $\Phi'$, then $\int \, H \, dX \in S^{c}$. 
\item If $X$ is a $\mathcal{M}^{c}_{loc}$-cylindrical local martingale in $\Phi'$, then $\int \, H \, dX \in \mathcal{M}^{c}_{loc}$.
\item If $X$ is a $\mathcal{A}_{loc}$-cylindrical semimartingale in $\Phi'$, then $\int \, H \, dX \in \mathcal{A}_{loc}$. 
\end{enumerate}
\end{proposition}
\begin{proof} To prove (1), assume $X$ is a $S^{c}$-cylindrical semimartingale in $\Phi'$. Then for each $(\phi,h) \in \Phi \times b\mathcal{P}$ we have $\int \, \phi \otimes h \, dX=h \cdot X(\phi)$ belongs to $S^{c}$. By linearity of the integral mapping, the above property extends to every $H \in \Phi \otimes b\mathcal{P}$. But then since $S^{c}$  is a closed subspace in $S^{0}$, the integral mapping $H \mapsto \int \, H \, dX$ is continuous from $\Phi \, \widehat{\otimes}_{\nu} \, b\mathcal{P}$ into $S^{0}$, $\Phi \otimes b\mathcal{P}$ is dense in $\Phi \, \widehat{\otimes}_{\nu} \, b\mathcal{P}$, and because the induced topology from $S^{0}$ coincides with the topology in $S^{c}$, then $\int \, H \, dX$ belongs to $S^{c}$ for each $H \in \Phi \, \widehat{\otimes}_{\nu} \, b\mathcal{P}$. The proof of (2) (respectively of (3)) follows from similar arguments using that  $\int \, \phi \otimes h \, dX$ belongs to $\mathcal{M}^{c}_{loc}$ (respectively to $\mathcal{A}_{loc}$) whenever $(\phi,h) \in \Phi \times b\mathcal{P}$, and that $\mathcal{M}^{c}_{loc}$ (respectively $\mathcal{A}_{loc}$) is a closed subspace of $S^{0}$.  
\end{proof}

We know from Theorem \ref{theoExisWeakStochIntegral} that by construction the stochastic integral mapping is linear and continuous on the integrands. The following result shows that the stochastic integral is linear with respect to the integrators. 

\begin{proposition} \label{propLineaWeakInteInIntegrators} Suppose that $X$ and $Y$ are two cylindrical semimartingales in $\Phi'$. Then, for each $H \in \Phi \, \widehat{\otimes}_{\nu} \, b\mathcal{P}$, the processes $\int \, H d (X+Y)$ and $\int \, H d X +\int \, H d Y$ are indistinguishable. 
\end{proposition}
\begin{proof} First, it is clear that the cylindrical semimartingale $X+Y=(X_{t}+Y_{t}: t \geq 0)$ also satisfies the conditions in Assumption \ref{assuCylSemimartingale}. Therefore, the mappings $H \mapsto \int \, H d X$, $H \mapsto \int \, H d Y$,  and $H \mapsto \int \, H d (X+Y)$ from  $\Phi \, \widehat{\otimes}_{\nu} \, b\mathcal{P}$ into $S^{0}$ satisfy the conclusions of Theorem \ref{theoExisWeakStochIntegral}. In particular, for each $H = \phi \otimes h \in \Phi \otimes b\mathcal{P}$ we have 
$$\int \, H d (X+Y) = h \cdot (X+Y)(\phi)= h \cdot X(\phi)+h \cdot Y(\phi)=  \int \, H d X +\int \, H d Y. $$
Then, by uniqueness of the stochastic integral mapping we have that for each $H \in \Phi \, \widehat{\otimes}_{\nu} \, b\mathcal{P}$,  
$$\int \, H d (X+Y)=\int \, H d X +\int \, H d Y,$$
with equality in $S^{0}$, hence under indistinguishability. 
\end{proof}

Other properties of the stochastic integral are provided in the following theorem. In particular, we study what happens when we take the continuous part and the stopped process of our weak stochastic integral. In order to formulate our result, we will need first to extend the concepts of continuous part and stopped process to cylindrical semimartingales. 

Given the cylindrical semimartingale $X$, observe that for each $\phi \in \Phi$ the real-valued semimartingale $X(\phi)$ has a uniquely determined continuous local martingale part $X(\phi)^c$ (see VIII.45 in \cite{DellacherieMeyer}). We define a cylindrical continuous local martingale $X^{c}$ by the prescription $X^{c}(\phi) \defeq X(\phi)^{c}$ for each $\phi \in \Phi$. The fact that the continuous local martingale part is unique easily shows that the map $X^{c}$ is linear. Similarly, if $\tau$ is an stopping time, we define a cylindrical semimartingale $X^{\tau}$ by $X^{\tau}(\phi) \defeq X(\phi)^\tau$ for each $\phi \in \Phi$. The map $X^{\tau}$ is clearly linear. Therefore, $X^{c}$ and $X^{\tau}$ are cylindrical semimartingales in $\Phi'$. Since the operations $z \mapsto z^{c}$ and $z \mapsto z^{\tau}$ are continuous from $S^{0}$ into $S^{0}$ (see \cite{Emery:1979}), then the cylindrical semimartingales $X^{c}$ and $X^{\tau}$ both satisfy Assumption \ref{assuCylSemimartingale} and hence the stochastic integral is defined with respect to each of them. Moreover, we have:
 
\begin{theorem}\label{theoContPartStoppedWeakInt}  Let $H \in \Phi \, \widehat{\otimes}_{\nu} \, b\mathcal{P}$. Then, 
\begin{enumerate}
\item $\displaystyle{\left(\int H \, dX \right)^{c}= \int H \, dX^{c}}$. 
\item $\displaystyle{\left(\int H \, dX \right)^{\tau}= \int H \, dX^{\tau}}$, for every stopping time $\tau$. 
\end{enumerate} 
\end{theorem}
\begin{proof}
For every $\phi \in \Phi$ and $h \in b\mathcal{P}$, well-known properties of the stochastic integral with respect to real-valued semimartingales (see e.g. Theorem 12.3.22 in \cite{CohenElliott}) show that 
\begin{enumerate}[label=(\alph*)]
\item $\left( h \cdot X(\phi) \right)^{c} = h \cdot X(\phi)^{c}$. 
\item $( h \cdot X(\phi))^{\tau}=  h \cdot \left( X(\phi)^{\tau} \right)$, for every stopping time $\tau$.   
\end{enumerate}
To prove $(1)$, denote by $A$ the mapping $A(H)=\left(\int H \, dX \right)^{c}$ for each $H \in \Phi \, \widehat{\otimes}_{\nu} \, b\mathcal{P}$. Then by Theorem \ref{theoExisWeakStochIntegral} we have $A \in \mathcal{L}(\Phi \, \widehat{\otimes}_{\nu} \, b\mathcal{P}, S^{0})$. 

On the other hand, since $X^{c}$ is a cylindrical semimartingale satisfying the conditions in Assumption \ref{assuCylSemimartingale}, it follows from Theorem \ref{theoExisWeakStochIntegral} that the mapping $B$ defined by $B(H)=\int H \, dX^{c}$ satisfies $B \in \mathcal{L}(\Phi \, \widehat{\otimes}_{\nu} \, b\mathcal{P}, S^{0})$. By the property (a) given above and the action of the stochastic integral on elementary integrands (Theorem \ref{theoExisWeakStochIntegral}), we have
$$A(\phi \otimes h)=\left( h \cdot X(\phi) \right)^{c} = h \cdot X(\phi)^{c}=B(\phi \otimes h),$$
for every $\phi \in \Phi$ and $h \in b\mathcal{P}$. By linearity, we have that $ A(H) = B(H)$ for every $H \in \Phi \, \otimes \, b\mathcal{P}$. This property can be extended to every $H \in \Phi \, \widehat{\otimes}_{\nu} \, b\mathcal{P}$ by continuity of the mappings $A$ and $B$. This shows \emph{(1)}. The proof of \emph{(2)} can be carried out using similar arguments and the property (b) listed above. 
\end{proof}

\subsection{Locally Convex Integrands}\label{subSecLocalConvInteg}

We learnt from  Section \ref{subSectionConstStocIntegral} that the topological vector tensor product $\Phi \, \widehat{\otimes}_{\nu} \, b\mathcal{P}$ was exactly what we needed for our general construction of the stochastic integral. However, in general there is not a nice description of the elements in $\Phi \, \widehat{\otimes}_{\nu} \, b\mathcal{P}$. Indeed, even if we assume that $\Phi$ is metrizable, unlike the case of the tensor products of locally convex spaces (see \cite{Schaefer}, Theorem III.6.4, p.94-5), the elements of the space $\Phi \, \widehat{\otimes}_{\nu} \, b\mathcal{P}$ in general can not be expressed in terms of an absolutely convergent series. Indeed, the space $\Phi \, \widehat{\otimes}_{\nu} \, b\mathcal{P}$ might not even be metrizable (for a counterexample see \cite{Hollstein:1973}). In particular we do not know of a characterization for the elements in $\Phi \, \widehat{\otimes}_{\nu} \, b\mathcal{P}$ as random processes taking values in $\Phi$. 

In the next result we will try to amend this situation by redefining our stochastic integral for integrands in the locally convex projective tensor product $\Phi \, \widehat{\otimes}_{\pi} \, b\mathcal{P}$. 
However, the reader should be aware that because the space $S^{0}$ is not locally convex, we cannot replicate the arguments used in the proof of Theorem \ref{theoExisWeakStochIntegral} to show the existence of a linear mapping from $\Phi \, \widehat{\otimes}_{\pi} \, b\mathcal{P}$  into $S^{0}$. Nevertheless, we will see in the next result that by using the convexified topology of $S^{0}$ the stochastic integral of Theorem \ref{theoExisWeakStochIntegral} can be used to define such a linear mapping. 

\begin{theorem} \label{theoContiWeakIntegConvexifSemiSpace}
The stochastic integral mapping $I$ of Theorem \ref{theoExisWeakStochIntegral} defines a unique continuous linear mapping from $\Phi \, \widehat{\otimes}_{\pi} \, b\mathcal{P}$ into $(S^{0})_{lcx}$ such that for each $(\phi,h) \in \Phi \times b\mathcal{P}$ we have $I(\phi \otimes h)=h \cdot X(\phi)$. Moreover this mapping preserves the properties of the stochastic integral for integrands in $\Phi \, \widehat{\otimes}_{\pi} \, b\mathcal{P}$ given in Proposition \ref{propLineaWeakInteInIntegrators} and Theorem \ref{theoContPartStoppedWeakInt} by replacing continuity into $S^{0}$ by continuity into $(S^{0})_{lcx}$. 
\end{theorem}
\begin{proof} 
First, observe that because the canonical inclusion $S^{0} \rightarrow (S^{0})_{lcx}$ is linear and continuous, it follows from Theorem \ref{theoExisWeakStochIntegral} that the stochastic integral mapping $I$ induces a continuous linear operator from $\Phi \otimes_{\nu} b\mathcal{P}$ into $(S^{0})_{lcx}$. Then, since $(\Phi \otimes_{\nu} b\mathcal{P})_{lcx} \simeq \Phi \otimes_{\pi} b\mathcal{P}$ (see Lemma 6.3 in \cite{Glockner:2012}) and $(S^{0})_{lcx}$ is locally convex, it follows that (see \cite{Shuchat:1972}, Proposition 6)
$$\mathcal{L}(\Phi \otimes_{\nu} b\mathcal{P},(S^{0})_{lcx}) =\mathcal{L}((\Phi \otimes_{\nu} b\mathcal{P})_{lcx},(S^{0})_{lcx})= \mathcal{L}(\Phi \otimes_{\pi} b\mathcal{P},(S^{0})_{lcx}). $$
Hence, the map $I$ extends to a continuous linear mapping from $\Phi \, \widehat{\otimes}_{\pi} \, b\mathcal{P}$ into $(S^{0})_{lcx}$. This map preserves the property that for each $(\phi,h) \in \Phi \times b\mathcal{P}$ we have $I(\phi \otimes h)=h \cdot X(\phi)$ and hence the arguments used in the proofs of Proposition \ref{propLineaWeakInteInIntegrators} and Theorem \ref{theoContPartStoppedWeakInt} remain valid for integrands in $\Phi \, \widehat{\otimes}_{\pi} \, b\mathcal{P}$ by replacing continuity into $S^{0}$ by continuity into $(S^{0})_{lcx}$.
\end{proof}

The result in Theorem \ref{theoContiWeakIntegConvexifSemiSpace} deserves some further clarifications. First, since the topology in $S^{0}$ is not locally convex, the spaces $S^{0}$ and $(S^{0})_{lcx}$ are not homeomorphic. Therefore, the integral mapping $I^{w}: \Phi \, \widehat{\otimes}_{\pi} \, b\mathcal{P} \rightarrow (S^{0})_{lcx}$ defined in Theorem \ref{theoContiWeakIntegConvexifSemiSpace} does not necessarily implies the existence of a continuous linear operator from $\Phi \, \widehat{\otimes}_{\pi} \, b\mathcal{P}$ into $S^{0}$. Indeed, we do not know if such a continuity is satisfied for every cylindrical semimartingale $X$. Even though, the conclusion of Theorem  \ref{theoContiWeakIntegConvexifSemiSpace} is still useful since it shows that to every $H \in \Phi \, \widehat{\otimes}_{\pi} \, b\mathcal{P}$ we can associate a real-valued semimartingale $I(H)$ which satisfies most of the basic properties one might expect for a stochastic integral. 

Our next objective in this section is to study some classes of cylindrical semimartingales for which a construction of the  stochastic integral can be carried out in order to obtain a continuous and linear operator from $\Phi \, \widehat{\otimes}_{\pi} \, b\mathcal{P}$ into $S^{0}$. These classes of semimartingales share in common that their range is in a subspace of a space of semimartingales with a Banach space structure. 

Let $p \in [1, \infty]$. Let $X=(X_{t}: t \geq 0)$ be a $\mathcal{H}^{p}_{S}$-cylindrical semimartingale. A direct calculation (e.g. as in Chapter IV in \cite{Protter}) shows that for every $h \in b\mathcal{P}$ and $\phi \in \Phi$, $h \cdot X(\phi) \in \mathcal{H}^{p}_{S}$ and we have 
\begin{equation}\label{eqInequaIntegSpSemimartingales}
\norm{h \cdot X(\phi)}_{\mathcal{H}^{p}_{S}} \leq \norm{h}_{u} \norm{ X(\phi)}_{\mathcal{H}^{p}_{S}}.
\end{equation}
As the next result shows, for this class of cylindrical semimartingales, we can make a construction of  the stochastic integral map with range in $\mathcal{H}^{p}_{S}$. 
 
\begin{proposition} \label{propWeakInteForSpSemimar}
If $X$ is a $\mathcal{H}^{p}_{S}$-cylindrical semimartingale, then there exists a unique continuous linear map $I: \Phi \, \widehat{\otimes}_{\pi} \, b\mathcal{P} \rightarrow \mathcal{H}^{p}_{S}$ such that for each $(\phi,h) \in \Phi \times b\mathcal{P}$ we have $I(\phi \otimes h)=h \cdot X(\phi)$. Moreover this mapping preserves the properties of the stochastic integral for integrands in $\Phi \, \widehat{\otimes}_{\pi} \, b\mathcal{P}$ given in Proposition \ref{propLineaWeakInteInIntegrators} and Theorem \ref{theoContPartStoppedWeakInt}. 
\end{proposition}
\begin{proof}
First, if $\theta$ and $X^{\theta}$ are as in Assumption \ref{assuCylSemimartingale}, the continuity of the map $X^{\theta}:\widehat{\Phi}_{\theta} \rightarrow S^{0}$, the fact that the $\mathcal{H}^{p}_{S}$-topology is finer than that $S^{0}$-topology, and by following standard arguments one can show that $X^{\theta}$ is sequentially closed; therefore continuous from $\widehat{\Phi}_{\theta}$ into $\mathcal{H}^{p}_{S}$ by the closed graph theorem (see \cite{NariciBeckenstein}, Theorem 14.3.4, p.465). Moreover, by \eqref{eqInequaIntegSpSemimartingales} it is clear that the bilinear form $J: \Phi \times b\mathcal{P}  \rightarrow \mathcal{H}^{p}_{S}$  defined by $ J(\phi,h)=h \cdot X(\phi)$, $\forall \, \phi \in \Phi$, $h \in  b\mathcal{P}$, is separately continuous. Then, it has an extension $J^{\theta}: \widehat{\Phi}_{\theta} \times b\mathcal{P}  \rightarrow \mathcal{H}^{p}_{S}$,  defined by $J^{\theta} (\phi,h)=h \cdot X^{\theta}(\phi)$ for all $(\phi,h )  \in \widehat{\Phi}_{\theta} \times b\mathcal{P}$, that is bilinear and separately continuous, hence continuous. 
Then, using the properties of the locally convex projective topology we can follow line-by-line the proof of Theorem \ref{theoExisWeakStochIntegral} by replacing $S^{0}$ with $\mathcal{H}^{p}_{S}$ to get the existence of the mapping $I$ such that for each $(\phi,h) \in \Phi \times b\mathcal{P}$ we have $I(\phi \otimes h)=h \cdot X(\phi)$. Hence the arguments used in the proofs of Proposition \ref{propLineaWeakInteInIntegrators} and Theorem \ref{theoContPartStoppedWeakInt} (with $S^{0}$ replaced by $\mathcal{H}^{p}_{S}$) remain valid and so their conclusions as well. 
\end{proof}

We can use the result of Proposition \ref{propWeakInteForSpSemimar} to prove a series representation for the stochastic integral for integrands in $H \in \Phi \, \widehat{\otimes}_{\pi} \, b\mathcal{P}$. 
Suppose that $\Phi$ is a metrizable locally convex space. Since $b\mathcal{P}$ is Banach, each element $H \in \Phi \, \widehat{\otimes}_{\pi} \, b\mathcal{P}$ can be represented as the sum of an absolute convergent series (see \cite{Schaefer}, Theorem III.6.4, p.94-5), 
\begin{equation} \label{decompElemeProjTensProduc}
H = \sum_{i=1}^{\infty} \lambda_{i} (\phi_{i} \otimes h_{i}), 
\end{equation}
where $\sum \abs{\lambda_{i}}< \infty$, $\phi_{i} \rightarrow 0$ in $\Phi$, and $h_{i} \rightarrow 0$ in $b\mathcal{P}$. 

\begin{corollary} \label{coroSeriesRepreWeakIntegral}
Let $\Phi$ be a metrizable locally convex space and let $H \in \Phi \, \widehat{\otimes}_{\pi} \, b\mathcal{P}$ with the representation \eqref{decompElemeProjTensProduc}. If $X$ is a $\mathcal{H}^{p}_{S}$-cylindrical semimartingale  then  $I(H)$ satisfies the following representation with convergence in $\mathcal{H}^{p}_{S}$: 
\begin{equation}\label{eqSeriesReprWeakInteg}
I(H)=\sum_{i=1}^{\infty} \lambda_{i} ( h_{i} \cdot X(\phi_{i})).
\end{equation}
\end{corollary}
\begin{proof} Let $X$ be a $\mathcal{H}^{p}_{S}$-cylindrical semimartingale. 
The linearity and continuity of the stochastic integral mapping $I: \Phi \, \widehat{\otimes}_{\pi} \, b\mathcal{P} \rightarrow \mathcal{H}^{p}_{S}$ and its action on elementary integrands (Proposition \ref{propWeakInteForSpSemimar}) implies that 
$$ I(H)= \sum_{i=1}^{\infty} \lambda_{i} \int \phi_{i} \otimes h_{i} \, dX = \sum_{i=1}^{\infty} \lambda_{i} ( h_{i} \cdot X(\phi_{i})), $$
where the convergence occurs in the space $\mathcal{H}^{p}_{S}$.  
\end{proof}

We can extend the result in Proposition \ref{propWeakInteForSpSemimar} to other classes of cylindrical semimartingales.   

\begin{proposition} \label{propWeakInteForSquaMarting}
If $X$ is a $\mathcal{M}_{\infty}^{2}$-cylindrical martingale (respectively a $\mathcal{A}$-cylindrical semimartingale), then there exists a unique continuous linear map $I$ from $\Phi \, \widehat{\otimes}_{\pi} \, b\mathcal{P}$ into $\mathcal{M}_{\infty}^{2}$ (respectively into $\mathcal{A}$) such that for each $(\phi,h) \in \Phi \times b\mathcal{P}$ we have $I(\phi \otimes h)=h \cdot X(\phi)$. Moreover this mapping preserves the properties of the integral given in Proposition \ref{propLineaWeakInteInIntegrators}, Theorem \ref{theoContPartStoppedWeakInt} and Corollary \ref{coroSeriesRepreWeakIntegral} with convergence in $\mathcal{M}_{\infty}^{2}$ (respectively in $\mathcal{A}$) for the series \eqref{eqSeriesReprWeakInteg}. 
\end{proposition}
\begin{proof}
Let $X=(X_{t}: t \geq 0)$ be a $\mathcal{M}_{\infty}^{2}$-cylindrical martingale. Observe that the arguments used in the proof of Proposition  \ref{propWeakInteForSpSemimar} remain valid if we replace $\mathcal{H}^{p}_{S}$ with $\mathcal{M}_{\infty}^{2}$ as long as we can prove that the bilinear form $J: \Phi \times b\mathcal{P}  \rightarrow \mathcal{M}_{\infty}^{2}$  defined by $ J(\phi,h)=h \cdot X(\phi)$, $\forall \, \phi \in \Phi$, $h \in  b\mathcal{P}$, is separately continuous.

In effect, for every $\phi \in \Phi$ and $h \in b\mathcal{P}$ that $h \cdot X(\phi) \in \mathcal{M}_{\infty}^{2}$ and by the It\^{o} isometry (see e.g. \cite{DellacherieMeyer}, Th\'{e}or\`{e}me VIII.5, p.331) and Burkholder's inequality (e.g. see \cite{Protter}, Theorem IV.74, p.226) there exists $c_{1}, c_{2}>0$ such that,
$$ \norm{ h \cdot X(\phi)}_{\mathcal{M}_{\infty}^{2}}^{2} \leq c_{1} \Exp \int^{\infty}_{0} \abs{h_{t}}^{2} d [X(\phi),X(\phi)]_{t} \leq  c_{2} \norm{h}_{u}^{2} \norm{ X(\phi)}_{\mathcal{M}_{\infty}^{2}}^{2},$$ 
from which we obtain the desired conclusion. Hence, this proves that Proposition \ref{propWeakInteForSquaMarting} holds for the case of a $\mathcal{M}_{\infty}^{2}$-cylindrical martingale. 

Likewise, suppose that $X=(X_{t}: t \geq 0)$ is a $\mathcal{A}$-cylindrical semimartingale. For every $h \in b\mathcal{P}$ and $\phi \in \Phi$, we have $h \cdot X(\phi) \in \mathcal{A}$ and a direct calculation shows that 
$$ \norm{h \cdot X(\phi)}_{\mathcal{A}}= \Exp \int_{0}^{\infty} \abs{d(h \cdot X(\phi))_{s}} \leq \norm{h}_{u} \norm{ X(\phi)}_{\mathcal{A}}. $$
The above inequality shows that the bilinear form $J: \Phi \times b\mathcal{P}  \rightarrow \mathcal{A}$, $ J(\phi,h)=h \cdot X(\phi)$, $\forall \, \phi \in \Phi$, $h \in  b\mathcal{P}$, is separately continuous. Then as explained before by following the arguments used in the proof of Proposition \ref{propWeakInteForSpSemimar} we can show  that the conclusion of Proposition \ref{propWeakInteForSquaMarting} holds for the case of a $\mathcal{A}$-cylindrical semimartingale.
\end{proof}

\begin{remark}\label{remaLocaConvWeakIntegSpSemima}
If $X$ is either a $\mathcal{H}^{p}_{S}$-cylindrical semimartingale, a $\mathcal{M}_{\infty}^{2}$-cylindrical martingale, or a $\mathcal{A}$-cylindrical semimartingale, Propositions   \ref{propWeakInteForSpSemimar} and \ref{propWeakInteForSquaMarting} show that for these classes of cylindrical semimartingales the corresponding integral mapping $I$ defines a continuous linear operator from $\Phi \, \widehat{\otimes}_{\pi} \, b\mathcal{P}$ into $S^{0}$. Moreover, this operator extends to a linear continuous operator from 
$\Phi \, \widehat{\otimes}_{\pi} \, b\mathcal{P}$ into $(S^{0})_{lcx}$. Since the action of this operator on the elementary integrands is the same to that of the operator of Theorem \ref{theoContiWeakIntegConvexifSemiSpace}, then both integral mappings coincide for such an $X$. 
\end{remark}

We finalize this section by showing that under the additional assumption that $\Phi$ is metrizable it is possible to characterize our integrands in $\Phi \, \widehat{\otimes}_{\pi} \, b\mathcal{P}$ as families of $\Phi$-valued weakly predictable bounded processes. 
 
\begin{proposition}\label{propInducedBoundPredProcessMetricCase} Let $\Phi$ be a metrizable locally convex space and suppose $H \in \Phi \, \widehat{\otimes}_{\pi} \, b\mathcal{P}$ has the representation \eqref{decompElemeProjTensProduc}. Then $\tilde{H}: \R_{+} \times \Omega \rightarrow \Phi$ defined for each $(t,\omega) \in \R_{+} \times \Omega$ by 
\begin{equation} \label{eqInducedBoundPredProcessMetriCase}
\tilde{H}(t,\omega)= \sum_{i=1}^{\infty} \lambda_{i}  h_{i}(t,\omega) \, \phi_{i},
\end{equation}
is a $\Phi$-valued bounded (i.e. its image is a bounded subset in $\Phi$) weakly predictable process (i.e. $\forall f \in \Phi'$, the map $(t,\omega) \mapsto \inner{f}{\tilde{H}(t,\omega)}$ is predictable). 
\end{proposition}
\begin{proof} Let $(p_{n}: n \in \N)$ be sequence of seminorms generating the topology on $\Phi$. For each $n \in \N$, since $\phi_{i} \rightarrow 0$ in $\Phi$, and $h_{i} \rightarrow 0$ in $b\mathcal{P}$, it follows that the sets $\{ \phi_{i}: i \in \N \}$ and $\{ h_{i}: i \in \N \}$ are bounded in $\Phi$ and $b\mathcal{P}$ respectively. Therefore, there exists $M>0$ such that $\sup_{(t,\omega)} \abs{h_{i}(t,\omega)} \leq M$ for all $i \in \N$, and for each $n \in \N$ there exists $N_{n}>0$ such that $p_{n}(\phi_{i}) \leq N_{n}$ for all $i \in \N$. Then, for each $n \in \N$ we have 
$$ \sum_{i=1}^{\infty} \sup_{(t,\omega)} p_{n} \left( \lambda_{i}  h_{i}(t,\omega) \phi_{i} \right) = \sum_{i=1}^{\infty} \abs{\lambda_{i}} \sup_{(t,\omega)} \abs{h_{i}(t,\omega)} p_{n} \left( \phi_{i} \right) \leq M N_{n}  \sum_{i=1}^{\infty} \abs{\lambda_{i}}< \infty. $$
Hence, for each $(t,\omega) \in \R_{+} \times \Omega$, the series $\displaystyle{\sum_{i=1}^{\infty} \lambda_{i}  h_{i}(t,\omega) \, \phi_{i}}$ converges in $\Phi$. So, the map $\tilde{H}$ is well-defined and is clearly weakly predictable. Moreover, the image of $\tilde{H}$ is bounded in $\Phi$ since each seminorm $p_{n}$ is bounded under the image of $\tilde{H}$ (see \cite{Treves}, Proposition 14.5, p.139).     
\end{proof}

In Section \ref{sectionCharacWeakIntegransNuclear} and under the assumption that $\Phi$ is a complete barrelled nuclear space (not necessarily metrizable), we will study characterizations for integrands in $\Phi \, \widehat{\otimes}_{\pi} \, b\mathcal{P}$ as $\Phi$-valued weakly predictable bounded processes.

\section{Stochastic Integration in Nuclear Spaces}\label{sectStochIntegNuclear}

In this section we will apply the theory of stochastic integration developed in the last section to the context of stochastic integration in nuclear spaces. We will see that in this setting the class of locally convex integrands can be fully characterized in terms of families of weakly bounded processes. 

\subsection{Nuclear Spaces and Cylindrical Semimartingales}
\label{sectNuclearSpaceCylSemi}

In this section we quickly review some concepts on nuclear spaces and (cylindrical) semimartingales defined in their dual spaces. 

Recall that a (Hausdorff) locally convex space $\Phi$ is called \emph{nuclear} if its topology is generated by a family $\Pi$ of Hilbertian semi-norms such that for each $p \in \Pi$ there exists $q \in \Pi$, satisfying $p \leq q$ and such that $i_{p,q}: \Phi_{q} \rightarrow \Phi_{p}$ is Hilbert-Schmidt. Other equivalent definitions of nuclear spaces can be found in \cite{Jarchow, Pietsch, Treves}. 
 
For our application of the theory of stochastic integration we will usually require that our nuclear space $\Phi$ be also  complete and barrelled. In such case the spaces $\Phi$ and $\Phi'$ are reflexive (see Theorem IV.5.6 in \cite{Schaefer}, p.145).  

The following are all examples of complete, ultrabornological (hence barrelled) nuclear spaces: 
the spaces of functions $\mathscr{E}_{K} \defeq \mathcal{C}^{\infty}(K)$ ($K$: compact subset of $\R^{d}$) and $\mathscr{E}\defeq \mathcal{C}^{\infty}(\R^{d})$, the rapidly decreasing functions $\mathscr{S}(\R^{d})$, and the space of test functions $\mathscr{D}(U) \defeq \mathcal{C}_{c}^{\infty}(U)$ ($U$: open subset of $\R^{d}$), as well are the spaces of distributions $\mathscr{E}'_{K}$, $\mathscr{E}'$, $\mathscr{S}'(\R^{d})$, and $\mathscr{D}'(U)$. Other examples are the space of harmonic functions $\mathcal{H}(U)$ ($U$: open subset of $\R^{d}$), the space of polynomials $\mathcal{P}_{n}$ in $n$-variables and the space of real-valued sequences $\R^{\N}$ (with direct product topology). For references see \cite{Pietsch, Schaefer, Treves}. 

Semimartingales and cylindrical semimartingales in the dual of nuclear space $\Phi$ where studied in \cite{FonsecaMora:Semi}. In the next paragraphs we recall some of the key properties of such semimartingales that we will need for our study of the stochastic integral. We start with the following result which shows that Assumption \ref{assuCylSemimartingale} is very natural within the context of nuclear spaces.   
  
\begin{proposition}\label{propCylSemimContOpeSpaceSemim}
Suppose that $\Phi$ is nuclear and let $X=( X_{t}: t \geq 0) $ be a cylindrical semimartingale in $\Phi'$. The following statements are equivalent:
\begin{enumerate} 
\item There exists a weaker countably Hilbertian topology $\theta$ on $\Phi$ such that $X$ extends to a continuous map from $\widehat{\Phi}_{\theta}$ into $S^{0}$. 
\item The mapping $X:\Phi \rightarrow S^{0}$, $\phi \mapsto X(\phi)$, is continuous. 
\item  For each $T>0$, the family of linear maps $( X_{t}: t \in [0,T])$ from $\Phi$ into $L^{0}\ProbSpace$ is equicontinuous (at the origin).
\item For each $T > 0$, the Fourier transforms 
of the family $( X_{t}: t \in [0,T] )$ are equicontinuous (at the origin) in $\Phi$.
\end{enumerate}
If we additionally assume that $\Phi$ is ultrabornological the above statements are also equivalent to 
\begin{enumerate}\setcounter{enumi}{4}
\item  For each $t \geq 0$, $X_{t}:\Phi \rightarrow L^{0}\ProbSpace$ is continuous (at the origin). 
\end{enumerate}
\end{proposition} 
\begin{proof} We have $(1) \Rightarrow (2)$ since the canonical inclusion from $\Phi$ into $\widehat{\Phi}_{\theta}$ is continuous. For $(2) \Leftrightarrow (3) \Leftrightarrow (5)$ and $(2) \Rightarrow (1)$ see Propositions 3.14 and 3.15 in \cite{FonsecaMora:Semi}. For $(3) \Leftrightarrow (4)$ see Proposition IV.3.4 in \cite{VakhaniaTarieladzeChobanyan}. 
\end{proof}

If $\Phi$ is a nuclear space, a $\Phi'$-valued process $X=(X_{t}: t \geq 0)$ is called a \emph{semimartingale} if the induced cylindrical process is a cylindrical semimartingale. In a completely analogue way we define the concepts of $\Phi'$-valued special semimartingale, martingale, local martingale and process of finite variation, etc. 
Observe that in the above definition we have not assumed that a $\Phi'$-valued semimartingale is c\`{a}dl\`{a}g. However, the next result shows that under some conditions a c\`{a}dl\`{a}g version exists. 

\begin{theorem}[\cite{FonsecaMora:Semi}, Theorem 3.7] \label{theoRegulCylinSemimartigalesNuclear}
Let $X=( X_{t}: t \geq 0) $ be a cylindrical semimartingale in $\Phi'$ satisfying any of the equivalent conditions in Proposition \ref{propCylSemimContOpeSpaceSemim}. Then, there exists a weaker countably Hilbertian topology $\theta$ on $\Phi$ and a $(\widehat{\Phi}_{\theta})'$-valued c\`{a}dl\`{a}g semimartingale $Y= ( Y_{t} : t \geq 0)$, such that for every $\phi \in \Phi$, $\inner{Y}{\phi}= ( \inner{Y_{t}}{\phi}: t \geq 0)$ is a version of $X(\phi)= ( X_{t}(\phi): t \geq 0)$. Moreover, $Y$ is a $\Phi'$-valued, regular, c\`{a}dl\`{a}g semimartingale that is a version of $X$ and it is unique up to indistinguishable versions.  Furthermore, if for each $\phi \in \Phi$ the real-valued semimartingale $X(\phi)$ is continuous, then $Y$ is a  continuous process in $(\widehat{\Phi}_{\theta})'$ and in $\Phi'$. 
\end{theorem} 

\begin{example} Let $Z=(Z_{t}: t \geq 0)$ be a $\R^{d}$-valued semimartingale and let $\delta_{x}$ denotes the Dirac measure at $x \in \R^{d}$. As an application of Theorem \ref{theoRegulCylinSemimartigalesNuclear} it is shown in Example 3.10 in \cite{FonsecaMora:Semi} (see also Example III.1 in \cite{Ustunel:1982}) that the cylindrical processes in $\mathscr{D}'(\R^{d})$ defined for each $t \geq 0$ by 
$$ X_{t} (\phi)= \delta_{Z_{t}}(\phi)=\phi(Z_{t}), \quad \forall \, \phi \in \mathscr{D}(\R^{d}), $$
defines a $\mathscr{D}'(\R^{d})$-valued c\`{a}dl\`{a}g semimartingale $(Y_{t}: t \geq 0)$ with Radon distributions such that $\forall \phi \in \mathscr{D}(\R^{d})$, $\Prob$-a.e. $\inner{Y_{t}}{\phi}=X_{t} (\phi)=\phi(Z_{t})$ for all $t \geq 0$. 

The above definition can be used also to produce examples of $\mathcal{H}^{1}_{S}$-cylindrical semimartingales. In effect, if $Z$ is a $\mathcal{H}^{2}_{S}$-semimartingale it is a direct consequence of It\^{o}'s formula that $X_{t}(\phi)=\phi(Z_{t})$ is a $\mathcal{H}^{1}_{S}$ semimartingale for each $\phi \in \mathscr{D}(\R^{d})$, hence $X$ is a $\mathcal{H}^{1}_{S}$-cylindrical semimartingale in $\mathscr{D}'(\R^{d})$. 
\end{example}

\begin{remark}\label{remaSemiDefiUstunel}
In \cite{Ustunel:1982}, \"{U}st\"{u}nel introduced another definition for a semimartingale in the dual of a nuclear space by means of the concept of projective system of stochastic process. In \cite{Ustunel:1982} it is assumed that $\Phi$ is a complete, bornological nuclear space and $\Phi'$ is nuclear and complete. \"{U}st\"{u}nel's definition follows in the following way: let $( q_{i}: i \in I)$ be a family of Hilbertian seminorms generating the nuclear topology on $\Phi'$ and for each $i \in I$ let $k_{q_{i}}$ denotes the canonical inclusion of $\Phi'$ into the separable Hilbert space $(\Phi')_{q_{i}}$, and for $q_{i} \leq q_{j}$ let $k_{q_{i},q_{j}}$ denotes the canonical inclusion of $(\Phi')_{q_{j}}$ into $(\Phi')_{q_{i}}$. A \emph{projective system of semimartingales} is a family $X=( X^{i}: i \in I)$ where each $X^{i}$ is a $(\Phi')_{q_{i}}$-valued semimartingale, and  such that if $q_{i} \leq q_{j}$ then $k_{q_{i},q_{j}} X^{j}$ and $X^{i}$ are indistinguishable. A $\Phi'$-valued processes $Y$ is called a \emph{semimartingale} if $k_{q_{i}} Y= X^{i}$ for each $i \in I$. It is clear that any semimartingale defined in the projective limit sense is also a semimartingale in our cylindrical sense. Conversely if $\Phi'$ is a Fr\'{e}chet nuclear space then every $\Phi'$-valued semimartingale in the cylindrical sense is a semimartingale defined in the projective limit sense
(see \cite{Ustunel:1982-1}, Theorem II.1 and Corollary II.3). Observe that the definition of $\Phi'$-valued semimartingales as projective systems of semimartingales only makes sense if the strong dual $\Phi'$ is complete and nuclear. The above because if $\Phi'$ does not satisfy these assumptions it might not be possible to express $\Phi'$ as a projective limit of Hilbert spaces. 
\end{remark}
  
\subsection{The Stochastic Integral in the Nuclear Space Setting} \label{sectionCharacWeakIntegransNuclear}

We start this section by introducing the following class of processes:

\begin{definition} \label{defiBoundWeakInteg}
Let $\Phi$ be locally convex. We denote by $b\mathcal{P}(\Phi)$ the space of all $\Phi$-valued processes $H=(H_{t}: t \geq 0)$ which are:
\begin{enumerate}
\item \emph{weakly predictable}, that is  
$\forall f  \in \Phi'$ the mapping $(t,\omega) \mapsto \inner{f}{H_{t}(\omega)}$ is predictable, 
\item \emph{weakly bounded}, that is $\forall f  \in \Phi'$ we have $\sup_{(t,\omega)} \abs{\inner{f}{H_{t}(\omega)}}< \infty$. 
\end{enumerate} 
In other words, we have $\inner{f}{H}\defeq \{ \inner{f}{H_{t}(\omega)}: t \geq 0, \omega \in \Omega   \} \in b\mathcal{P}$ for every $f \in \Phi'$.  
\end{definition}

The main objective of this section is to construct the  stochastic integral for integrands that belongs to $b\mathcal{P}(\Phi)$ under the assumption that $\Phi$ is a complete, barrelled, nuclear space and that $X$ is a cylindrical semimartingale satisfying Assumption \ref{assuCylSemimartingale}. To do this, we will show that the spaces $\Phi \, \widehat{\otimes}_{\pi} \, b\mathcal{P}$ and $b\mathcal{P}(\Phi)$ are homeomorphic when $b\mathcal{P}(\Phi)$ is equipped with the topology of uniform convergence on bounded subsets in $\Phi'$. This result will be used together with Theorem \ref{theoContiWeakIntegConvexifSemiSpace} to define the integral for processes in $b\mathcal{P}(\Phi)$.  To carry out this program, it will be of great importance the following result which is a consequence of the nice properties of tensor products and nuclear spaces. 

\begin{proposition} \label{propProjTensProdNuclearCase}
If $\Phi$ is a complete, barrelled, nuclear space, $\Phi \, \widehat{\otimes}_{\pi} \, b\mathcal{P} \simeq \mathcal{L}_{b}(\Phi', b\mathcal{P})$ (here $b$ denotes the topology of uniform convergence on bounded subsets in $\Phi'$). 
\end{proposition}
\begin{proof}
Since $\Phi$ is complete nuclear and $b\mathcal{P}$ is a Banach space, it follows from Proposition 50.4 in \cite{Treves}, p.522, that $\Phi \, \widehat{\otimes}_{\pi} \,b\mathcal{P} \simeq \mathcal{L}_{e}(\Phi', b\mathcal{P})$ (here $e$ denotes the topology of equicontinuous convergence). Then, as $\Phi$ is barrelled the equicontinuous and bounded subsets in $\Phi'$ coincide (see \cite{Schaefer}, Theorem IV.5.2, p.141) and hence we have $\mathcal{L}_{e}(\Phi', b\mathcal{P}) \simeq \mathcal{L}_{b}(\Phi', b\mathcal{P})$, which shows the result. 
\end{proof}

In view of Proposition \ref{propProjTensProdNuclearCase} our objective is to show that there exists an isomorphism between the spaces $b\mathcal{P}(\Phi)$ and $\mathcal{L}_{b}(\Phi', b\mathcal{P})$. In the next result we will show that such isomorphism exists in the more general context of reflexive locally convex spaces.   

\begin{proposition} \label{propIsomorBounPredSpaceLinOper}
Suppose that the locally convex space $\Phi$ is reflexive. Then, the map from $b\mathcal{P}(\Phi)$ into $\mathcal{L}_{b}(\Phi', b\mathcal{P})$ defined by 
\begin{equation} \label{defIsomorBounPredSpaceLinOper}
H \mapsto \left[ f \mapsto \inner{f}{H} \right].  
\end{equation}
is an isomorphism.  
\end{proposition}
\begin{proof}
First we show that the map defined by \eqref{defIsomorBounPredSpaceLinOper} is well-defined. Let $H \in b\mathcal{P}(\Phi)$. The map $f \mapsto \inner{f}{H}$  is clearly linear. We will show that it is also continuous. In effect, because $\Phi$ is reflexive then  $\Phi'$ is barrelled (see \cite{Schaefer}, Theorem IV.5.6, p.145), hence from the closed graph theorem (see \cite{NariciBeckenstein}, Theorem 14.3.4, p.465) we only need to show that the map $f \mapsto \inner{f}{H}$ is closed. 

Suppose the net $f_{\lambda}$ converges to $f$ in $\Phi'$ and that $\inner{f_{\lambda}}{H}$ converges to $Y$ in $b\mathcal{P}$. Since $\Phi$ is reflexive, for a subset of $\Phi$ the property of being weakly bounded implies strong boundedness (see \cite{Schaefer}, Theorem IV.5.2, p.141), therefore the image of $H$ is contained in a bounded subset of $\Phi$, say $B$. Then we have that $q_{B}(g)=\sup_{\phi \in B} \abs{\inner{g}{\phi}}$ $\forall g \in \Phi'$, is a continuous seminorm on $\Phi'$, and hence $q_{B}(f_{\lambda}-f) \rightarrow 0$. Therefore
$$ \sup_{(t,\omega)} \abs{ \inner{f_{ \lambda}}{H_{t}(\omega)}-\inner{f}{H_{t}(\omega)} } \leq q_{B}(f_{\lambda}-f) \rightarrow 0, $$
and hence $\inner{f_{\lambda}}{H}$ converges to $\inner{f}{H}$ in $b\mathcal{P}$. By uniqueness of limits we have $Y=\inner{f}{H}$. This proves that $f \mapsto \inner{f}{H} \in \mathcal{L}(\Phi', b\mathcal{P})$ and therefore the map given  in \eqref{defIsomorBounPredSpaceLinOper} is well-defined.

Now, the map defined by \eqref{defIsomorBounPredSpaceLinOper} is clearly linear and has kernel $\{0\}$, hence it is injective. To prove that it is also surjective, let $G \in \mathcal{L}(\Phi', b\mathcal{P})$. Then, for all $f \in \Phi'$ we have $G(f) \in b\mathcal{P}$. Given $(t,\omega) \in \R_{+} \times \Omega$, let $H_{t}(\omega)$  be defined by 
\begin{equation} \label{defInverIsomoBoundPredSpaceLinOper}
\inner{f}{H_{t}(\omega)}=G(f)_{t}(\omega), \quad \, \forall \, f \in \Phi'.
\end{equation}
Because $\Phi$ is reflexive, the prescription given above defines $H_{t}(\omega)$ as an element of $\Phi$. Hence, $H:\R_{+} \times \Omega \rightarrow \Phi$ is well-defined. Moreover, \eqref{defInverIsomoBoundPredSpaceLinOper} shows that $H$ is weakly predictable. Furthermore, because $G(f) \in b\mathcal{P}$ it follows from \eqref{defInverIsomoBoundPredSpaceLinOper} that $H$ is weakly bounded. Hence, $H$ defined in \eqref{defInverIsomoBoundPredSpaceLinOper} is an element of $ b\mathcal{P}(\Phi)$. Therefore the map defined by \eqref{defIsomorBounPredSpaceLinOper}  is surjective and consequently it is an isomorphism. 
\end{proof}

If $\Phi$ is reflexive, in view of Proposition \ref{propIsomorBounPredSpaceLinOper} we can (and we will) equip the space $b\mathcal{P}(\Phi)$ with the topology of uniform convergence on bounded subsets in $\Phi'$ induced via the isomorphism \eqref{defIsomorBounPredSpaceLinOper}. That is, we equip $b\mathcal{P}(\Phi)$ with the topology generated by the family of seminorms 
$\displaystyle{q_{B}(H)=\sup_{f \in B} \norm{ \inner{f}{H}}_{u}}$, where $B$ ranges over the bounded subsets in $\Phi'$. 

Moreover, since $\Phi$ is reflexive, a fundamental system of bounded, convex, balanced subsets in $\Phi'$ is the family of all the polar sets $B_{p}(1)^{0}$ where $p$ ranges over the continuous seminorms on $\Phi$ (see Theorem IV.5.2 in \cite{Schaefer}, p.141). Hence, the topology of uniform convergence on bounded subsets in $\Phi'$ introduced on $b\mathcal{P}(\Phi)$ in the above paragraph can be equivalently described as the topology generated by the family of seminorms $\displaystyle{p(H)=\sup_{(t,\omega)} p(H_{t}(\omega))}$, where $p$ ranges over a generating family of seminorms for the topology on $\Phi$. Hence, the topology in $b\mathcal{P}(\Phi)$ is the \emph{topology of convergence in $\Phi$ uniformly on $\R_{+} \times \Omega$}. Unless otherwise specified we will always consider $b\mathcal{P}(\Phi)$ equipped with this topology.  Observe that $b\mathcal{P}(\Phi)$ is Fr\'{e}chet if $\Phi$ is a Fr\'{e}chet nuclear space.  

Combining the results in Propositions  \ref{propProjTensProdNuclearCase} and  \ref{propIsomorBounPredSpaceLinOper} we obtain the following characterization for the space of weak integrands for some classes of nuclear spaces.

\begin{theorem}\label{theoWeakIntegNuclearBoundPrec} Assume that $\Phi$ is a complete, barrelled, nuclear space. Then, 
$$\Phi \, \widehat{\otimes}_{\pi} \, b\mathcal{P} \simeq \mathcal{L}_{b}(\Phi', b\mathcal{P}) \simeq b\mathcal{P}(\Phi).$$ 
\end{theorem} 


\begin{corollary}\label{coroDenseSimpleBoundedProcesses} Assume that $\Phi$ is a complete, barrelled, nuclear space.
The collection of all the $\Phi$-valued elementary processes:
\begin{equation}\label{eqDefiSimpleBoundProceNuclear}
H_{t}(\omega)=\sum_{k=1}^{n} h_{k}(r,\omega) \phi_{k}, \quad \forall \, t \geq 0, \, \omega \in \Omega, 
\end{equation}
where $n \in \N$, $h_{k} \in b\mathcal{P}$ and $\phi_{k} \in \Phi$ for $k=1,\dots, n$, is dense in $b\mathcal{P}(\Phi)$.  
\end{corollary}
\begin{proof}
The homeomorphism of Proposition \ref{propProjTensProdNuclearCase} maps the element $\sum_{k=1}^{n} \phi_{k} \otimes h_{k}$ in $\Phi \otimes b\mathcal{P}$ into the continuous linear operator $\left( f \mapsto \sum_{k=1}^{n} \inner{f}{\phi_{k}} h_{k} \right)$ from $\Phi'$ into $b\mathcal{P}$. But by Proposition \ref{propIsomorBounPredSpaceLinOper} the later is mapped into \eqref{eqDefiSimpleBoundProceNuclear} which is clearly an element in $b\mathcal{P}(\Phi)$. Since the elements in $\Phi \otimes b\mathcal{P}$ are dense in $\Phi \, \widehat{\otimes}_{\pi} \, b\mathcal{P}$, the homeomorphism in Theorem \ref{theoWeakIntegNuclearBoundPrec} shows that the elementary processes of the form \eqref{eqDefiSimpleBoundProceNuclear} are dense in $b\mathcal{P}(\Phi)$.   
\end{proof}

We are ready to define the stochastic integral for integrands in $b\mathcal{P}(\Phi)$. 

\begin{theorem} \label{theoBoundedWeakIntegNucleSpace}
Let $\Phi$ be a complete, barrelled, nuclear space and let $X$ be a cylindrical semimartingale satisfying any of the equivalent conditions in Proposition \ref{propCylSemimContOpeSpaceSemim}. There exists a unique continuous and linear mapping $H \mapsto \int \, H \, dX$ from $ b\mathcal{P}(\Phi)$ into $(S^{0})_{lcx}$ such that for each $H \in  b\mathcal{P}(\Phi)$ of the form \eqref{eqDefiSimpleBoundProceNuclear} we have 
\begin{equation}\label{eqActionWeakIntegSimpleIntegNuclear}
 \int \, H \, dX =  \sum_{k=1}^{n} \, h_{k} \cdot X(\phi_{k}). 
\end{equation}
Moreover for every $H \in b\mathcal{P}(\Phi)$ we have:
\begin{enumerate}
\item $\displaystyle{\left(\int H \, dX \right)^{c}= \int H \, dX^{c}}$. 
\item $\displaystyle{\left(\int H \, dX \right)^{\tau}=\int H \mathbbm{1}_{[0,\tau]} \, dX= \int H \, dX^{\tau}}$, for every stopping time $\tau$. 
\item If $Y$ is another cylindrical semimartingale in $\Phi'$ satisfying any of the equivalent conditions in Proposition \ref{propCylSemimContOpeSpaceSemim}, then the processes $\int \, H d (X+Y)$ and $\int \, H d X +\int \, H d Y$ are indistinguishable. 
\end{enumerate} 
\end{theorem}
\begin{proof}
For each $H \in  b\mathcal{P}(\Phi)$ we define the  stochastic integral $\int H \, dX$ of $H$ with respect to $X$ as the stochastic integral of Theorem \ref{theoContiWeakIntegConvexifSemiSpace} for the unique image of $H$ in $\Phi \, \widehat{\otimes}_{\pi} \,  b\mathcal{P}$ under the homeomorphism of Theorem \ref{theoWeakIntegNuclearBoundPrec}. It then follows from Theorem \ref{theoContiWeakIntegConvexifSemiSpace} that the  mapping $H \mapsto \int \, H \, dX$ is linear and continuous from $ b\mathcal{P}(\Phi)$ into $(S^{0})_{lcx}$ and that it acts on $H \in  b\mathcal{P}$ of the form \eqref{eqDefiSimpleBoundProceNuclear} as given in \eqref{eqActionWeakIntegSimpleIntegNuclear}. Properties \emph{(1)}, \emph{(2)} and \emph{(3)} are also a consequence of the respective properties of the integral mapping of Theorem \ref{theoContiWeakIntegConvexifSemiSpace}. In the particular case of property \emph{(2)}, observe that  $H \mathbbm{1}_{[0,\tau]} \in  b\mathcal{P}(\Phi)$ and hence  $\displaystyle{\int H \mathbbm{1}_{[0,\tau]} \, dX}$ is defined. Since for each $h \in b\mathcal{P}$, $\phi \in \Phi$,  $( h \cdot X(\phi))^{\tau}=( h \mathbbm{1}_{[0,\tau]}  ) \cdot X(\phi)= h \cdot \left( X(\phi)^{\tau} \right)$, then it is a consequence of \eqref{eqActionWeakIntegSimpleIntegNuclear} that equality in \emph{(2)} holds for $H$ of the form \eqref{eqDefiSimpleBoundProceNuclear}. The result extends to every $H \in  b\mathcal{P}(\Phi)$ by the density of the elementary processes in $b \mathcal{P}(\Phi)$ (Corollary \ref{coroDenseSimpleBoundedProcesses}) and the continuity of the mapping $H \mapsto \int \, H \, dX$ from $ b\mathcal{P}(\Phi)$ into $(S^{0})_{lcx}$.
\end{proof}

In the second part of this section we will explore some additional properties of the stochastic integral defined in Theorem \ref{theoBoundedWeakIntegNucleSpace}. In order to prove these properties we will request our cylindrical semimartingales to satisfy the following:

\begin{definition}\label{defiGoodIntegrators}
Let $\Phi$ be a complete, barrelled, nuclear space. We will say that a cylindrical semimartingale $X=(X_{t}: t \geq 0)$ in $\Phi'$ is a \emph{good integrator} if $X$ 
satisfies any of the equivalent conditions in Proposition \ref{propCylSemimContOpeSpaceSemim} and if the stochastic integral mapping $H \mapsto \int \, H \, dX$ of Theorem \ref{theoBoundedWeakIntegNucleSpace} defines a continuous linear mapping from $b\mathcal{P}(\Phi)$ into $S^{0}$. 
\end{definition}

The definition of good integrators introduced above is a generalization of the definition for real-valued semimartingales. In efect, every $z \in S^{0}$ defines a cylindrical semimartingale $\alpha \mapsto \alpha z$ which is by definition continuous from $\R$ into $S^{0}$. By construction, if $\Phi=\R$ our stochastic integral coincides with that for real-valued semimartingales and hence the Bichteler–Dellacherie theorem implies that the stochastic integral mapping  $H \mapsto \int \, H \, dX$ of Theorem \ref{theoBoundedWeakIntegNucleSpace} is a continuous linear mapping from $b\mathcal{P}(\R)=b\mathcal{P}$ into $S^{0}$.  

We do not know if every cylindrical semimartingale $X$ in $\Phi'$ (satisfying any of the equivalent conditions in Proposition \ref{propCylSemimContOpeSpaceSemim}) is a good integrator.  
However, the next result shows that a positive answer is obtained for some important classes of cylindrical semimartingales. 

\begin{proposition}\label{propGoodIntegratorsNuclearSpace}
Let $\Phi$ be a complete, barrelled, nuclear space,and let $X$ be a cylindrical semimartingale in $\Phi'$ satisfying any of the equivalent conditions in Proposition \ref{propCylSemimContOpeSpaceSemim}. For the stochastic  integral mapping $H \mapsto \int \, H \, dX$ of Theorem \ref{theoBoundedWeakIntegNucleSpace} it follows that:
\begin{enumerate}
\item If $X$ is a $\mathcal{H}^{p}_{S}$-cylindrical semimartingale, $H \mapsto \int \, H \, dX$ is continuous from $b\mathcal{P}(\Phi)$ into $\mathcal{H}^{p}_{S}$.  
\item If $X$ is a $\mathcal{M}_{\infty}^{2}$-cylindrical martingale, $H \mapsto \int \, H \, dX$ is continuous from $b\mathcal{P}(\Phi)$ into $\mathcal{M}_{\infty}^{2}$. 
\item If $X$ is a $\mathcal{A}$-cylindrical semimartingale, $H \mapsto \int \, H \, dX$ is continuous from $b\mathcal{P}(\Phi)$ into $\mathcal{A}$.
\end{enumerate}
In particular, if  $X$ is either a $\mathcal{H}^{p}_{S}$-cylindrical semimartingale, a $\mathcal{M}_{\infty}^{2}$-cylindrical martingale, or a $\mathcal{A}$-cylindrical semimartingale, then $X$ is a good integrator.  
\end{proposition}
\begin{proof}
To prove (1), assume that $X$ is a $\mathcal{H}^{p}_{S}$-cylindrical semimartingale. We know by Remark \ref{remaLocaConvWeakIntegSpSemima} that the stochastic  integral mapping defined in Theorem \ref{theoContiWeakIntegConvexifSemiSpace} coincides with the stochastic integral mapping defined in Proposition \ref{propWeakInteForSpSemimar}. But we know from this last result that the stochastic integral mapping is continuous from $\Phi \, \widehat{\otimes}_{\pi} \, b\mathcal{P}$ into $\mathcal{H}^{p}_{S}$  and by our definition of the integral in  Theorem \ref{theoBoundedWeakIntegNucleSpace} we obtain the desired conclusion. The proof of (2) and (3) follows from the arguments used above but applying Proposition \ref{propWeakInteForSquaMarting} instead of Proposition \ref{propWeakInteForSpSemimar}. 
\end{proof} 

\begin{example}
Let $B=(B_{t}: t \geq 0)$ denotes a real-valued Brownian motion. For every $t \geq 0$ define 
$$X_{t}(\phi)= \int_{0}^{t} \phi(s) dB_{s}, \quad \forall \phi \in \mathcal{S}(\R). $$
Then, using the properties of the It\^{o} stochastic integral one can show that $X=(X_{t}: t \geq 0)$ is a $\mathcal{S}'(\R)$-valued process which defines a $\mathcal{M}_{\infty}^{2}$-cylindrical martingale  satisfying Assumption  \ref{assuCylSemimartingale} (for details see Example 3.20 in \cite{FonsecaMora:Semi}). Therefore $X$ is a good integrator by Proposition \ref{propGoodIntegratorsNuclearSpace}.
 
Since $\mathcal{S}(\R)$ is metrizable, the stochastic integral with respect to $X$ can be represented as a series of stochastic integrals. In effect, let $H \in b\mathcal{P}(\mathcal{S}(\R))$. If the unique image of $H$ in $\mathcal{S}(\R) \, \widehat{\otimes}_{\pi} \, b\mathcal{P}$ as given in Theorem \ref{theoWeakIntegNuclearBoundPrec} satisfies the decomposition \eqref{decompElemeProjTensProduc} for some $\sum \abs{\lambda_{i}}< \infty$, $\phi_{i} \rightarrow 0$ in $\mathcal{S}(\R)$, $h_{i} \rightarrow 0$ in $b\mathcal{P}$, then as a $\mathcal{S}(\R)$-valued process $H$ satisfies the representation
$$H(t,\omega)=\sum_{i=1}^{\infty} \lambda_{i} h_{i}(t,\omega) \phi_{i}, $$ 
with convergence in $b\mathcal{P}(\mathcal{S}(\R))$. Hence, by Proposition \ref{propGoodIntegratorsNuclearSpace}, \eqref{eqActionWeakIntegSimpleIntegNuclear}, and the associativity of the stochastic integral for real-valued semimartingales we have
$$ \int^{t}_{0} \, H \, dX=  \sum_{i=1}^{\infty} \lambda_{i} ( h_{i} \cdot X(\phi_{i}))= \sum_{i=1}^{\infty} \lambda_{i} \int_{0}^{t} h_{i}(s) \phi_{i}(s) dB_{s}, \quad \forall t \geq 0. $$
\end{example}

As the next result shows, we can use Proposition \ref{propGoodIntegratorsNuclearSpace} to construct new examples of good integrators. We will need the Banach space $\mathcal{M}^{2}_{\infty} \oplus \mathcal{A}$ (introduced by Memin in \cite{Memin:1980}) of special semimartingales with canonical decomposition $x=m+a$ where $m \in \mathcal{M}^{2}_{\infty}$ and $a \in \mathcal{A}$, equipped with the norm $\norm{x}_{\mathcal{M}^{2}_{\infty} \oplus \mathcal{A}}= \norm{m}_{\mathcal{M}_{\infty}^{2}}+\norm{a}_{\mathcal{A}}$. 

\begin{corollary}\label{coroSpecialSemiIsGoodInteg}
Let $\Phi$ be a complete, barrelled, nuclear space,and let $X$ be a $\mathcal{M}^{2}_{\infty} \oplus \mathcal{A}$-cylindrical semimartingale in $\Phi'$ satisfying any of the equivalent conditions in Proposition \ref{propCylSemimContOpeSpaceSemim}. Then, there exists a $\mathcal{M}^{2}_{\infty}$-cylindrical martingale $M$ and a $\mathcal{A}$-cylindrical semimartingale $A$, $M$ and $A$ satisfying any of the equivalent conditions in Proposition \ref{propCylSemimContOpeSpaceSemim}, such that $X=M+A$ and 
$$ \int \, H \, dX =  \int \, H \, dM + \int \, H \, dA  \in \mathcal{M}^{2}_{\infty} \oplus \mathcal{A}, \quad \forall  H \in b\mathcal{P}(\Phi). $$
Moreover, the stochastic  integral mapping $H \mapsto \int \, H \, dX$ of Theorem \ref{theoBoundedWeakIntegNucleSpace} is continuous from $b\mathcal{P}(\Phi)$ into $\mathcal{M}^{2}_{\infty} \oplus \mathcal{A}$. In particular, $X$ is a good integrator.  
\end{corollary}
\begin{proof}
For each $\phi \in \Phi$ we have $X(\phi) \in \mathcal{M}^{2}_{\infty} \oplus \mathcal{A}$ and hence we have the canonical decomposition 
$X(\phi)=m_{\phi}+a_{\phi}$. Define $M: \Phi \rightarrow \mathcal{M}^{2}_{\infty}$ and $A: \Phi \rightarrow \mathcal{A}$ respectively by   $M(\phi)=m_{\phi}$ and $A(\phi)=A_{\phi}$ for each $\phi \in \Phi$. The uniqueness of the canonical decomposition shows that $M$ and $A$ are linear, hence $M$ is a $\mathcal{M}^{2}_{\infty}$-cylindrical martingale and $A$ is a $\mathcal{A}$-cylindrical semimartingale. Moreover, $X=M+A$. 

Observe that by our assumption $X$ is linear and continuous from $\Phi$ into $S^{0}$, and since the image of $X$ is in $\mathcal{M}^{2}_{\infty} \oplus \mathcal{A}$, by an application of the closed graph theorem it follows that  $X$ is linear and continuous from $\Phi$ into  $\mathcal{M}^{2}_{\infty} \oplus \mathcal{A}$. Now, since $\mathcal{M}^{2}_{\infty} \oplus \mathcal{A}$ is a direct sum of the Banach spaces $\mathcal{M}^{2}_{\infty}$ and $\mathcal{A}$, the projection $ x=m+a \mapsto m$ (respectively $ x=m+a \mapsto a$) is linear continuous from $\mathcal{M}^{2}_{\infty} \oplus \mathcal{A}$ into $\mathcal{M}^{2}_{\infty}$ (respectively into $\mathcal{A}$). Thus, since $X$ is continuous from $\Phi$ into $\mathcal{M}^{2}_{\infty} \oplus \mathcal{A}$, then $M$ (respectively $A$) is continuous from $\Phi$ into $\mathcal{M}^{2}_{\infty}$ (respectively $\mathcal{A}$). Since the canonical inclusion from $\mathcal{M}^{2}_{\infty}$ (respectively $\mathcal{A}$) into $S^{0}$ is linear and continuous, we conclude that $M$ (respectively $A$) is linear continuous from $\Phi$ into $S^{0}$, hence satisfies the equivalent conditions in Proposition \ref{propCylSemimContOpeSpaceSemim}. 

Hence, by linearity of the stochastic integral and Proposition \ref{propGoodIntegratorsNuclearSpace}, for each $H \in b\mathcal{P}(\Phi)$ we have 
$$ \int \, H \, dX =  \int \, H \, dM + \int \, H \, dA  \in \mathcal{M}^{2}_{\infty} \oplus \mathcal{A}, $$ 
and the stochastic integral mapping $H \mapsto \int \, H \, dX$  is continuous from $b\mathcal{P}(\Phi)$ into $\mathcal{M}^{2}_{\infty} \oplus \mathcal{A}$.  
\end{proof}

An application to the result in Corollary \ref{coroSpecialSemiIsGoodInteg} will be given in Section \ref{subsecIntegSequenSemima} in our study of the stochastic integral with respect to a sequence of real-valued semimartingales. 


We return to the study of the properties of the stochastic integral and in particular we will start by studying its jumps.  For a $\Phi'$-valued c\`{a}dl\`{a}g adapted process $(Z_{t}: t \geq 0)$, recall $(\Delta Z)_{t}\defeq Z_{t}-Z_{t-}$ is the jump al time $t$ of $Z$.  

\begin{proposition}\label{propJumpStochIntegraBoundPred} Let $\Phi$ be a complete, barrelled, nuclear space.  
Suppose that $X=(X_{t}: t \geq 0)$ is a good integrator in $\Phi'$ and let $\widetilde{X}=(\widetilde{X}_{t}: t \geq 0)$ denotes the $\Phi'$-valued, regular, c\`{a}dl\`{a}g version of $X$ as in Theorem \ref{theoRegulCylinSemimartigalesNuclear}. For every $H \in b\mathcal{P}(\Phi)$ we have
\begin{equation}\label{eqJumpsStochInteg}
\Delta \left( \int_{0}^{t} H \, dX \right)= \inner{ \Delta \widetilde{X}_{t}}{H(t)}, \quad \forall \, t \geq 0. 
\end{equation}
\end{proposition}
\begin{proof}
We prove the result first for elementary processes. To to this, observe that for every $\phi \in \Phi$ and $h \in b\mathcal{P}$, it is a well-known property of the stochastic integral (see e.g. Theorem 12.3.22 in \cite{CohenElliott}) that we have $\Delta \left( h \cdot X(\phi) \right)= h \, \Delta X(\phi)$.
Then for an elementary process $H \in  b\mathcal{P}(\Phi)$ of the form \eqref{eqDefiSimpleBoundProceNuclear}  it follows from \eqref{eqActionWeakIntegSimpleIntegNuclear} that 
\begin{eqnarray*}
\Delta \left( \int \, H \, dX \right)
& = & \sum_{k=1}^{n} \Delta \left( h_{k} \cdot X (\phi_{k}) \right) = \sum_{k=1}^{n}  h_{k} \, \Delta X(\phi_{k}) \\
& = &  \sum_{k=1}^{n}  h_{k} \, \inner{\Delta \widetilde{X}}{\phi_{k}} =  \inner{ \Delta \widetilde{X}}{H}.
\end{eqnarray*}

We now prove the general case. Given $H \in b\mathcal{P}(\Phi)$, it follows from Corollary \ref{coroDenseSimpleBoundedProcesses} that there exists a net $(H^{i}: i \in I)$ of elementary processes such that $H^{i} \rightarrow H$ in $b\mathcal{P}(\Phi)$. Since the convergence in $b\mathcal{P}(\Phi)$ is uniform, we then have that for each $r \geq 0$ and $\omega \in \Omega$ we have $H^{i}_{r}(\omega) \rightarrow H_{r}(\omega)$ in $\Phi$. Let $\Omega_{0}$ be a subset of $\Omega$ with $\Prob(\Omega_{0})=1$ and such that for each $\omega \in \Omega_{0}$, $\Delta \widetilde{X}_{t}(\omega) \in \Phi'$ for each $t \geq 0$. Then, for any given $t \geq 0$ we have 
$$ \inner{\Delta \widetilde{X}_{t}(\omega)}{H^{i}_{t}(\omega)} \rightarrow   
\inner{\Delta \widetilde{X}_{t}(\omega)}{H_{t}(\omega)}, \quad \forall  \, \omega \in \Omega_{0}. $$

On the other hand, as $X$ is a good integrator we have  $\int \, H^{i} \, dX \rightarrow \int \, H \, dX$ in $S^{0}$. Since the mapping that takes a c\`{a}dl\`{a}g process into its jump process is continuous under the topology of uniform convergence in probability on compact intervals of time (see Lemma 12.4.2 in \cite{CohenElliott}, p.277-8) and the latter topology is weaker than the semimartingale topology, then we have $\Delta \left( \int \, H^{i} \, dX \right) \rightarrow \Delta \left( \int \, H \, dX \right)$ in the topology of uniform convergence in probability on compact intervals of time. In particular, for every $t \geq 0$ we have 
 $\Delta \left( \int_{0}^{t} \, H^{i} \, dX \right) \rightarrow \Delta \left( \int_{0}^{t} \, H \, dX \right)$ 
with convergence in probability. Since \eqref{eqJumpsStochInteg} holds for every $H^{i}$, by uniqueness of limits we conclude \eqref{eqJumpsStochInteg}. 
\end{proof}


We continue our study with the following associativity property for the stochastic integral that relates our theory of stochastic integration with the stochastic integral in finite dimensions. 

\begin{proposition}\label{propAssociativeWeakIntegralNuclear} Let $\Phi$ be a complete, barrelled, nuclear space. Suppose that $X=(X_{t}: t \geq 0)$ is a good integrator in $\Phi'$. 
For every $g \in b\mathcal{P}$ and $H \in b\mathcal{P}(\Phi)$, we have
\begin{equation}\label{eqAssociatWeakIntegNuclear}
\int \, g \, dZ(H) = \int \, g H \, dX, 
\end{equation}
where $Z(H)= \int \, H \, dX$ and $(g H)_{t}(\omega)=g_{t}(\omega)H_{t}(\omega)$ for all $t \geq 0$, $\omega \in \Omega$. 
\end{proposition}
\begin{proof}
First we show that both sides of \eqref{eqAssociatWeakIntegNuclear} are well-defined. 
Since $Z(H) \in S^{0}$, then for every $g \in b\mathcal{P}$ we have that $\int \, g \, dZ(H) \in S^{0}$. Similarly, for every $g \in b\mathcal{P}$ and $H \in b\mathcal{P}(\Phi)$ we have that $g H \in  b\mathcal{P}(\Phi)$, thus $\int \, g H \, dX \in S^{0}$. 

Now we show the equality in \eqref{eqAssociatWeakIntegNuclear}. First, as $X$ is a good integrator the mapping $H \mapsto Z(H)$ is continuous from $b\mathcal{P}(\Phi)$ into $S^{0}$. Since the stochastic integral for real-valued semimartingales is continuous in the integrator, the mapping $H \mapsto \int \, g \, dZ(H)$ is continuous from $b\mathcal{P}(\Phi)$ into $S^{0}$. 

On the other hand, the mapping $H \mapsto gH$ is continuous from $b\mathcal{P}(\Phi)$ into $b\mathcal{P}(\Phi)$. Hence, as $X$ is a good integrator the mapping $H \mapsto \int \, g H \, dX$ is continuous from $b\mathcal{P}(\Phi)$ into $S^{0}$.  Thus, to show \eqref{eqAssociatWeakIntegNuclear} and in view of Corollary \ref{coroDenseSimpleBoundedProcesses} we only need to check that the equality in \eqref{eqAssociatWeakIntegNuclear} is valid for elementary integrands. 

In effect, if $H$ has the form \eqref{eqDefiSimpleBoundProceNuclear}, then by our the definition of the stochastic integral we have
$$ Z(H)= \sum_{k=1}^{n} \, h_{k} \cdot X(\phi_{k}). $$
Then, by linearity and associativity of the stochastic integral for real-valued semimartingales we have
$$ \int \, g \, dZ(H) = \sum_{k=1}^{n} \, g \cdot \left(h_{k} \cdot X(\phi_{k}) \right)
= \sum_{k=1}^{n} \, (g h_{k}) \cdot X(\phi_{k}) 
= \int \, g H \, dX. $$ 
\end{proof}

\begin{corollary}\label{coroWeakIntegF0Measurable} Let $\Phi$ be a complete, barrelled, nuclear space. Suppose that $X=(X_{t}: t \geq 0)$ is a good integrator in $\Phi'$. If $\xi$ is a $\mathcal{F}_{0}$-measurable real-valued bounded random variable, then for every $H \in b\mathcal{P}(\Phi)$ we have
$$ \int \, \xi  H \, dX= \xi \int \, H \, dX. $$
\end{corollary}
\begin{proof}
It is a well-known fact that if $z \in S^{0}$ and $h \in b\mathcal{P}$, then $\int \, \xi h \, dz= \xi \int \, h \, dz$ (see e.g. Lemma 4.14 in \cite{KarandikarRao}, p.95). Hence, the above fact and Proposition \ref{propAssociativeWeakIntegralNuclear} imply that:
$$ \int \, \xi  H \, dX=  \int \, \xi \, dZ(H)= \xi \int \, dZ(H) = \xi \int \, H \, dX. $$
\end{proof}

\begin{remark}
A simple direct consequence of Corollary \ref{coroWeakIntegF0Measurable} is that the integral is the same for members in $b\mathcal{P}(\Phi)$ of the same class of indistinguishability. Moreover, for a sequence 
$H^{n}$ to converge to $H$ in $b\mathcal{P}(\Phi)$ the convergence only needs to occurs almost surely  in $\Phi$ uniformly on $\R_{+} \times \Omega$. 
\end{remark}

\section{Extension of the Stochastic Integral in a Nuclear Space} \label{sectExteStocIntegNuclear}

\subsection{Stochastic Integral For Locally Bounded Integrands} \label{subSectStocIngLocalBound}

In the previous section we have introduced the stochastic integral for integrands that belongs to the space $b\mathcal{P}(\Phi)$. In this section we extend the stochastic integral mapping to the a more general family of integrands which are locally bounded in the sense defined below:

\begin{definition}\label{defiLocallyBoundedProcess}
Let $\Phi$ be locally convex. We will say that $H:\R_{+} \times \Omega \rightarrow \Phi$ is  \emph{locally bounded}  if there exists a  sequence $( \tau_{n} : n \in \N)$ of stopping times increasing to $\infty$ $\Prob$-a.e. such that for each $n \in \N$, the mapping $(t, \omega) \mapsto H^{\tau_{n}}_{t}(\omega) \defeq H_{t \wedge \tau_{n}} (\omega)$ takes its values in a bounded subset of $\Phi$ for $\Prob$-almost all $\omega \in \Omega$. 
\end{definition}

%

Let $\Phi$ be locally convex and $H=(H_{t}: t \geq 0)$ be a  $\Phi$-valued c\`{a}dl\`{a}g process. The process obtained by considering its left limits $H_{-}=(H_{t-}: t \geq 0)$ is a c\`{a}gl\`{a}d process. If $H$ is \emph{weakly adapted}, i.e. if  $\inner{f}{H} $ is adapted for each $f \in \Phi'$, the same is satisfied for $H_{-}$. Therefore, for each $f \in \Phi'$ we have that $\inner{f}{H_{-}} $ is a  c\`{a}gl\`{a}d adapted process and hence $\inner{f}{H_{-}}$ is locally bounded. However if $\Phi$ is only locally convex the weak locally boundedness of $H_{-}$ is not in general enough to show that $H_{-}$ is locally bounded. Nevertheless, a sufficient condition for c\`{a}dl\`{a}g processes in the dual of a nuclear space is given below.

\begin{proposition}\label{propLeftLimiLocallyBounDualNuclear}
Let $\Phi$ be a nuclear space. Let $X=(X_{t}: t \geq 0)$ be a $\Phi'$-valued, regular, c\`{a}dl\`{a}g, weakly adapted process such that $\forall \, T>0$, the family $(X_{t}: t \in [0,T])$ of linear mappings from $\Phi$ into $L^{0} \ProbSpace$ is equicontinuous. 
Then, the process $X_{-}=(X_{t-}: t \geq 0)$ is predictable and locally bounded in $\Phi'$. 
\end{proposition}
\begin{proof}
Our assumptions on $X$ together with the Regularization Theorem (see Theorem 3.2 in \cite{FonsecaMora:2018}) show that there exists a weaker countably Hilbertian topology $\vartheta$ on $\Phi$ and a $(\widehat{\Phi}_{\vartheta})'$-valued c\`{a}dl\`{a}g process $\widetilde{X}=(\widetilde{X}_{t}: t \geq 0)$, such that for every $\phi \in \Phi$, $\inner{\widetilde{X}}{\phi}$ is a version of $\inner{X}{\phi}$. Since $\widetilde{X}$ is also a $\Phi'$-valued, regular, c\`{a}dl\`{a}g process, then $X$ and $\widetilde{X}$ are indistinguishable (see Proposition 2.12 in \cite{FonsecaMora:2018}). Observe that for each $\phi \in \Phi$ the real-valued process $\inner{\widetilde{X}_{-}}{\phi}$ is left-continuous and adapted, hence is predictable. But as $\widetilde{X}_{-}$ is a  $(\widehat{\Phi}_{\vartheta})'$-valued process, and the Borel and cylindrical $\sigma$-algebras in $(\widehat{\Phi}_{\vartheta})'$ coincide, the fact that $\widetilde{X}_{-}$ is weakly predictable implies it is (strongly) predictable. By indistinguishability we have that $X_{-}$ is predictable.  

We now prove the existence of the localizing sequence of stopping times.  Following Section 7 in \cite{FonsecaMora:2018-1}, if we choose an increasing sequence of continuous Hilbertian semi-norms $( \rho_{n}: n \in \N)$ on $\Phi$ that generates the topology $\vartheta$, then 
$$  (\widehat{\Phi}_{\vartheta})' = \bigcup_{n \in \N} B_{\rho'_{n}}(n),$$
where $B_{\rho'_{n}}(n)=\{ f \in \Phi': \rho_{n}'(f) \leq n  \}$. The collection $\{ B_{\rho_{n}'}(n): n \in \N \}$ is an increasing sequence of  bounded, closed, convex, balanced subsets of $\Phi'$. Furthermore, since $\widetilde{X}$ is a  $(\widehat{\Phi}_{\vartheta})'$-valued c\`{a}dl\`{a}g process we have
\begin{equation} \label{levyAlmostSureCountUnionBallsInDual}
\widetilde{X}_{t}(\omega) \in \bigcup_{n \in \N} B_{\rho'_{n}}(n), \, \forall \, (t,\omega) \in \R_{+} \times \Omega. 
\end{equation} 

For each $n \in \N$ define $\tau_{n}$ by 
\begin{equation} \label{defiStoppingTimesTauNLevyNoise}
 \tau_{n}(\omega) \defeq \inf \{ t \geq 0: \widetilde{X}_{t}(\omega) \notin B_{\rho'_{n}}(n) \mbox{ or } \widetilde{X}_{t-}(\omega) \notin B_{\rho'_{n}}(n) \mbox{ or } t \geq n \}, \quad \forall \, \omega \in \Omega.
\end{equation}  
Each $\tau_{n}$ is a stopping time and from \eqref{levyAlmostSureCountUnionBallsInDual} it follows that $\tau_{n} \rightarrow \infty$ $\Prob$-a.e. as $n \rightarrow \infty$. Moreover, by the definition of $\tau_{n}$ we have
$\widetilde{X}^{\tau_{n}}_{t-}(\omega) \in B_{\rho'_{n}}(n)$ for all $(t,\omega) \in \R_{+} \times \Omega$. Since   $X$ and $\widetilde{X}$ are indistinguishable, we therefore have that  $X^{\tau_{n}}_{-}$ has a bounded image in $\Phi'$. Thus $X_{-}=(X_{t-}: t \geq 0)$ is predictable and locally bounded in $\Phi'$
\end{proof}

For our theory of stochastic integration the most useful version of Proposition \ref{propLeftLimiLocallyBounDualNuclear} is the following:

\begin{proposition} \label{propLeftLimitProcessLocallyBounded}
Suppose that $\Phi$ is reflexive and that $\Phi'$ is  nuclear. Let $X=(X_{t}: t \geq 0)$ be a $\Phi$-valued, regular, c\`{a}dl\`{a}g weakly adapted process such that $\forall \, T>0$, the family $(X_{t}: t \in [0,T])$ of linear mappings from $\Phi'$ into $L^{0} \ProbSpace$ is equicontinuous. Then, the process $X_{-}=(X_{t-}: t \geq 0)$ is predictable and locally bounded in $\Phi$. 
\end{proposition}
\begin{proof}
The result is a direct consequence of Proposition \ref{propLeftLimiLocallyBounDualNuclear} using that by reflexivity $\Phi$ is the strong dual of $\Phi'$. 
\end{proof}

\begin{corollary}\label{coroLeftLimitProcessLocallyBoundedSouslin}
Let $\Phi$ be a reflexive Souslin space whose dual $\Phi'$ is ultrabornological and nuclear. If $X=(X_{t}: t \geq 0)$ is a  $\Phi$-valued c\`{a}dl\`{a}g weakly adapted process, then  $X_{-}=(X_{t-}: t \geq 0)$ is predictable and locally bounded. 
\end{corollary}
\begin{proof}
First, since $\Phi$ is Souslin, the probability distribution of each $X_{t}$ is a Radon measure on $\Phi$ (see \cite{BogachevMT}, Theorem 7.4.3, p.85)  and hence by Theorem 2.10 in \cite{FonsecaMora:2018} the mapping $X_{t}: \Phi' \rightarrow L^{0} \ProbSpace$ is continuous and the process $X$ is regular. As $\Phi'$ is an ultrabornological nuclear space, then it follows from Proposition 3.10 in \cite{FonsecaMora:2018}  that $\forall \, T>0$ the family $(X_{t}: t \in [0,T])$ of linear mappings from $\Phi'$ into $L^{0} \ProbSpace$ is equicontinuous. Then Proposition  \ref{propLeftLimitProcessLocallyBounded} shows that $X_{-}=(X_{t-}: t \geq 0)$ is a predictable and locally bounded process with values in $\Phi$.  
\end{proof}

\begin{example}\label{examLeftLimiProcesSpaceDistribu}
The spaces $\mathscr{E}_{K}$, $\mathscr{E}$, $\mathscr{S}(\R^{d})$, $\mathscr{D}(U)$, $\mathscr{E}'_{K}$, $\mathscr{E}'$, $\mathscr{S}'(\R^{d})$, and $\mathscr{D}'(U)$ are all reflexive Souslin spaces (see \cite{SchwartzRM}, p.115) which are ultrabornological and nuclear (references in Section \ref{sectNuclearSpaceCylSemi}). Therefore, if $X=(X_{t}: t \geq 0)$ is a c\`{a}dl\`{a}g weakly adapted process taking values in any of these spaces, then by Corollary \ref{coroLeftLimitProcessLocallyBoundedSouslin} the process of its left-limits $X_{-}=(X_{t-}: t \geq 0)$ is  predictable and locally bounded.
\end{example}

We now formally introduce our main class of integrands. 

\begin{definition}\label{defiLocalBoundWeakInteg}
Let $\Phi$ be locally convex. We denote by $\mathcal{P}_{loc}(\Phi)$ the space (of all equivalence classes) of mappings $H:\R_{+} \times \Omega \rightarrow \Phi$ that are weakly predictable and locally  bounded. 
\end{definition}

If $\Phi$ is a complete barrelled nuclear space, it should be clear that $b \mathcal{P}(\Phi) \subseteq \mathcal{P}_{loc}(\Phi)$ since for each $H \in   b \mathcal{P}(\Phi)$ the image of $H$ is contained in a bounded subset of $\Phi$ (see the proof of Proposition \ref{propIsomorBounPredSpaceLinOper}). 

\begin{example}
Let $Z=(Z_{t}: t \geq 0)$ be a real-valued semimartingale. 
Given $\phi \in \mathscr{D}(\R)$, define $X=(X_{t}: t \geq 0)$ by $X_{t}(\omega)=\phi(\cdot + Z_{t}(\omega))$. Observe that for every $t \geq 0$, $\omega \in \Omega$, it is a consequence of the It\^{o} formula that $X_{t}(\omega) \in \mathscr{D}(\R)$. 

Moreover, the continuity of the translation mapping $\varphi \mapsto \varphi(\cdot +y)$ in $\mathscr{D}(\R)$ for every $y \in \R$ shows that $X$ is a $\mathscr{D}(\R)$-valued c\`{a}dl\`{a}g adapted process. Furthermore, it is shown in (\cite{Ustunel:1982-1}, Theorem III.2) that for every $f \in \mathscr{D}'(\R)$ the real-valued process $\inner{f}{X}$ is a semimartingale. Therefore, $X$ is a $\mathscr{D}(\R)$-valued c\`{a}dl\`{a}g semimartingale. Then, the conclusion of Example \ref{examLeftLimiProcesSpaceDistribu} shows that 
$X_{-}=(X_{t-}: t \geq 0) \in \mathcal{P}_{loc}(\mathscr{D}(\R))$. In particular, if $Z$ has continuous trajectories then $X$ has continuous trajectories  in $\mathscr{D}(\R)$ and in this case $X=(X_{t}: t \geq 0) \in \mathcal{P}_{loc}(\mathscr{D}(\R))$.   
\end{example}

The stochastic integral for integrands in $\mathcal{P}_{loc}(\Phi)$ is defined in the next theorem. 

\begin{theorem}\label{theoLocalBoundWeakIntegNucleSpace}
Let $\Phi$ be a complete, barrelled, nuclear space and let $X$ be a cylindrical semimartingale satisfying any of the equivalent conditions in Proposition \ref{propCylSemimContOpeSpaceSemim}. Then for every  
$H \in \mathcal{P}_{loc}(\Phi)$ there exists a real-valued semimartingale $\displaystyle{\int H \, dX}$ such that:  
\begin{enumerate}
\item $\displaystyle{\left(\int H \, dX \right)^{c}= \int H \, dX^{c}}$. 
\item $\displaystyle{\left(\int H \, dX \right)^{\tau}=\int H \mathbbm{1}_{[0,\tau]} \, dX= \int H \, dX^{\tau}}$, for every stopping time $\tau$. 
\item The mapping $(H,X) \mapsto \int \, H \, dX$ is bilinear.   
\end{enumerate} 
\end{theorem}
\begin{proof} 
Let $H \in \mathcal{P}_{loc}(\Phi)$. Then, there exists a sequence $( \tau_{n} : n \in \N)$ of stopping times, increasing to infinity $\Prob$-a.e., such that for each $n \in \N$ we have $H \mathbbm{1}_{[0,\tau_{n}]} \in b\mathcal{P}(\Phi)$. Hence by Theorem \ref{theoBoundedWeakIntegNucleSpace}  each $\displaystyle{\int \, H \mathbbm{1}_{[0,\tau_{n}]}  \, dX}$ is a well-defined real-valued semimartingale. 

Then for each $t \geq 0$ we can define 
$$  \int_{0}^{t} H \, dX = \int_{0}^{t} H \mathbbm{1}_{[0,\tau_{n}]} \, dX, $$ 
for any $n \in \N$ such that $\tau_{n} \geq t$. A standard localization argument (e.g. see Chapter 4 in \cite{DaPratoZabczyk}) using property \emph{(2)} in Theorem  \ref{theoBoundedWeakIntegNucleSpace} shows that this definition for $\int H \, dX$ is consistent and that it is independent (up to indistinguishable versions) of the localizing sequence for $H$. Since the property of being a real-valued semimartingale is stable by localization, we have that $\int H \, dX$ is a real-valued semimartingale. Likewise, the fact that properties \emph{(1)}, \emph{(2)} and \emph{(3)} are satisfied  follows from Theorem \ref{theoBoundedWeakIntegNucleSpace} by choosing an appropriate localizing sequence.  
\end{proof}

\begin{example}\label{examSemimAsIntegrators}
Let $\Phi$ be a complete, barrelled, nuclear space whose strong dual is nuclear and let $Y=(Y_{t}: t \geq 0)$ be a $\Phi$-valued, regular, c\`{a}dl\`{a}g semimartingale. By its definition it is clear that $Y$ is weakly adapted. If $Y$ satisfies (as a cylindrical process) any of the equivalent conditions in Proposition \ref{propCylSemimContOpeSpaceSemim}, then it follows from Proposition  \ref{propLeftLimitProcessLocallyBounded}  that $Y_{-}=(Y_{t-}: t \geq 0) \in \mathcal{P}_{loc}$. 

If $X$ is a cylindrical semimartingale in $\Phi'$ satisfying any of the equivalent conditions in Proposition \ref{propCylSemimContOpeSpaceSemim}, then by Theorem \ref{theoLocalBoundWeakIntegNucleSpace} the stochastic integral $\displaystyle{\int \, Y_{-} \, dX}$ exists and it is a real-valued semimartingale. 
\end{example}

\begin{remark}\label{remaUstunelStochInteg} 
In \cite{Ustunel:1982} and under the assumption that $\Phi$ is a complete, bornological, reflexive nuclear space whose strong dual $\Phi'$ is complete and nuclear, \"{U}st\"{u}nel introduced the concept of projective system of  semimartingales in $\Phi'$ (see Remark \ref{remaSemiDefiUstunel}) and defined stochastic integrals for $\Phi$-valued weakly predictable process which are locally  bounded. The construction of the stochastic integral in \cite{Ustunel:1982} is very different from ours and relies on the theory of stochastic integration in separable Hilbert spaces and in the concept of projective system of  semimartingales in $\Phi'$. As mentioned in Remark \ref{remaSemiDefiUstunel}, the concept of projective system only makes sense if $\Phi'$ is complete and nuclear. Observe however that for our construction of the stochastic integral we have the less demanding assumption on  $\Phi$ to be a complete, barrelled, nuclear space and that we have assumed nothing on the dual space $\Phi'$. 
\end{remark}

In the next result we list some properties of the stochastic integral for integrands in $\mathcal{P}_{loc}(\Phi)$ with respect to a good integrator. These results follow directly from Theorem \ref{theoLocalBoundWeakIntegNucleSpace}, and Propositions \ref{propJumpStochIntegraBoundPred} and \ref{propAssociativeWeakIntegralNuclear}.

\begin{proposition}\label{propProperStochIntegLocalBounded}
Let $\Phi$ be a complete, barrelled, nuclear space. Suppose that $X$ is a good integrator in $\Phi'$ and let $H \in \mathcal{P}_{loc}(\Phi)$. 
\begin{enumerate}
\item Let $\widetilde{X}=(\widetilde{X}_{t}: t \geq 0)$ denotes the $\Phi'$-valued, regular, c\`{a}dl\`{a}g version of $X$ as in Theorem \ref{theoRegulCylinSemimartigalesNuclear}. Then, 
$$ \Delta \left( \int_{0}^{t} H \, dX \right)= \inner{ \Delta \widetilde{X}_{t}}{H(t)}, \quad \forall \, t \geq 0.$$

\item For every $g \in \mathcal{P}_{loc}(\R)$, 
\begin{equation*}
\int \, g \, dZ(H) = \int \, g H \, dX, 
\end{equation*}
where $Z(H)= \int \, H \, dX$ and $(g H)_{t}(\omega)=g_{t}(\omega)H_{t}(\omega)$ for all $t \geq 0$, $\omega \in \Omega$. 
\end{enumerate}
\end{proposition}


\subsection{Riemann Representation and Integration by Parts} \label{subSectRiemanRepre}

In this section we study further continuity properties of the stochastic integral with respect to good integrators and for integrands which have c\`{a}gl\`{a}d paths. In particular, we will show that the stochastic integral satisfies a Riemman representation and a integration by parts formula.  

Let $\Phi$ be a complete locally convex space. Denote by $\mathbbm{L}(\Phi)$ the linear space of all the $\Phi$-valued c\`{a}gl\`{a}d weakly adapted processes. We will introduce on $\mathbbm{L}(\Phi)$ the \emph{topology of convergence in probability uniformly on compact intervals of time} (abbreviated as \emph{ucp}) by means of a family of F-seminorms (see Section 2.7 in \cite{Jarchow} for more details on topologies defined by F-seminorms). 

Let $\Pi$ denotes a system of seminorms generating the topology on $\Phi$. For each $p \in \Pi$, we define a $F$-seminorm on 
$\mathbbm{L}(\Phi)$ by the prescription
$$ d^{p}_{ucp}(H)=\sum_{n=1}^{\infty} 2^{-n} \Exp \left(1 \wedge \sup_{0 \leq t \leq n} p(H_{t}) \right), \quad \forall H=(H_{t}: t \geq 0) \in   \mathbbm{L}(\Phi).$$
The collection of all the $F$-seminorms $(d^{p}_{ucp}: p \in \Pi)$ generates a linear topology on $\mathbbm{L}(\Phi)$  wherein a  fundamental system of neighborhoods of zero is the family of all the sets of the form 
$$\left\{ H \in \mathbbm{L}(\Phi) : \Prob \left( \sup_{0 \leq t \leq T} p(H_{t}) > \eta \right) < \epsilon \right\}, \quad \,  \epsilon, \eta, T>0, \, p \in \Pi. $$
Now, observe that $\Phi$, being complete, coincides with the projective limit $\mbox{proj}_{p \in \Pi} (\Phi_{p}, i_{p})$ and that  each $\mathbbm{L}(\Phi_{p})$ is a complete metrizable topological vector space (this because each $\Phi_{p}$ is a separable Banach space; see e.g. Section II.1.3 in \cite{VakhaniaTarieladzeChobanyan}). Then, by comparing the fundamental system of neighborhoods of zero in  $\mathbbm{L}(\Phi)$ and the corresponding fundamental system of neighborhoods of zero in each $\mathbbm{L}(\Phi_{p})$ one can conclude (see Proposition 2.6.1 in \cite{Jarchow}, p.38) that  $(\mathbbm{L}(\Phi),ucp)$ coincides with the projective limit $\mbox{proj}_{p \in \Pi} (\mathbbm{L}(\Phi_{p}), i_{p})$, where $i_{p}$ induces the inclusion $H \mapsto i_{p} H$ from  $\mathbbm{L}(\Phi)$ into $\mathbbm{L}(\Phi_{p})$, and if $p \leq q$, $i_{p,q}$ induces the inclusion $H \mapsto i_{p,q} H$ from  $\mathbbm{L}(\Phi_{q})$ into $\mathbbm{L}(\Phi_{p})$. We obtain two important conclusions from the fact that $(\mathbbm{L}(\Phi),ucp)$ is the above projective system. First, that it is complete (see Corollary 3.2.7 in \cite{Jarchow}, p.59) and second that it  is metrizable whenever $\Phi$ is so (see Proposition 2.8.3 in \cite{Jarchow}, p.41). 
   

Denote by $b\mathbbm{L}(\Phi)$ the subspace of $\mathbbm{L}(\Phi)$ consisting of those processes whose image is contained in a bounded subset in $\Phi$. Every $H \in b\mathbbm{L}(\Phi)$ is weakly predictable and hence 
$b\mathbbm{L}(\Phi) \subseteq b \mathcal{P}(\Phi)$. Then, if $\Phi$ is a complete, barrelled, nuclear space the stochastic integral is defined for every $H \in b\mathbbm{L}(\Phi)$. More generally, if $H \in \mathbbm{L}(\Phi) \cap \mathcal{P}_{loc}(\Phi)$ it follows from 
Theorem \ref{theoLocalBoundWeakIntegNucleSpace} that the stochastic integral of such an $H$ exists. Example \ref{examSemimAsIntegrators} provides us with examples of processes which are in $\mathbbm{L}(\Phi) \cap \mathcal{P}_{loc}(\Phi)$ but not in $b\mathbbm{L}(\Phi)$. 

The next result shows that when $X$ is a good integrator the stochastic integral mapping $H \mapsto \int \, H \, dX$ defines a sequentially continuous operator from $(\mathbbm{L}(\Phi) \cap \mathcal{P}_{loc}(\Phi), ucp)$ into $(S^{0}, ucp)$. 

\begin{theorem}\label{theoContStocIntegUCP} Let $\Phi$ be a complete, barrelled, nuclear space. Suppose that $X=(X_{t}: t \geq 0)$ is a good integrator in $\Phi'$.  
Let $H$, $H^{n} \in \mathbbm{L}(\Phi) \cap \mathcal{P}_{loc}(\Phi)$, $n \geq 1$, and suppose that $H^{n} \overset{ucp}{\rightarrow} H$. Then, 
$\displaystyle{\int \, H^{n} \, dX \overset{ucp}{\rightarrow} \int \, H^{n} \, dX}$.  
\end{theorem}
\begin{proof} 
By linearity of the stochastic integral it suffices to prove that $H^{n} \overset{ucp}{\rightarrow} 0$. Our proof is a modification of the arguments used in the proof of Theorem II.11 in \cite{Protter} for the finite dimensional case. 

First, as $X$ is a good integrator the stochastic integral mapping $H \mapsto \int \, H \, dX$ is linear and continuous from $b\mathcal{P}(\Phi)$ into $S^{0}$, hence into $S^{0}$ equipped with the ucp topology. Then, given $\epsilon, \delta, T >0$, there exists a continuous seminorm $p$ on $\Phi$ such that $ \displaystyle{\sup_{(t,\omega)} p(H_{t}(\omega)) \leq 1}$ implies 
$$\Prob \left( \sup_{0 \leq t \leq T} \abs{\int_{0}^{t} \, H \, dX} \geq \delta \right) < \frac{\epsilon}{2}. $$
For each $n \geq 1$, let $\tau_{n}(\omega)=\inf\{ t \geq 0: p(H^{n}_{t}(\omega))>1 \}$, which is a stopping time. Define $\widetilde{H}^{n}= H^{n} \mathbbm{1}_{[0,\tau_{n}]}\mathbbm{1}_{\{\tau_{n}>0\}}$. Then  $\widetilde{H}^{n} \in b\mathbbm{L}(\Phi)$ and $ \displaystyle{\sup_{(t,\omega)} p(\widetilde{H}^{n}_{t}(\omega)) \leq 1}$.  Since $\tau_{k} \geq T$ implies
$$ \sup_{0 \leq t \leq T} \abs{\int_{0}^{t} \, \widetilde{H}^{n} \, dX} = \sup_{0 \leq t \leq T} \abs{\int_{0}^{t} \, H^{n} \, dX}, $$
then we have 
\begin{eqnarray*}
 \Prob \left( \sup_{0 \leq t \leq T} \abs{\int_{0}^{t} \, H^{n} \, dX} \geq \delta \right)
& = &   \Prob \left( \sup_{0 \leq t \leq T} \abs{\int_{0}^{t} \, H^{n} \, dX} \geq \delta , \,  \tau_{k} \geq T \right) \\
& {} &  + \Prob \left( \sup_{0 \leq t \leq T} \abs{\int_{0}^{t} \, H^{n} \, dX} \geq \delta , \,  \tau_{k} < T \right) \\
& \leq & \Prob \left( \sup_{0 \leq t \leq T} \abs{\int_{0}^{t} \, \widetilde{H}^{n} \, dX} \geq \delta \right)
+ \Prob \left( \tau_{k} <T \right)\\
& \leq &\frac{\epsilon}{2}  + \Prob \left(  \sup_{0 \leq t \leq T} p(H^{n}_{t})>1 \right). 
\end{eqnarray*}
Finally, since $H^{n} \overset{ucp}{\rightarrow} 0$, there exists $N>0$ (depending on $\epsilon, \delta, T$ and $p$) such that for all $n \geq N$ we have
$$ \Prob \left(  \sup_{0 \leq t \leq T} p(H^{n}_{t})>1 \right)<  \frac{\epsilon}{2}, $$
hence for $n \geq N$ we have 
$$ \Prob \left( \sup_{0 \leq t \leq T} \abs{\int_{0}^{t} \, H^{k} \, dX} \geq \delta \right) < \epsilon. $$
Thus, $\displaystyle{\int \, H^{n} \, dX \overset{ucp}{\rightarrow} 0}$. 
 \end{proof}

Our next objective is to prove that our stochastic integral satisfies a Riemann representation. In order to formulate and prove our result we will need the following terminology from Section II.2 in \cite{Protter} which we adapt to our context. 

\begin{definition} \hfill
\begin{enumerate}
\item A \emph{random partition} $\sigma$ is a finite sequence of finite stopping times:
$$ 0=\tau_{0} \leq \tau_{1} \leq \cdots \leq \tau_{m+1} < \infty. $$
\item Given a $\Phi$-valued process $H$ and a random partition $\sigma$, we define the process $H$ \emph{sampled} at $\sigma$ to be
\begin{equation}\label{eqDefiSampledProcess}
 H^{\sigma} \defeq H_{0} \mathbbm{1}_{\{0\}}+\sum_{k=1}^{m} H_{\tau_{k}} \mathbbm{1}_{(\tau_{k},\tau_{k+1}]}.
\end{equation}
\item A sequence of random partitions $(\sigma_{n})$, 
$$ \sigma_{n}: \tau_{0}^{n} \leq \tau_{1}^{n} \leq \dots \leq \tau^{n}_{m_{n}+1}, $$
is said to \emph{tend to the identity} if 
\begin{enumerate}
\item $\displaystyle{\lim_{n} \sup_{k} \tau^{n}_{k}=\infty}$ a.s., and
\item $\norm{\sigma_{n}}=\sup_{k} \abs{\tau_{k+1}^{n}-\tau_{k}^{n}} \rightarrow 0$ a.s.
\end{enumerate}  
\end{enumerate}
\end{definition}

We are ready for our main result of this section.  

\begin{theorem}[Riemann representation]\label{theoRiemannRepresentation}
Suppose that $\Phi$ is a complete, barrelled nuclear space whose strong dual $\Phi'$ is also nuclear. Let $X$ be $\Phi'$-valued semimartingale which is a good integrator.
Let $H \in \mathbbm{L}(\Phi) \cap \mathcal{P}_{loc}(\Phi)$. Let $(\sigma_{n})$ be a sequence of random partitions tending to the identity. Then, 
\begin{equation}\label{eqIntegOfStoppedProcess}
\int_{0}^{t} \, H^{\sigma_{n}} \, dX = \inner{X_{0}}{H_{0}} +\sum_{k=1}^{m_{n}} \, \inner{X_{\tau^{n}_{k+1} \wedge t}-X_{\tau^{n}_{k} \wedge t}}{H_{\tau^{n}_{k}}},   
\end{equation}
and 
\begin{equation}\label{eqUCPConvergRiemannRepre}
\int_{0}^{t} \, H^{\sigma_{n}} \, dX \overset{ucp}{\rightarrow} \int \, H_{-} \, dX.
\end{equation}
\end{theorem}

In order to prove Theorem \ref{theoRiemannRepresentation} we will need to go over several steps. First, we will show 
that \eqref{eqIntegOfStoppedProcess} is satisfied. We will prove indeed something more general, said that our stochastic integral acts on simple predictable integrands in a similar way as the stochastic integral for semimartingales in the scalar case. This is carried out below in Lemma \ref{lemmActionStoInteElemBounInteg}. It is important to stress that this result is not evident since unlike the the theory of integration in finite dimensions, we have not defined our integral first for simple predictable integrands and hence there is no clear indication that it should be of this form. 

Denote by $\mathcal{C}(\Phi)$ the collection of all the $\Phi$-valued processes of the form:
\begin{equation}\label{eqDefiSimpleBoundIntegrNuclear}
H_{r}(\omega)=A_{0}(\omega)\mathbbm{1}_{\{0\}}(r)+ \sum_{k=1}^{N} A_{k}(\omega) \mathbbm{1}_{(\tau_{k},\tau_{k+1}]}(r), 
\end{equation}
where $0=\tau_{0} \leq \tau_{1} \leq \tau_{2} \leq \dots \leq \tau_{N+1}$ are bounded stopping times  and $A_{k}$ is a $\Phi$-valued weakly bounded $\mathcal{F}_{\tau_{k}}$-measurable random variable, $0 \leq k \leq N$, $N \geq 1$. The elements in $\mathcal{C}(\Phi)$ are called \emph{simple predictable bounded integrands}. 

It is clear that $\mathcal{C}(\Phi) \subseteq b \mathbbm{L}$ and hence the stochastic integral exists to every $H \in \mathcal{C}(\Phi)$. Moreover, we have the following:

\begin{lemma}\label{lemmActionStoInteElemBounInteg} Suppose that $\Phi$ is a complete, barrelled nuclear space whose strong dual $\Phi'$ is also nuclear. Let $X$ be $\Phi'$-valued semimartingale which is a good integrator. If $H \in \mathcal{C}(\Phi)$ is of the form \eqref{eqDefiSimpleBoundIntegrNuclear}, then for every $t \geq 0$, 
\begin{equation}\label{eqActionWeakIntSimpBoundIntegrNuclear}
\int_{0}^{t} \, H \, dX = \inner{X_{0}}{A_{0}}+\sum_{k=1}^{N} \inner{X_{\tau_{k+1} \wedge t}- X_{\tau_{k} \wedge t}}{A_{k}}. 
\end{equation}
\end{lemma}
\begin{proof}
We carry out the proof in two steps. 

\textbf{Step 1} Suppose that for each $k=0,1, \dots, N$, the random variable $A_{k}$ is of the simple form $\displaystyle{A_{k}= \sum_{j=0}^{m_{k}} h_{j,k} \phi_{j,k}}$ where each $h_{j,k}$ is a real-valued $\mathcal{F}_{\tau_{k}}$-measurable random variable and each $\phi_{j,k} \in \Phi$.  
In such a case, by the action of the stochastic integral on the elementary integrands \eqref{eqActionWeakIntegSimpleIntegNuclear} we have
\begin{eqnarray*}
\int_{0}^{t} \, H \, dX 
& = & \int_{0}^{t} \, \sum_{j=0}^{m_{0}} \, h_{j,0} \mathbbm{1}_{\{0\}} \phi_{j,0} \, dX  + \sum_{k=1}^{N}\int_{0}^{t} \, \sum_{j=0}^{m_{k}} h_{j,k} \mathbbm{1}_{(\tau_{k},\tau_{k+1}]} \phi_{j,k} \, dX \\ 
& = &  \sum_{j=0}^{m_{0}} \, \left( \, ( h_{j,0} \mathbbm{1}_{\{0\}} ) \cdot X(\phi_{j,0}) \right)_{t} +\sum_{k=1}^{N} \sum_{j=0}^{m_{k}} \, \left( \, ( h_{j,k} \mathbbm{1}_{(\tau_{k},\tau_{k+1}]} ) \cdot X(\phi_{j,k}) \right)_{t} \\ 
& = & \sum_{j=0}^{m_{0}} \, h_{j,0} \inner{ X_{0}}{\phi_{j,0}}  + \sum_{k=1}^{N} \sum_{j=0}^{m_{k}} \, h_{j,k} \inner{ X_{\tau_{k+1} \wedge t} - X_{\tau_{k} \wedge t}}{\phi_{j,k}}  \\ 
& = &   \inner{ X_{0}}{\sum_{j=0}^{m_{0}} \, h_{j,0} \phi_{j,0}}+ \sum_{k=1}^{N}  \inner{ X_{\tau_{k+1} \wedge t} - X_{\tau_{k} \wedge t}}{ \sum_{j=0}^{m_{k}} \, h_{j,k} \phi_{j,k}}  \\
& = & \inner{X_{0}}{A_{0}} +\sum_{k=1}^{N} \inner{X_{\tau_{k+1} \wedge t}- X_{\tau_{k} \wedge t}}{A_{k}}. 
\end{eqnarray*}
Hence \eqref{eqActionWeakIntSimpBoundIntegrNuclear} holds in this case. 

\textbf{Step 2} Assume that each $A_{k}$ is a $\Phi$-valued weakly bounded $\mathcal{F}_{\tau_{k}}$-measurable random variable. 

Let $0 \leq k \leq N$. Then as in the proof of Proposition \ref{propIsomorBounPredSpaceLinOper} our assumptions on $A_{k}$ imply that the image of the random variable $A_{k}$ is contained in a bounded subset of $\Phi$. By taking the union of these finite family of bounded subsets, we can assume that the image of all the $A_{k}$ is contained in the same bounded subset of $\Phi$. Furthermore, since $\Phi$ is complete, barrelled and nuclear, there exists a compact, convex, balanced subset $K$ of $\Phi$ such that the image of each random variable $A_{k}$ is contained in $K$ (see Corollary 50.2.1 in \cite{Treves}, p.520).  

Denote by $\Phi[K]$ the vector space on $\Phi$ spanned by $K$. It is a Banach space when equipped with the gauge norm $q_{K}(\phi)=\inf\{\lambda> 0:  \phi \in \lambda K\}$ (see Corollary 36.1 in \cite{Treves}, p.371). Since $\Phi'$ is nuclear and $\Phi$ is reflexive, there exists another such set $B$, $K \subseteq B$, with the property that the canonical injection $\Phi[K] \rightarrow \Phi[B]$ is nuclear. Observe that we therefore have that the image of the set $K$ in the space $\Phi[B]$ is compact and that $\Phi[B]$ is separable. 

The information provided in the above paragraphs shows that each $A_{k}$ can be regarded as a $\Phi[B]$-valued $\mathcal{F}_{\tau_{k}}$-measurable random variable with a relatively compact image. It then follows that each $A_{k}$ can be uniformly approximated by a sequence of simple random variables $\displaystyle{A_{n,k}= \sum_{j=0}^{m_{n,k}} h_{n,j,k} \phi_{n,j,k}}$ as those in Step 1 (see Proposition I.1.9 in \cite{VakhaniaTarieladzeChobanyan}, p.12). Then, each 
$$ H_{n}(r,\omega)=A_{n,0}(\omega)\mathbbm{1}_{\{0\}}(r)+ \sum_{k=1}^{N} A_{n,k}(\omega) \mathbbm{1}_{(\tau_{k},\tau_{k+1}]}(r), $$
is an element in $\mathcal{C}(\Phi)$ and the sequence $(H_{n}: n \in \N)$ converges to $H$ in $b\mathcal{P}(\Phi)$. To prove this assertion, let $p$ be any continuous seminorm on $\Phi$. Since the canonical inclusion from $\Phi[B]$ into $\Phi$ is linear and continuous, there exists some $C>0$ (depending on $p$) such that $p(\phi) \leq C \, q_{K}(\phi)$ for every $\phi \in \Phi[B]$.  Then, by the uniform convergence of the sequence $A_{n,k}$ to $A_{k}$ for $0 \leq k \leq N$ we have
\begin{flalign*}
& \sup_{(r,\omega)} p\left( H(r,\omega)- H_{n}(r,\omega) \right) \\
& \leq  \sup_{(r,\omega)}  p\left(A_{0}(\omega) - A_{n,0}(\omega) \right) \mathbbm{1}_{\{0 \}} + \sum_{k=1}^{N} \sup_{(r,\omega)}  p\left(A_{k}(\omega) - A_{n,k}(\omega) \right) \mathbbm{1}_{(\tau_{k},\tau_{k+1}]}(r) \\
& \leq  C\sup_{\omega}  q_{K} \left(A_{0}(\omega) - A_{n,0}(\omega) \right) + C \sum_{k=0}^{N} \sup_{w} q_{K} \left(A_{k}(\omega) - A_{n,k}(\omega) \right) \\
& \rightarrow 0, \mbox{ as } n \rightarrow \infty. 
\end{flalign*}
Now, because $X$ is a good integrator we have that 
$\displaystyle{\int \, H_{n} \, dX \rightarrow \int \, H \, dX}$ in $S^{0}$. Hence, for each $t \geq 0$ we have $\displaystyle{\int_{0}^{t} \, H_{n} \, dX \rightarrow \int_{0}^{t} \, H \, dX}$ in probability.   

On the other hand, for each $t \geq 0$, by the uniform convergence of the sequence $A_{n,k}$ to $A_{k}$ for $0 \leq k \leq N$, and because the canonical inclusion from $\Phi[B]$ into $\Phi$ is linear and continuous, we have pointwise 
$$ \inner{X_{0}}{A_{n,0}}+ \sum_{k=1}^{N} \inner{X_{\tau_{k+1} \wedge t}- X_{\tau_{k} \wedge t}}{A_{n,k}} \rightarrow \inner{X_{0}}{A_{0}}+ \sum_{k=1}^{N} \inner{X_{\tau_{k+1} \wedge t}- X_{\tau_{k} \wedge t}}{A_{k}}, $$
and since by Step 1 we have that for each $t \geq 0$ 
$$ \inner{X_{0}}{A_{n,0}}+ \sum_{k=1}^{N} \inner{X_{\tau_{k+1} \wedge t}- X_{\tau_{k} \wedge t}}{A_{n,k}} = \int_{0}^{t} \, H_{n} \, dX, $$
we conclude by uniqueness of limits in probability that 
\eqref{eqActionWeakIntSimpBoundIntegrNuclear} is satisfied. 
\end{proof}

Our next objective is to show that the convergence in \eqref{eqUCPConvergRiemannRepre} holds. To shows this we will need the following approximation result:

\begin{lemma}\label{lemmApproByElementaryProcess}
Suppose that $\Phi$ is a complete, barrelled nuclear space whose strong dual $\Phi'$ is also nuclear. 
Let $H \in \mathbbm{L}(\Phi) \cap \mathcal{P}_{loc}(\Phi)$. Then, there exists a sequence $(H^{n}: n \in \N) \subseteq \mathcal{C}(\Phi)$ such that $H^{n} \overset{ucp}{\rightarrow} H$. 
\end{lemma}
\begin{proof}
Since $H \in \mathbbm{L}(\Phi) \cap \mathcal{P}_{loc}(\Phi)$, there exists a  sequence $( \tau_{n} : n \in \N)$ of stopping times increasing to $\infty$ $\Prob$-a.e. such that for each $n \in \N$, $H^{\tau_{n}} \in b\mathbbm{L}(\Phi)$. This means we can assume $H \in b\mathbbm{L}(\Phi)$.

Now, since $H \in b\mathbbm{L}(\Phi)$, the image of $H$ is contained in a bounded subset in $\Phi$. As in the proof of Lemma \ref{lemmActionStoInteElemBounInteg}, because $\Phi$ is complete, barrelled and nuclear, and because $\Phi'$ is nuclear and reflexive, there exists a compact, convex, balanced subset $B$ of $\Phi$ such that $\Phi[B]$ is a separable Banach space with norm $q_{B}$ and $H$ can be considered as a c\`{a}gl\`{a}d adapted process with values in $\Phi[B]$, i.e. $H \in \mathbbm{L}(\Phi[B])$. 
Let $Z$ be defined by $Z_{t}=\lim_{u \rightarrow t, u>t} H_{t}$. We therefore have that $Z$ is a c\`{a}dl\`{a}g adapted process with values in $\Phi[B]$. 
As in the proof of Theorem II.10 in \cite{Protter}, given $\epsilon >0$, if we define $\tau_{0}^{\epsilon}=0$ and $\displaystyle{\tau^{\epsilon}_{n+1}= \inf\{ t: t> \tau^{\epsilon}_{n} \mbox{ and } q_{B}(Z_{t}-Z_{\tau^{\epsilon}_{n}}) > \epsilon \}}$, then one can show (this as in Theorem II.10 in \cite{Protter}) that $(\tau^{\epsilon}_{n})$ is a sequence of stopping times increasing to $\infty$ $\Prob$-a.e. and that 
$$ H^{n,\epsilon}=H_{0} \mathbbm{1}_{\{0\}} + \sum_{k=1}^{n} Z_{\tau^{\epsilon}_{k}} \mathbbm{1}_{(\tau^{\epsilon}_{k} \wedge n, \tau^{\epsilon}_{k+1} \wedge n]}, $$  
can be made arbitrarily close to $H$ in ucp in the norm $q_{B}$ by making $\epsilon$ small enough and $n$ large enough. Finally, since the topology in $(\Phi[B],q_{B})$ is finer than that in $\Phi$, convergence in ucp in the norm $q_{B}$ implies convergence in ucp in $\Phi$; this as in the proof of Lemma \ref{lemmActionStoInteElemBounInteg}.   
\end{proof}

\begin{proof}[Proof of Theorem \ref{theoRiemannRepresentation}]
We begin by proving the equality in  \eqref{eqIntegOfStoppedProcess}. 
Let $(u_{j}: j \in \N)$ a sequence of stopping times increasing to $\infty$ $\Prob$-a.e. such that for each $j \in \N$ the process $H^{u_{j}}$ is bounded in $\Phi$.

For any fixed $n \in \N$, denote by $\sigma_{n} \wedge u_{j}$ the random partition 
$$ \sigma_{n} \wedge u_{j}: \tau_{0}^{n} \wedge u_{j} \leq \tau_{1}^{n} \wedge u_{j} \leq \dots \leq \tau_{m_{n}}^{n} \wedge u_{j}. $$
It is clear from \eqref{eqDefiSampledProcess} that for every $n \in \N$ we have $H^{\sigma_{n} \wedge u_{j}} \in \mathcal{C}(\Phi)$ and hence by Lemma \ref{lemmActionStoInteElemBounInteg} it follows that 
$$ \int_{0}^{t} \, H^{\sigma_{n} \wedge u_{j}} \, dX = \inner{X_{0}}{H_{0}} +\sum_{k=1}^{m_{n}} \, \inner{X_{\tau^{n}_{k+1}\wedge u_{j}  \wedge t}-X_{\tau^{n}_{k} \wedge u_{j} \wedge t}}{H_{\tau^{n}_{k} \wedge u_{j}}}. $$ 
Since $u_{j}$ increases to $\infty$ $\Prob$-a.e then $H^{\sigma_{n} \wedge u_{j}} \overset{ucp}{\rightarrow} H^{\sigma_{n}}$ and hence it follows from Theorem \ref{theoContStocIntegUCP} that $\displaystyle{\int \, H^{\sigma_{n} \wedge u_{j}} \, dX \overset{ucp}{\rightarrow} \int \, H^{\sigma_{n}} \, dX}$. Likewise, because $u_{j}$ increases to $\infty$ $\Prob$-a.e, then we have pointwise $\Prob$-a.e for each $t \geq 0$, 
$$ \sum_{k=1}^{m_{n}} \, \inner{X_{\tau^{n}_{k+1}\wedge u_{j}  \wedge t}-X_{\tau^{n}_{k} \wedge u_{j} \wedge t}}{H_{\tau^{n}_{k} \wedge u_{j}}} \rightarrow \sum_{k=1}^{m_{n}} \, \inner{X_{\tau^{n}_{k+1} \wedge t}-X_{\tau^{n}_{k} \wedge t}}{H_{\tau^{n}_{k}}}.$$
Then by uniqueness of limits we obtain the equality in  \eqref{eqIntegOfStoppedProcess}. 

We now show that the convergence in \eqref{eqUCPConvergRiemannRepre} holds. To do this we will  adapt to our context the proof of Theorem II.21 in \cite{Protter}. 

First, by Lemma \ref{lemmApproByElementaryProcess} there exists a sequence $(H^{k}: k \in \N) \subseteq \mathcal{C}(\Phi)$ such that $H^{k} \overset{ucp}{\rightarrow} H$. 
We have 
\begin{eqnarray}
\int \, ( H - H^{\sigma_{n}}) \, dX & = &  \int \, ( H - H^{k})  \, dX + \int \, ( H^{k} - (H^{k})^{\sigma_{n}}) \, dX \nonumber \\ 
& {} & + \int \, ( (H^{k})^{\sigma_{n}} - H^{\sigma_{n}}) \, dX.  \label{eqProofRiemannRepre}
\end{eqnarray}
By Theorem \ref{theoContStocIntegUCP}, since $H^{k} \overset{ucp}{\rightarrow} H$ the first term in the right-hand side of \eqref{eqProofRiemannRepre} converges to 0 in ucp. In a similar way, we have that $(H^{k})^{\sigma_{n}} \overset{ucp}{\rightarrow} H^{\sigma_{n}}$ uniformly on $n$ and therefore  by Theorem \ref{theoContStocIntegUCP} the third  term in the right-hand side of \eqref{eqProofRiemannRepre} converges to 0 in ucp.

For the middle term in the right-hand side of \eqref{eqProofRiemannRepre}, since for each $k \in \N$, $H^{k} \in \mathcal{C}(\Phi)$, the one can write the stochastic integral $\displaystyle{ \int \, ( (H^{k})^{\sigma_{n}} - H^{\sigma_{n}}) \, dX}$ in closed form (as in Lemma \ref{lemmActionStoInteElemBounInteg}). Then, the right-continuity of $X$ implies that for any fixed $(k,\omega)$ the stochastic integral $\displaystyle{ \int \, ( (H^{k})^{\sigma_{n}} - H^{\sigma_{n}}) \, dX}$
 tends to $0$ as $n \rightarrow \infty$. Hence, to show that \eqref{eqProofRiemannRepre} tends to $0$ as $n \rightarrow \infty$ one chooses $k$ so large that the first and third terms in the right-hand side of \eqref{eqProofRiemannRepre} are small, and then for a fixed value of $k$ one chooses $n$ so large that middle term in the right-hand side of \eqref{eqProofRiemannRepre} is small.   
\end{proof}

Thanks to the Riemann representation formula, we can now prove that a integration by parts formula holds by following similar arguments to those in the finite dimensional setting. We must stress the fact that the two semimartingales are required to be good integrators and hence the formula might not hold for every semimartingale.

\begin{theorem}[Integration by Parts Formula]\label{theoIntegByPartsFormula}
Suppose that $\Phi$ and $\Phi'$ are complete, barrelled nuclear spaces. Let $Y$ be respectively a $\Phi$-valued c\`{a}dl\`{a}g semimartingale and $X$ be a $\Phi'$-valued c\`{a}dl\`{a}g semimartingale, both $Y$ and $X$ being good integrators. Then there exists a real-valued adapted c\`{a}dl\`{a}g process $[X,Y]$, such that 
\begin{equation}\label{eqDefIntegByParts}
\inner{X_{t}}{Y_{t}}=\inner{X_{0}}{Y_{0}}+\int_{0}^{t} \, Y_{-} \, dX + \int_{0}^{t} \, X_{-} \, dY + [X,Y]_{t}, \quad \forall t \geq 0.
\end{equation} 
Furthermore, if $(\sigma_{n})$ is a sequence of random partitions tending to the identity, then 
\begin{equation}\label{eqDefQuadraCovaria}
[X,Y]_{t}= \inner{X_{0}}{Y_{0}}+ \lim_{n \rightarrow \infty} \sum_{k=1}^{m_{n}} \, \inner{X_{\tau^{n}_{k+1} \wedge t}-X_{\tau^{n}_{k} \wedge t}}{Y_{\tau^{n}_{k+1} \wedge t}-Y_{\tau^{n}_{k} \wedge t}}, 
\end{equation}
where the convergence is in ucp. 
\end{theorem}
\begin{proof}
Observe first that from Propositions \ref{propCylSemimContOpeSpaceSemim} and \ref{propLeftLimitProcessLocallyBounded} that the integrals $\displaystyle{\int_{0}^{t} \, Y_{-} \, dX}$ and $\displaystyle{ \int_{0}^{t} \, X_{-} \, dY }$ are well-defined. 
Now, note that for every $n \in \N$ we have 
\begin{eqnarray}
\inner{X_{t}}{Y_{t}} & = & \inner{X_{0}}{Y_{0}} + 
\sum_{k=1}^{m_{n}} \, \inner{X_{\tau^{n}_{k+1} \wedge t}-X_{\tau^{n}_{k} \wedge t}}{Y_{\tau^{n}_{k} \wedge t}} \nonumber \\
& {} & +  \sum_{k=1}^{m_{n}} \, \inner{X_{\tau^{n}_{k} \wedge t}}{Y_{\tau^{n}_{k+1} \wedge t}-Y_{\tau^{n}_{k} \wedge t}} \nonumber \\
& {} & + \sum_{k=1}^{m_{n}} \, \inner{X_{\tau^{n}_{k+1} \wedge t}-X_{\tau^{n}_{k} \wedge t}}{Y_{\tau^{n}_{k+1} \wedge t}-Y_{\tau^{n}_{k} \wedge t}} \label{eqDecompTheoIntByParts}. 
\end{eqnarray}
Observe that by Theorem \ref{theoRiemannRepresentation}, the first and second terms in \eqref{eqDecompTheoIntByParts} converges in ucp to $\displaystyle{\int_{0}^{t} \, Y_{-} \, dX}$ and $\displaystyle{ \int_{0}^{t} \, X_{-} \, dY }$. 
Then, if we define the process $[X,Y]$ by 
$$ [X,Y]_{t} \defeq \inner{X_{t}}{Y_{t}}-\inner{X_{0}}{Y_{0}}-\int_{0}^{t} \, Y_{-} \, dX - \int_{0}^{t} \, X_{-} \, dY,  $$
we immediately get \eqref{eqDefIntegByParts} and that $[X,Y]$ is a real-valued adapted c\`{a}dl\`{a}g process. Moreover, from \eqref{eqDecompTheoIntByParts} and the arguments given above we conclude \eqref{eqDefQuadraCovaria}. 
\end{proof}

\begin{definition}
We call the real-valued processes $[X,Y]$ defined in Theorem \ref{theoIntegByPartsFormula} as the \emph{quadratic covariation} of $X$ and $Y$. 
\end{definition}

Some of the properties of the quadratic covariation are given in the following result. 

\begin{corollary}\label{coroPropQuadraCova} 
Suppose that $\Phi$, $\Phi'$, $X$ and $Y$ are as in Theorem \ref{theoIntegByPartsFormula}. Then
\begin{enumerate}
\item $\Delta [X,Y]_{0}=\inner{X_{0}}{Y_{0}}$ and $\Delta [X,Y]_{t}=\inner{\Delta X_{t}}{\Delta Y_{t}}$ for $t>0$.  
\item If $X$ or $Y$ is a continuous process, then $[X,Y]$ is a continuous process.  
\item If $\tau$ is a stopping time, 
$[X^{\tau},Y]=[X,Y^{\tau}]=[X^{\tau},Y^{\tau}]=[X,Y]^{\tau}$. 
\end{enumerate}
\end{corollary}
\begin{proof}
For $(1)$, it is clear from \eqref{eqDefQuadraCovaria} that $\Delta [X,Y]_{0}=\inner{X_{0}}{Y_{0}}$. If $t>0$ we have from Proposition \ref{propProperStochIntegLocalBounded} that
\begin{eqnarray*}
\inner{\Delta X_{t}}{\Delta Y_{t}} 
& = & \Delta\inner{X_{t}}{Y_{t}}-\inner{X_{t-}}{Y_{t}-Y_{t-}}-\inner{X_{t}-X_{t-}}{Y_{t-}} \\
& = & \Delta\inner{X_{t}}{Y_{t}} -\Delta \left( \int_{0}^{t} \, X_{-} \, dY \right) -\Delta \left( \int_{0}^{t} \, Y_{-} \, dX \right) 
= \Delta [X,Y]_{t}. 
\end{eqnarray*}
Property $(2)$ follows immediately from $(1)$ and property $(3)$ follows from  \eqref{eqDefQuadraCovaria}. 
\end{proof}

\begin{example}\label{examQuadraCova}
In this example we discuss some situations for which the quadratic covariation can be computed using the quadratic covariation for Hilbert space-valued processes. 

Suppose that $\Phi$ and $\Phi'$ are complete, barrelled nuclear spaces. Assume that $Y$ is a $\Phi$-valued good integrator which is also a semimartingale in the projective limit sense, i.e. for every Hilbertian seminorm $q$ on $\Phi$ we have that $\tilde{Y}^{q} \defeq i_{q} Y$ is a semimartingale with values in $\Phi_{q}$ (this definition coincides with that in Remark \ref{remaSemiDefiUstunel} but for processes with values in $\Phi$ instead of $\Phi'$). Let us recall that when $\Phi$ is a nuclear Fr\'{e}chet space any $\Phi$-valued semimartingale (in our cylindrical definition) is a semimartingale in the projective limit sense (see \cite{Ustunel:1982-1}, Corollary II.3). 

Assume that $X$ is a $\Phi'$-valued $\mathcal{H}^{1}_{S}$-semimartingale satisfying any of the equivalent conditions in Proposition \ref{propCylSemimContOpeSpaceSemim}. Observe that $X$ is a good integrator by Proposition \ref{propGoodIntegratorsNuclearSpace}. Moreover, by Theorem 3.22 in \cite{FonsecaMora:Semi} there exists a continuous Hilbertian seminorm $p$ on $\Phi$ and a $\Phi'_{p}$-valued c\`{a}dl\`{a}g $\mathcal{H}^{1}_{S}$-semimartingale $\tilde{X}^{p}$ such that $X=i'_{p} \tilde{X}^{p}$. Then for each $t \geq 0$ we have $\Prob$-a.e. 
$\inner{X_{t}(\omega)}{Y_{t}(\omega)}=\inner{i'_{p} \tilde{X}^{p}}{i_{p}Y_{t}(\omega)}=\inner{\tilde{X}^{p}}{\tilde{Y}^{p}},$
where $\inner{\tilde{X}^{p}}{\tilde{Y}^{p}}$ corresponds to the duality in the pair $(\Phi_{p},\Phi'_{p})$.  Then by \eqref{eqDefQuadraCovaria}, we have 
\begin{eqnarray*}
[X,Y]_{t} & = & \inner{X_{0}}{Y_{0}}+ \lim_{n \rightarrow \infty} \sum_{k=1}^{m_{n}} \, \inner{X_{\tau^{n}_{k+1} \wedge t}-X_{\tau^{n}_{k} \wedge t}}{Y_{\tau^{n}_{k+1} \wedge t}-Y_{\tau^{n}_{k} \wedge t}} \\
& = & \inner{\tilde{X}^{p}_{0}}{\tilde{Y}^{p}_{0}}+ \lim_{n \rightarrow \infty} \sum_{k=1}^{m_{n}} \, \inner{\tilde{X}^{p}_{\tau^{n}_{k+1} \wedge t}-\tilde{X}^{p}_{\tau^{n}_{k} \wedge t}}{\tilde{Y}^{p}_{\tau^{n}_{k+1} \wedge t}-\tilde{Y}^{p}_{\tau^{n}_{k} \wedge t}} \\
& = & [\tilde{X}^{p},\tilde{Y}^{p}], 
\end{eqnarray*}
where $[\tilde{X}^{p},\tilde{Y}^{p}]$ is the quadratic covariation  of $\tilde{X}^{p}$ and $\tilde{Y}^{p}$ as Hilbert space valued processes with respect to the duality in the pair $(\Phi_{p},\Phi'_{p})$ (see Result 26.9 in \cite{Metivier}, p.185). 
Observe that from this result we can conclude that $[X,Y]$ has finite variation and by \eqref{eqDefIntegByParts} that $(\inner{X_{t}}{Y_{t}}: t \geq 0)$ has a c\`{a}dl\`{a}g semimartingale version. 

Among other properties of $[X,Y]$ which can be ``transferred'' from corresponding properties of  $[\tilde{X}^{p},\tilde{Y}^{p}]$, observe that if $X$ has continuous paths and $Y$ is a projective limit of processes of finite variation  (i.e. each $\tilde{Y}^{q}$ is of finite variation in $\Phi_{q}$) then $[\tilde{X}^{p},\tilde{Y}^{p}]=0$, thus $[X,Y]=0$. 

We can obtain the same conclusions of the above paragraphs if we assume that $X$ is a $\Phi'$-valued good integrator which is also a semimartingale in the projective limit sense (if $\Phi'$ is nuclear Fr\'{e}chet any $\Phi'$-valued semimartingale in our cylindrical sense is a semimartingale in the projective limit sense; see \cite{Ustunel:1982-1}, Corollary II.3) and $Y$ is a $\Phi$-valued $\mathcal{H}^{1}_{S}$-semimartingale satisfying any of the equivalent conditions in Proposition \ref{propCylSemimContOpeSpaceSemim}. 
\end{example}

\begin{remark}\label{remaUstunelIntegByParts}
In this section we have extended the results of \"{U}st\"{u}nel (Theorem IV.1 in \cite{Ustunel:1982}) on the integration by parts formula and the existence of quadratic covariation for processes taking values in nuclear spaces. Indeed, in Theorem \ref{theoIntegByPartsFormula}, Corollary \ref{coroPropQuadraCova} and Example \ref{examQuadraCova} we consider less demanding assumptions on $\Phi$ and $\Phi'$  that those considered by  \"{U}st\"{u}nel in  \cite{Ustunel:1982}. In particular, we do not require neither that $\Phi$ or $\Phi'$ be bornological nor separable.  
\end{remark}

\section{The Stochastic Fubini Theorem and Applications to Stochastic Evolution Equations}\label{sectStochFubiniSEE}


\subsection{The Stochastic Fubini Theorem}\label{subsectStochFubiniTheorem}

The main objective of this section is to prove the following result:


\begin{theorem}[Stochastic Fubini Theorem]\label{theoStochFubini}
Let $(E,\mathcal{E},\varrho)$ be a finite measure space. Assume that $\Phi$ is a complete, barrelled nuclear space   whose strong dual $\Phi'$ is nuclear. Let $X$ be a $\Phi'$-valued semimartingale which is a good integrator. 
Let $H: [0,\infty) \times \Omega \times E \rightarrow \Phi$ such that $\forall f \in \Phi'$, the mapping $(e,t,\omega) \mapsto \inner{f}{H(e,t,\omega)}$ is $\mathcal{E} \otimes \mathcal{P}$-measurable and $\sup_{(e,t,\omega)}\abs{\inner{f}{H(e,t,\omega)}}< \infty$. Then, 
\begin{enumerate}
\item For every $(t,\omega) \in [0,\infty) \times \Omega$, the mapping $e \mapsto H(e,t,\omega)$ is Bochner integrable in $\Phi$, and the mapping $(t,\omega) \mapsto \int_{E}  H(e,t,\omega) \varrho(de)$ belongs to $b\mathcal{P}(\Phi)$. 
\item There exists a $\mathcal{E} \otimes \mathcal{B}(\R_{+})\otimes \mathcal{F}$-measurable function $K:E \times [0,\infty) \times \Omega \rightarrow \R$ such that for each $e \in E$, 
$K(e, \cdot,\cdot)$ is a c\`{a}dl\`{a}g version of $\int  H(e,\cdot, \cdot) dX$. 
\item $\int_{E} K(e, t,\cdot) \varrho(de)$  exists and has a c\`{a}dl\`{a}g version. Moreover,  up to indistinguishability
\begin{equation}\label{eqStochasticFubiniIdentity}
 \int_{E} K(e, t,\cdot) \varrho(de) = \int_{0}^{t} \left( \int_{E}  H(e, s,\cdot) \varrho(de)\right) dX_{s}. 
\end{equation}
\end{enumerate}
\end{theorem}

Before we can prove Theorem \ref{theoStochFubini} we will need the following preliminary results. Below $(E,\mathcal{E},\varrho)$, $\Phi$ and $\Phi'$ will be as in 
Theorem \ref{theoStochFubini}. 

Denote by $b\mathcal{P}_{E}$ the linear space of all the bounded mappings $h:E \times [0,\infty) \times \Omega  \rightarrow \R$ which are $\mathcal{E} \otimes \mathcal{P}$-measurable. This space is Banach when equipped with the uniform norm $\norm{h}_{u}=\sup_{(e,t,\omega)}\abs{h(e,t,\omega)}$. 

Let $b\mathcal{P}_{E}(\Phi)$ be the linear space of all mappings $H: E \times  [0,\infty) \times \Omega  \rightarrow \Phi$ such that for every $f \in \Phi'$, the mapping $(e,t,\omega) \mapsto \inner{f}{H(e,t,\omega)}$ is $\mathcal{E} \otimes \mathcal{P}$-measurable and $\sup_{(e,t,\omega)}\abs{\inner{f}{H(e,t,\omega)}}< \infty$. As in Proposition \ref{propIsomorBounPredSpaceLinOper} our assumptions on $\Phi$ imply that the mapping  \eqref{defIsomorBounPredSpaceLinOper} is an isomorphism  from $b\mathcal{P}_{E}(\Phi)$ into $\mathcal{L}_{b}(\Phi', b\mathcal{P}_{E})$. 

In analogy with our development in Section \ref{sectionCharacWeakIntegransNuclear}, we can equip $b\mathcal{P}_{E}(\Phi)$ with the topology of convergence in $\Phi$ uniformly on $E \times  \R_{+} \times \Omega$, i.e. 
the topology generated by the family of seminorms $\displaystyle{p(H)=\sup_{(e,t,\omega)} p(H(e,t,\omega))}$, where $p$ ranges over a generating family of seminorms for the topology on $\Phi$. The next result shows that our hypothesis that $\Phi'$ is nuclear extend the conclusion of Corollary \ref{coroDenseSimpleBoundedProcesses}. 

\begin{lemma}\label{lemmSequeApproxFubini} 
For every $H \in b\mathcal{P}_{E}(\Phi)$ there exists a sequence of elementary processes: 
\begin{equation}\label{eqSequElemIntegFubini}
 H_{n}(e,t,\omega)=\sum_{k=1}^{n} h_{k}(e,t,\omega) \phi_{k}, \quad \forall t \geq 0, \, \omega \in \Omega, e \in E,
\end{equation}
where $(h_{k}: k \in \N) \subseteq  b\mathcal{P}_{E}$ and $(\phi_{k}: k \in \N) \subseteq \Phi$; such that $(H_{n}: n \in \N)$ converges to $H$ in $b\mathcal{P}_{E}(\Phi)$. 
\end{lemma}
\begin{proof}
Let $A$ the image of $H$ under the isomorphism \eqref{defIsomorBounPredSpaceLinOper} from $b\mathcal{P}_{E}(\Phi)$ onto $\mathcal{L}_{b}(\Phi', b\mathcal{P}_{E})$. Since $\Phi'$ is nuclear, $A$ is a nuclear operator and has representation (Theorems III.7.1 and III.7.2 in \cite{Schaefer}, p.99-101)
$$ A(f)=\sum_{m=1}^{\infty} \lambda_{m} \inner{f}{\phi_{m}}h_{m}, $$
where $(\phi_{m}: m \in \N) \subseteq K$ for $K$ a compact, convex, balanced subset of $\Phi$ (that $K$ can be chosen to be compact is a consequence that $\Phi$ is a reflexive nuclear space; see Corollary 50.2.1 in \cite{Treves}, p.520), $\norm{h_{m}}_{u} \leq 1$, and $\sum_{m \in \N} \abs{\lambda_{m}} < \infty$.  Since the convergence of the series is in $\mathcal{L}_{b}(\Phi', b\mathcal{P}_{E})$, by the isomorphism   
\eqref{defIsomorBounPredSpaceLinOper} we have that the sequence $H_{n}=\sum_{m=1}^{n} \lambda_{m} \inner{f}{\phi_{m}}h_{m}$ of elementary processes  converges to $H$ in  $b\mathcal{P}_{E}(\Phi)$. 
\end{proof}

\begin{proof}[Proof of Theorem \ref{theoStochFubini}] For convenience we divide the proof in three parts.

\textbf{Proof of (1)} 
From our assumption $\sup_{(t,\omega,e)}\abs{\inner{f}{H(t,\omega,e)}}< \infty$ $\forall f \in \Phi'$, and by the arguments used in the proof of Proposition \ref{propIsomorBounPredSpaceLinOper}, there exists a compact, convex, balanced subset $D$ of $\Phi$ such that the image of $H$ is contained in $D$.   
As in the proof of Lemma \ref{lemmActionStoInteElemBounInteg} and using that $\Phi'$ is nuclear, there exists another such set $B$, $D \subseteq B$, with the property that the vector space $\Phi[B]$ spanned by $B$ is a separable Banach space when equipped with the gauge norm $q_{B}(\phi)=\inf\{\lambda> 0:  \phi \in \lambda B\}$ (see Corollary 36.1 in \cite{Treves}, p.371). 

Given the above, we can consider $H$ as a bounded mapping from $[0, \infty) \times \Omega \times e$ into $\Phi[B]$ which is weakly $\mathcal{P} \otimes \mathcal{E}$-measurable. Since  $\Phi[B]$ is a separable Banach space, by Pettis measurability theorem $H$ is $\mathcal{P} \otimes \mathcal{E} / \mathcal{B}(\Phi[B])$-measurable, and hence for each $(t,\omega) \in [0,\infty) \times \Omega$ the Bochner integral $\int_{E} H(t,\omega,e) \varrho(de)$ exists as an element in  $\Phi[B]$. Therefore  the mapping $e \mapsto H(t,\omega,e)$ is Bochner integrable in $\Phi$ (see  \cite{BeckmannDeitmar:2015, Thomas:1975}). Moreover, by Fubini's theorem  for every $f \in \Phi'$ the mapping $(t,\omega) \mapsto \inner{f}{\int_{E} H(t,\omega,e) \varrho(de)}=\int_{E} \inner{f}{H(t,\omega,e)} \varrho(de)$ is $\mathcal{P}$-measurable. 
Likewise, from our hypothesis on $H$ we have
$$ \sup_{(t,\omega)} \abs{ \inner{f}{\int_{E} H(t,\omega,e) \varrho(de)} } \leq \varrho(E) \sup_{(t,\omega,e)}\abs{\inner{f}{H(t,\omega,e)}}< \infty.$$
Hence the mapping $(t,\omega) \mapsto \int_{E}  H(t,\omega,e) \varrho(de)$ belongs to $b\mathcal{P}(\Phi)$.

\textbf{Proof of (2)} Let $(H_{n}: n \in \N)$ be a sequence of elementary processes converging to $H$ in $b \mathcal{P}_{E}(\Phi)$  as in Lemma \ref{lemmSequeApproxFubini}. For every $e \in E$ we have
$H_{n}(e,\cdot,\cdot) \in b \mathcal{P}(\Phi)$ and by \eqref{eqSequElemIntegFubini}, 
$$ \int H_{n}(e,\cdot,\cdot) dX = \sum_{i=1}^{n} \left( h_{i}(e,\cdot, \cdot) \cdot X(\phi_{i}) \right).$$ 
By the finite dimensional stochastic Fubini Theorem (e.g see Theorem IV.64 in \cite{Protter}) for each $i \in \N$ there exists a  $\mathcal{E} \otimes  \mathcal{B}(\R_{+})\otimes \mathcal{F}$-measurable function $k_{i}: E \times [0,\infty) \times \Omega  \rightarrow \R$ such that for each $e \in E$, $k_{i}(e,\cdot,\cdot)$ is a c\`{a}dl\`{a}g version of $ h_{i}(e,\cdot, \cdot) \cdot X(\phi_{i})$. This way, the process $K_{n}: E \times [0,\infty) \times \Omega  \rightarrow \R$ 
$$K_{n}(e,t, \omega)= \sum_{i=1}^{n} k_{i}(e,t, \omega),$$ 
is $\mathcal{B}(\R_{+})\otimes \mathcal{F} \otimes \mathcal{E}$-measurable and for each $e \in E$, $K_{n}(e, \cdot,\cdot)$ is a c\`{a}dl\`{a}g version of $\int H_{n}(e, \cdot,\cdot) dX$. Now, since $H_{n}$ converges to $H$ in $b \mathcal{P}_{E}(\Phi)$, for any given $e \in E$ we have by the good integrator property of $X$ that $\int H_{n}(e, \cdot,\cdot) dX \rightarrow \int H(e, \cdot,\cdot) dX$ in $S^{0}$ (hence in ucp). Hence for each $e \in E$ the sequence $K_{n}(e,\cdot,\cdot)$ converges in ucp.  Then, by Theorem IV.62 in \cite{Protter} there exists a $\mathcal{E} \otimes \mathcal{B}(\R_{+})\otimes \mathcal{F} $-measurable function $K: E\times  [0,\infty) \times \Omega  \rightarrow \R$ such that for each $e \in E$, $K(e, \cdot,\cdot)$ is c\`{a}dl\`{a}g and $K(e, \cdot,\cdot)=\lim_{n \rightarrow \infty} K_{n}(e, \cdot,\cdot)$ with convergence in ucp. 

As a consequence of the arguments given in the above paragraphs and by uniqueness of limits we have that $K(e, \cdot,\cdot)$ is a c\`{a}dl\`{a}g version of $\int  H(e, \cdot, \cdot) dX$.

\textbf{Proof of (3)} We start by showing that 
\eqref{eqStochasticFubiniIdentity} holds for each $K_{n}$. To do this, observe that for every $n \in \N$, by the finite dimensional stochastic Fubini theorem (e.g see Theorem IV.64 in \cite{Protter}, p.210), basic properties of the Bochner integral, and the action of the stochastic integral on simple processes (see \eqref{eqActionWeakIntegSimpleIntegNuclear}), we have up to indistinghishability, 
\begin{eqnarray}
\int_{E} K_{n}(e,t,\cdot) \varrho(de) 
& = & \sum_{i=1}^{n} \int_{E} \left( h_{i}(e, \cdot, \cdot) \cdot X(\phi_{i}) \right)_{t}  \varrho(de) \nonumber \\
& = & \sum_{i=1}^{n} \left( \left( \int_{E}  h_{i}(e, \cdot, \cdot)  \varrho(de) \right) \cdot X(\phi_{i}) \right)_{t} \nonumber \\
& = & \int_{0}^{t} \left( \int_{E}  H_{n}(e,s,\cdot) \varrho(de)\right) dX_{s}. \label{eqFubiniSimpleProcesses}
\end{eqnarray}
Our plan is to show that \eqref{eqFubiniSimpleProcesses} implies \eqref{eqStochasticFubiniIdentity} by taking limits with respect to an appropriate subsequence. But before we an do this, we must show the existence of the process $\int_{E} K(e, t,\cdot) \varrho(de)$. 

Fix $T>0$. Recall from Step 2 that for each $e \in E$ and $k \in \N$ we have as $n \rightarrow \infty$, 
$$ \Prob \left( \sup_{0 \leq t \leq T} \abs{ K(e, t,\cdot) - K_{n}(e, t,\cdot) } >\frac{1}{k} \right) \rightarrow 0. $$
Following the idea used in the proofs of Lemma 12.4.17 and Theorem 12.4.18 in \cite{CohenElliott}, we define the sequence
$$ n_{k}(e)=\inf \left\{ m > n_{k-1}(e): \Prob \left( \sup_{0 \leq t \leq T} \abs{ K(e, t,\cdot) - K_{n}(e, t,\cdot) } >\frac{1}{k} \right) \leq \frac{1}{k}, \, \forall n>m \right\}.$$
Then $n_{k}(\cdot)$ is a $\mathcal{E}$-measurable function, and 
$$ \Prob \left( \sup_{e \in E} \sup_{0 \leq t \leq T}  \abs{K(e, t,\cdot) - K_{n_{k}(e)}(e,t,\cdot)}  > \frac{1}{k} \right) \leq \frac{1}{k}. $$
Therefore, $K_{n_{k}(e)}(e,t,\cdot) \rightarrow K(e,t,\cdot)$ in probability uniformly on $E \times [0,T]$. Moreover, for any $\epsilon >0$ it follows that as $k \rightarrow \infty$, 
\begin{multline}\label{eqUCPConvLebIntegStochFubini}
\Prob \left( \int_{E} \sup_{0 \leq t \leq T} \abs{K(e, t,\cdot) - K_{n_{k}(e)}(e, t,\cdot)} \varrho(de) > \epsilon  \right) \\
\leq \Prob \left( \sup_{e \in E} \sup_{0 \leq t \leq T} \abs{K(e,t,\cdot) - K_{n_{k}(e)}(e,t,\cdot)}  > \frac{\epsilon}{\varrho(E)} \right) \rightarrow 0.  
\end{multline}
By the  finite dimensional stochastic Fubini theorem (e.g see Theorem IV.64 in \cite{Protter}) each $\int_{E} k_{i}(e,t,\cdot) \varrho(de)$ has a c\`{a}dl\`{a}g version, the same is true for $\int_{E} K_{n_{k}(e)}(e,t,\cdot) \varrho(de)$. Then by \eqref{eqUCPConvLebIntegStochFubini} we have that the process $\int_{E} K(e, t,\cdot) \varrho(de)$ exists as the limit in ucp of the sequence $\int_{E} K_{n_{k}(e)}(e,t,\cdot) \varrho(de)$ as $k \rightarrow \infty$, and therefore $\int_{E} K(e, t,\cdot) \varrho(de)$ has a c\`{a}dl\`{a}g version. 

Our next objective will be to show that the sequence $\int_{0}^{t} \left( \int_{E} H_{n_{k}(e)}(e,s,\cdot) \varrho(de)\right) dX_{s}$ converges to $\int_{0}^{t} \left( \int_{E} H(e,s,\cdot) \varrho(de)\right) dX_{s}$ in ucp. 

In effect, for any given continuous seminorm $p$ on $\Phi$ by the basic properties of the Bochner integral and since $H_{n} \rightarrow H$ in $b\mathcal{P}_{E}(\Phi)$ we have as $k \rightarrow \infty$, 
\begin{flalign*}
&\sup_{(t,\omega)} p \left( \int_{E} H(e,t,\cdot) \varrho(de)  - \int_{E} H_{n_{k}(e)}(e,t,\cdot) \varrho(de) \right) \\
& \leq \sup_{(t,\omega)}  \int_{E} p \left( H(e,t,\cdot) - H_{n_{k}(e)}(e,t,\cdot) \right) \varrho(de) \\ 
& \leq \varrho(E) \sup_{(e,t,\omega)}  p \left( H(e,t,\cdot) - H_{n_{k}(e)}(e,t,\cdot) \right)  \rightarrow 0. 
\end{flalign*}
Hence $\int_{E} H_{n_{k}(e)}(e,t,\cdot) \varrho(de) \rightarrow \int_{E} H(e,t,\cdot) \varrho(de)$ in $b\mathcal{P}(\Phi)$ as $k \rightarrow \infty$. By the good integrator property of $X$ we have that $\int_{0}^{t} \left( \int_{E} H_{n_{k}(e)}(e,s,\cdot) \varrho(de)\right) dX_{s}$ converges to $\int_{0}^{t} \left( \int_{E} H(e,s,\cdot) \varrho(de)\right) dX_{s}$ in $S^{0}$, thus in ucp. 

Finally, by \eqref{eqFubiniSimpleProcesses} we have (with limits in ucp), 
\begin{eqnarray*}
\int_{E} K(e, t,\cdot) \varrho(de) 
& = & \lim_{k \rightarrow \infty} \int_{E} K_{n_{k}(e)}(e,t,\cdot) \varrho(de) \\
& = & \lim_{k \rightarrow \infty} \int_{0}^{t} \left( \int_{E} H_{n_{k}(e)}(e,s,\cdot) \varrho(de)\right) dX_{s} \\
& = & \int_{0}^{t} \left( \int_{E} H(e,s,\cdot) \varrho(de)\right) dX_{s}. 
\end{eqnarray*}
\end{proof}

\subsection{Stochastic Evolution Equations in the Dual of a Nuclear Space}\label{subsectSEECylSemiNoise}

Al through this section we assume  that $\Phi$ is a complete, barrelled nuclear space  whose strong dual $\Phi$ is nuclear.

Our main objective in this section is to show existence and uniqueness of solutions to the following class of stochastic evolution equations
\begin{equation}\label{eqStochEvolEqua}
dY_{t}=A' Y_{t} +  dX_{t}, t \geq 0,
\end{equation}
with initial condition $Y_{0}=\eta$ $\Prob$-a.e., where $\eta$ is a $\mathcal{F}_{0}$-measurable $\Phi'$-valued regular random variable, $X$ is $\Phi'$-valued semimartingale which is a good integrator, and $A$ is the generator of a strongly continuous $C_{0}$-semigroup $(S(t): t \geq 0) \subseteq \mathcal{L}(\Phi,\Phi)$ (see \cite{Komura:1968} for information on $C_{0}$-semigroups in locally convex spaces). 

\begin{definition}
A $\Phi'$-valued regular adapted process $Y=(Y_{t}: t \geq 0)$ is called a \emph{weak solution} to \eqref{eqStochEvolEqua} if for any given $t > 0$ and  $\phi \in \mbox{Dom}(A)$ we have $\int_{0}^{t} \abs{\inner{Y_{r}}{A\phi}} dr < \infty$ $\Prob$-a.e. and 
\begin{equation}\label{eqWeakSoluStochEvolEqua}
\inner{Y_{t}}{\phi}= \inner{\eta}{\phi} + \int_{0}^{t} \inner{Y_{r}}{A\phi} dr + \inner{X_{t}}{\phi}. 
\end{equation}
\end{definition}  

The natural candidate for a weak solution to \eqref{eqStochEvolEqua} is the so called \emph{mild solution}: 
\begin{equation}\label{eqMildSoluStochEvolEqua}
Z_{t}=S(t)'\eta + \int_{0}^{t} S(t-r)' \, dX_{r}, \quad t \geq 0. 
\end{equation}
Here observe that since $\Phi$ is reflexive, the family of dual operators $(S(t)': t \geq 0) \subseteq  \mathcal{L}(\Phi',\Phi')$ is a strongly continuous $C_{0}$-semigroup whose generator is $A'$ (see the Corollary to Theorem 1 in \cite{Komura:1968}).    

We must prove first that the stochastic convolution, i.e. the stochastic integral $ \int_{0}^{t} S(t-r)' \, dX_{r}$ is well-defined. This task is carried out in the following result. 

\begin{proposition}\label{propExistStochConvo}
There exists a $\Phi'$-valued regular adapted process $ \int_{0}^{t} S(t-r)' \, dX_{r}$,  $t \geq 0$, satisfying for each $t \geq 0 $ and $\phi \in \Phi$ that $\Prob$-a.e. 
\begin{equation}\label{eqDefiStochConvol}
\inner{\int_{0}^{t} S(t-r)' \, dX_{r}}{\phi} = \int_{0}^{t} S(t-r)\phi \, dX_{r}.
\end{equation} 
Moreover, the $\Phi'$-valued process $(Z_{t}: t \geq 0)$ defined in \eqref{eqMildSoluStochEvolEqua} is regular and adapted. 
\end{proposition}
\begin{proof}
Let $t \geq 0 $ and $\phi \in \Phi$. Let $H^{\phi}$ be defined as $H^{\phi}(r,\omega)=\ind{[0,t]}{r} S(t-r) \phi$. The continuity of the mapping $[0,t] \ni r \mapsto S(t-r)\phi$ easily shows that $H^{\phi} \in b\mathcal{P}(\Phi)$, therefore the stochastic integral $\int_{0}^{t} S(t-r) \phi \, dX_{r}$ exists. 

Likewise, by continuity of the mapping $[0,t] \ni r \mapsto S(t-r) \phi$ and since $\Phi$ is barrelled, by an application of the closed graph theorem it can be shown that the mapping $(r,\phi) \mapsto S(t-r)\phi$ is continuous from $[0,t] \times \Phi$ into $\Phi$. From this it follows that the mapping $\phi \mapsto H^{\phi}$ is continuous from $\Phi$ into $b\mathcal{P}(\Phi)$. By the good integrator property of $X$ the mapping $\phi \mapsto \int_{0}^{t} H^{\phi}(r) dX_{r} = \int_{0}^{t} S(t-r) \phi \, dX_{r}$ is continuous from $\Phi$ into $L^{0} \ProbSpace$. Then, the regularization theorem (see \cite{ItoNawata:1983}) shows that there exists a $\Phi'$-valued regular random variable $ \int_{0}^{t} S(t-r)' \, dX_{r}$ satisfying \eqref{eqDefiStochConvol}. Moreover, \eqref{eqDefiStochConvol} shows that $ \int_{0}^{t} S(t-r)' \, dX_{r}$ is weakly adapted, hence adapted because it is regular.    

Finally, since $\eta$ is a $\mathcal{F}_{0}$-measurable $\Phi'$-valued regular random variable, the strong continuity of the semigroup $(S(t)': t \geq 0)$ shows that $(S(t)'\eta : t \geq 0)$  is regular and adapted and hence so is $(Z_{t}: t \geq 0)$. 
\end{proof}

\begin{theorem}\label{theoExisSolutionsSEE}
The stochastic evolution equation \eqref{eqStochEvolEqua} has a weak solution given by \eqref{eqMildSoluStochEvolEqua}. 
\end{theorem}

A key step in the proof of Theorem \ref{theoExisSolutionsSEE} is carried out in the following result which is a consequence of the stochastic Fubini theorem (Theorem \ref{theoStochFubini})  

\begin{lemma}\label{lemmaStochasticFubiniSEE}
For every $t>0$ and $\phi \in \mbox{Dom}(A)$, we have $\Prob$-a.e.
\begin{equation}\label{eqFiniteIntegStocFubini}
\int_{0}^{t} \abs{ \int_{0}^{s} S(s-r) A \phi \, dX_{r} } ds < \infty, \quad \mbox{and} \quad \int_{0}^{t} \abs{ \int_{0}^{s} S(t-s) A \phi \, dX_{r} } ds < \infty.
\end{equation}
Moreover, the following identities holds $\Prob$-a.e. 
\begin{eqnarray} \label{eqStochasticFubiniSEE}
\int_{0}^{t} \left( \int_{0}^{s} S(s-r) A \phi \, dX_{r} \right) ds  
& = & \int_{0}^{t} S(t-r) \phi \, dX_{r} - \inner{X_{t}}{\phi} \nonumber \\
& = & \int_{0}^{t} \left( \int_{0}^{s} S(t-s) A \phi \, dX_{r} \right) ds. 
\end{eqnarray}
\end{lemma}
\begin{proof}
Fix $t >0$ and $\phi \in \mbox{Dom}(A)$. Consider the following families of $\Phi$-valued mappings: 
\begin{align*}
H_{1}(s,r) &= \ind{[0,s]}{r} S(s-r)A \phi. \\
H_{2}(s,r) &= \ind{[0,s]}{r} S(t-s)A \phi.
\end{align*}
for $s \in [0,t]$, $r \in [0,t]$. As a direct consequence of the strong continuity of the $C_{0}$-semigroup $(S(t): t \geq 0)$ we can verify that $H_{1}$ and $H_{2}$ satisfy the conditions of the stochastic Fubini theorem (Theorem \ref{theoStochFubini}). Hence we have \eqref{eqFiniteIntegStocFubini} and the iterated integrals in \eqref{eqStochasticFubiniSEE} exists. 

Now observe that for any given $r \in [0,t]$, from standard properties of $C_{0}$-semigroups (see \cite{Komura:1968}) we have
$$ \int_{0}^{t} \ind{[0,s]}{r} S(s-r)A \phi \, ds = S(t-r)\phi -\phi = \int_{0}^{t} \ind{[0,s]}{r} S(t-s)A \phi \, ds. $$
Then, from the above identity and the stochastic Fubini theorem applied to $H_{1}$ and $H_{2}$ we have
\begin{eqnarray*} 
\int_{0}^{t} \left( \int_{0}^{s} S(s-r) A \phi \, dX_{r} \right) ds  
& = & \int_{0}^{t} \left( \int_{0}^{s} S(t-s) A \phi \, dX_{r} \right) ds  \\
& = & \int_{0}^{t} \left( \int_{0}^{t} \ind{[0,s]}{r} S(t-s) A \phi \, ds \right) dX_{r}  \\
& = & \int_{0}^{t} S(t-r) \phi \, dX_{r} - \inner{X_{t}}{\phi}. 
\end{eqnarray*}  
\end{proof}

\begin{proof}[Proof of Theorem \ref{theoExisSolutionsSEE}]
Let $Z=(Z_{t}: t \geq 0)$ be as defined in \eqref{eqMildSoluStochEvolEqua}. We know from Proposition \ref{propExistStochConvo} that $Z$ is a $\Phi'$-valued regular and adapted process. 

Let $t >0$ and $\phi \in \Phi$. By the strong continuity of the $C_{0}$-semigroup $(S(t)':t \geq 0)$ it follows that $\int_{0}^{t} \abs{\inner{S(r)'\eta}{A\phi}} dr < \infty$. Therefore, it follows from \eqref{eqMildSoluStochEvolEqua} and \eqref{eqFiniteIntegStocFubini} that $\Prob$-a.e. $\int_{0}^{t} \abs{\inner{Z_{r}}{A\phi}} dr < \infty$. 

We must show $(Z_{t}: t \geq 0)$ satisfies \eqref{eqWeakSoluStochEvolEqua}. In effect, for any $s \in [0,t]$ from \eqref{eqMildSoluStochEvolEqua} and \eqref{eqDefiStochConvol} we have $\Prob$-a.e.
$$ \inner{Z_{s}}{A\phi}=\inner{\eta}{S(s) A\phi}
+ \int_{0}^{s} \, S(s-r) A \phi \, dX_{r}. $$
Integrating both sides on $[0,t]$ with respect to the Lebesgue measure, then from standard properties of $C_{0}$-semigroups (see \cite{Komura:1968}), from \eqref{eqStochasticFubiniSEE}, and then from \eqref{eqMildSoluStochEvolEqua} and \eqref{eqDefiStochConvol}, we have $\Prob$-a.e.
\begin{eqnarray*}
 \int_{0}^{t} \inner{Z_{s}}{A \phi} ds 
& = &  \int_{0}^{t} \inner{\eta}{S(s) A\phi} ds 
+ \int_{0}^{t} \left( \int_{0}^{s} \, S(s-r) A \phi \, dX_{r} \right) ds  \\
& = & \inner{\eta}{S(t)\phi-\phi} + \int_{0}^{t} S(t-r) \phi \, dX_{r} - \inner{X_{t}}{\phi} \\
& = & \inner{Z_{t}}{\phi} -\inner{\eta}{\phi} - \inner{X_{t}}{\phi}.  
\end{eqnarray*} 
Hence, the mild solution $Z=(Z_{t}: t \geq 0)$ is a weak solution to \eqref{eqStochEvolEqua}. 
\end{proof}

We can prove a converse of Theorem \ref{theoExisSolutionsSEE} if we assume that our weak solution has some regularity property on its paths. This property is described below. 

\begin{definition}\label{defiASLocalBocnTraje}
We say that a $\Phi'$-valued adapted process $Y=(Y_{t}: t \geq 0)$ has \emph{almost surely locally Bochner integrable trajectories} if for each $t>0$ there exists $\Omega_{t} \subseteq \Omega$ with $\Prob(\Omega_{t})=1$ and such that for each $\omega \in \Omega_{t}$ there exists a continuous seminorm $p=p(t,\omega)$ on $\Phi$ such that $Y_{s}(\omega) \in \Phi_{p}$ a.e. on $[0,t]$ and $\int_{0}^{t} \, p'(Y_{s}(\omega)) \, ds< \infty$. 
\end{definition}

\begin{example}\label{examASLocalBocnTraje}
For a $\Phi'$-valued adapted process $Y=(Y_{t}: t \geq 0)$ the property of having almost surely locally Bochner integrable trajectories is implied by any of the following conditions:
\begin{enumerate}
\item For each $T>0$, $Y$ is (weakly) jointly measurable on $[0,T] \times \Omega$  and $\Exp \int_{0}^{T} \abs{\inner{Y_{t}}{\phi}} dt < \infty$ for each $\phi \in \Phi$.   In effect, under these assumptions it is shown in Lemma 6.11 in \cite{FonsecaMora:2018-1} that for each $T>0$ there exists a continuous Hilbertian seminorm $p$ on $\Phi$ for which $\Exp \int_{0}^{T} p'(Y_{t}) dt < \infty$. 
\item $Y$ has c\`{a}dl\`{a}g paths. In that case, for $\Prob$-a.e. $\omega \in \Omega$, by Remark 3.6 in \cite{FonsecaMora:Skorokhod} for each $T>0$ there exists a continuous Hilbertian seminorm $p$ on $\Phi$ such that $t \mapsto Y_{t}(\omega)$ is  c\`{a}dl\`{a}g from $[0,T]$ into $\Phi'_{p}$. In particular we have
$$ \int_{0}^{T} \, p'(Y_{t}(\omega)) \, dt \leq T \, \sup_{0 \leq t \leq T}  p'(Y_{t}(\omega)) < \infty. $$
\end{enumerate}
\end{example}

We are ready for our main result on uniqueness of solutions. 

\begin{theorem}\label{theoUniqueSoluSEE} 
Let $Y=(Y_{t}: t \geq 0)$ be a weak solution  to \eqref{eqStochEvolEqua} which has almost surely locally Bochner integrable trajectories.  
Then for each $t \geq 0$, $Y_{t}$ is given $\Prob$-a.e. by 
\begin{equation}\label{eqWeakSolIsMild}
 Y_{t}=S(t)'\eta+\int_{0}^{t} \, S(t-r)' dX_{r}.
\end{equation}
\end{theorem}
\begin{proof} We modify the arguments used in the proof of Theorem 6.9 in \cite{FonsecaMora:2018-1}. 

Let $t \geq 0 $ and $\phi \in \mbox{Dom}(A)$.  Let  $\omega \in \Omega_{t}$ and $p=p(t,\omega)$ as in Definition \ref{defiASLocalBocnTraje}. By the strong continuity of the $C_{0}$-semigroup $(S(r): r \geq 0)$ and since $S(t-s)A \phi \in \mbox{Dom}(A)$ for $s \in[0,t]$, the set $\{ A S(t-s)A \phi: 0 \leq s \leq t\}$ is bounded in $\Phi$, hence is bounded under the seminorm $p$. Therefore
\begin{multline*}
\int_{0}^{t} \left( \int_{0}^{t} \, \abs{  \mathbbm{1}_{[0,s]}(r) \inner{Y_{r}(\omega)}{A S(t-s)A \phi} } dr \right) ds \\
 \leq  t \sup_{0 \leq r \leq s \leq t} p(A S(t-s)A \phi) \, \int_{0}^{t} \, p'(Y_{s}(\omega)) \, ds< \infty.
\end{multline*}
Then from Fubini's theorem and standard properties of the dual semigroup $(S(r)': r \geq 0)$, we have for $\omega \in \Omega_{t}$
\begin{eqnarray}
\int_{0}^{t} \left( \int_{0}^{s} \, \inner{Y_{r}(\omega)}{A S(t-s)A \phi} dr \right) ds 
& = & \int_{0}^{t} \left( \int_{0}^{t-r} \, \inner{Y_{r}(\omega)}{AS(s)A\phi} ds \right) dr  \nonumber \\
& = & \int_{0}^{t}  \inner{\int_{0}^{t-r} S(s)'A' Y_{r}(\omega) ds}{A\phi}  \,  dr  \nonumber \\
& = & \int_{0}^{t} \inner{S(t-r)'Y_{r}(\omega) -Y_{r}(\omega)}{A\phi}  \,  dr  \nonumber \\
& = & \int_{0}^{t}  \inner{Y_{r}}{(S(t-r) A\phi- A \phi)}  dr  \label{eqFubiniGeneratorEvoluSystem}
\end{eqnarray}
Given any $s \in [0,t]$, from the definition of weak solution (with $\phi$ replaced by $S(t-s)A \phi$), we have $\Prob$-a.e. 
\begin{eqnarray*}
\int_{0}^{s} \, S(t-s)A \phi \, dX_{r} 
& = & \inner{X_{s}}{S(t-s)A \phi} \\
& = & \inner{Y_{s}-\eta}{S(t-s)A \phi}- \int_{0}^{s} \, \inner{Y_{r}}{AS(t-s)A\phi}\, dr.
\end{eqnarray*}
Integrating both sides on $[0,t]$ with respect to the Lebesgue measure and then using \eqref{eqFubiniGeneratorEvoluSystem}, it follows that $\Prob$-a.e. 
\begin{eqnarray*}
 \int_{0}^{t} \left( \int_{0}^{s} \, S(t-s)A \phi \, dX_{r}  \right)  ds  
& = & \int_{0}^{t} \, \inner{Y_{s}}{S(t-s)A \phi}  ds -\int_{0}^{t} \inner{\eta}{S(t-s)A \phi}  ds \\
& {} &   - \int_{0}^{t} \left( \int_{0}^{s} \, \inner{Y_{r}}{AS(t-s)A\phi}\, dr \right) ds \\ 
& =  & \int_{0}^{t}  \inner{Y_{s}}{A\phi}  ds -\left( \inner{\eta}{S(t)\phi}- \inner{\eta}{\phi} \right).  
\end{eqnarray*}
Reordering, and then from \eqref{eqStochasticFubiniSEE} we get that 
\begin{flalign*}
& \inner{\eta}{\phi} + \int_{0}^{t} \inner{Y_{s}}{A\phi} ds \\
& = \inner{S(t)'\eta}{\phi}+ \int_{0}^{t} \left( \int_{0}^{s} \, S(t-s)A \phi \, dX_{r}  \right)  ds    \\
& =\inner{S(t)'\eta}{\phi} + \int_{0}^{t} S(t-r) \phi \, dX_{r} - \inner{X_{t}}{\phi}. 
\end{flalign*} 
Summing the term $\inner{X_{t}}{\phi}$ at both sides, then using  \eqref{eqDefiStochConvol} and the fact that $Y$ is a weak solution we conclude that $\Prob$-a.e
\begin{equation*}
\inner{Y_{t}}{\psi}=\inner{S(t)'\eta}{\psi} + \inner{\int_{0}^{t} S(t-r)' \, dX_{r}}{\phi}.
\end{equation*}
The above equality is valid for each $\psi \in \mbox{Dom}(A)$, and since $\mbox{Dom}(A)$ is dense in $\Phi$ we conclude that $Y_{t}$ is of the form \eqref{eqWeakSolIsMild}. 
\end{proof}

\begin{corollary}\label{coroExisteUniqueSoluSEE}
If the stochastic convolution process $(\int_{0}^{t} S(t-r)' \, dX_{r}:  t \geq 0)$ has almost surely locally Bochner integrable trajectories (respectively c\`{a}dl\`{a}g/continuous paths) then
the mild solution \eqref{eqMildSoluStochEvolEqua} is the unique weak solution to \eqref{eqStochEvolEqua} having almost surely locally Bochner integrable trajectories (respectively c\`{a}dl\`{a}g/continuous paths).
\end{corollary}
\begin{proof}
Observe that since $(S(t)'\eta: t \geq 0)$ has almost surely locally Bochner integrable trajectories (as it has continuous paths by the strong continuity of the dual semigroup), then if the stochastic convolution process $(\int_{0}^{t} S(t-r)' \, dX_{r}:  t \geq 0)$ has almost surely locally Bochner integrable trajectories (respectively c\`{a}dl\`{a}g/continuous paths) then the mild solution 
\eqref{eqMildSoluStochEvolEqua} possesses this property too. The result is then a consequence of Theorems \ref{theoExisSolutionsSEE} and \ref{theoUniqueSoluSEE}. 
\end{proof}
 
In the last part of this section we will show existence and uniqueness of a solution with continuous paths under the assumption that $A \in \mathcal{L}(\Phi,\Phi)$ is the generator of a  $(C_{0},1)$-semigroup and $X$ is a $\mathcal{H}^{1}_{S}$-semimartingale with continuous paths. 

Recall that a $C_{0}$-semigroup $( S(t) : t \geq 0)$ is called a \emph{$(C_{0},1)$-semigroup} if for each continuous seminorm $p$ on $\Phi$  there exists some $\vartheta_{p} \geq 0$ and a continuous seminorm $q$ on $\Phi$ such that  $ p(S(t)\phi) \leq e^{\vartheta_{p} t} q(\phi)$, for all $t \geq 0$ and $\phi \in \Phi$. $(C_{0},1)$-semigroups where introduced by Babalola in \cite{Babalola:1974} and include the class of \emph{quasiequicontinuous} $C_{0}$-semigroups
(case $\vartheta_{p}=\sigma$ with $\sigma$ a positive constant independent of $p$) and \emph{equicontinuous} $C_{0}$-semigroups
(case $\vartheta_{p}=0$ for each $p$).

\begin{theorem}\label{theoExisUniqCadlagSoluSEE}
Suppose that $\Phi$ is a complete bornological nuclear space whose strong dual $\Phi'$  is complete and nuclear. Assume $A \in \mathcal{L}(\Phi,\Phi)$ is the generator of a  $(C_{0},1)$-semigroup and $X$ is a $\Phi'$-valued $\mathcal{H}^{1}_{S}$-semimartingale with continuous paths satisfying any of the equivalent conditions in Proposition \ref{propCylSemimContOpeSpaceSemim}. Then, the stochastic evolution equation  \eqref{eqStochEvolEqua} has a unique weak solution $Z=(Z_{t}: t \geq 0)$ with continuous paths satisfying for each $t \geq 0$ and $\phi \in \Phi$ that $\Prob$-a.e.
\begin{equation}\label{eqCadlagVersiMildSolu}
\inner{Z_{t}}{\phi}= \inner{\eta -X_{0}}{S(t)\phi} +\inner{X_{t}}{\phi} +\int_{0}^{t} \inner{X_{t}}{S(t-s)A\phi} ds. 
\end{equation}
\end{theorem}
\begin{proof}
Our first objective is to show that the mild solution \eqref{eqMildSoluStochEvolEqua} satisfies \eqref{eqCadlagVersiMildSolu}. In view of \eqref{eqDefiStochConvol} it suffices to show that for each $t \geq 0$ and $\phi \in \Phi$  we have $\Prob$-a.e. 
\begin{equation}\label{eqCadlagVersiStocConvo}
\int_{0}^{t} S(t-r)\phi \, dX_{r} = \inner{X_{t}}{\phi} -  \inner{X_{0}}{S(t)\phi} +  \int_{0}^{t} \inner{X_{t}}{S(t-s)A\phi} ds.
\end{equation}
To do this, let $t \geq 0$ and $\phi \in \Phi$, and define $Y_{s}=\mathbbm{1}_{[0,t]}(s)S(t-s) \phi$. We will check that \eqref{eqCadlagVersiStocConvo} holds by applying the stochastic integration by parts formula to the processes $X$ and $Y$, but before we will verify that $[X,Y]=0$. This last result will follows from the arguments used in Example \ref{examQuadraCova} if we can show that $Y$ defines a projective limit of processes of bounded variation.

In effect, for any given continuous Hilbertian seminorm $q$ on $\Phi$, since $(S(s): s \geq 0)$ is a  $(C_{0},1)$-semigroup there exist a continuous seminorm $p$ on $\Phi$, $q \leq p$, and a $C_{0}$-semigroup $( S_{p}(t): t \geq 0)$ on the Banach space $\Phi_{p}$ such that
$S_{p}(s) i_{p} \varphi= i_{p} S(s) \varphi$, $\forall \varphi \in \Phi$, $s \geq 0$ (see Theorems 2.3 and 2.6 in \cite{Babalola:1974}).  
Hence, $s \mapsto i_{q} Y_{s}= \mathbbm{1}_{[0,t]}(s)i_{q,p} S_{p}(t-s) \phi$ is a $\Phi_{q}$-valued function with bounded variation on $[0,t]$ (recall $s \mapsto S_{p}(s)\phi$ is norm continuously differentiable on $[0,\infty)$). Since the above is true for every continuous Hilbertian seminorm on $\Phi$, then $Y$ defines a projective limit of processes of finite variation. Since $X$ is a $\Phi'$-valued $\mathcal{H}^{1}_{S}$-semimartingale with continuous paths, from the arguments used in Example \ref{examQuadraCova} we have $[X,Y]=0$. 

By an application of the stochastic integration by parts formula (Theorem \ref{theoIntegByPartsFormula}), and then from standard properties of $C_{0}$-semigroups (see \cite{Komura:1968}), we have $\Prob$-a.e.
\begin{eqnarray*}
\int_{0}^{t} S(t-s)\phi \, dX_{s}
& = & \inner{X_{t}}{\phi} -  \inner{X_{0}}{S(t)\phi} - \int_{0}^{t} X_{s} \,  d_{s} (S(t-s) \phi) \\ 
& = & \inner{X_{t}}{\phi} -  \inner{X_{0}}{S(t)\phi} + \int_{0}^{t} \inner{X_{t}}{S(t-s)A\phi} ds 
\end{eqnarray*}
Observe that by \eqref{eqCadlagVersiStocConvo} the stochastic convolution $U_{t}=\int_{0}^{t} S(t-r)' \, dX_{r}$ satisfies that for each  $\phi \in \Phi$ the process  $(\inner{U_{t}}{\phi}: t \geq 0)$ has a continuous version. Moreover, from the arguments used in the proof of Proposition  \ref{propExistStochConvo} we conclude that for each $t \geq 0$ the mapping $\phi \mapsto \inner{U_{t}}{\phi}$ is continuous from $\Phi$ into $L^{0}\ProbSpace$. Since $\Phi$ is ultrabornological (it is complete and bornological, see Theorem 13.2.12 in \cite{NariciBeckenstein}, p.449), the regularization theorem for ultrabornological nuclear spaces (Corollary 3.11 in \cite{FonsecaMora:2018}) shows that $(U_{t}: t \geq 0)$ has a $\Phi'$-valued version with continuous paths satisfying \eqref{eqCadlagVersiStocConvo}. Hence the mild solution $(Z_{t}: t \geq 0)$ satisfies \eqref{eqCadlagVersiMildSolu} and by Corollary \ref{coroExisteUniqueSoluSEE} we have that $(Z_{t}: t \geq 0)$ is the unique weak solution to \eqref{eqStochEvolEqua} with continuous paths. 
\end{proof}

\begin{example}
Let $\Phi=\mathscr{S}(\R^{d})$, for $d \geq 1$ (see Section \ref{sectNuclearSpaceCylSemi}). Let $X$ be a $\mathscr{S}(\R^{d})'$-valued $\mathcal{H}^{1}_{S}$-semimartingale with continuous paths satisfying any of the equivalent conditions in Proposition \ref{propCylSemimContOpeSpaceSemim}.
Consider the (lineal) stochastic heat equation on  $\mathscr{S}(\R^{d})'$:
\begin{equation}\label{eqHeatSEELevy}
d Y_{t}=\Delta Y_{t}+ dX_{t}, \quad t \geq 0,
\end{equation}
with initial condition $Y_{0}=\eta$, for a  $\mathscr{S}(\R^{d})'$-valued $\mathcal{F}_{0}$-measurable random variable, and here  $\Delta$ is the Laplace operator on $\mathscr{S}(\R^{d})'$. 

It is well-known that $\Delta \in \mathcal{L}(\mathscr{S}(\R^{d}),\mathscr{S}(\R^{d}))$ is the infinitesimal generator of the \emph{heat semigroup}  $(S(t): t \geq 0)$ which is the equicontinuous $C_{0}$-semigroup on $\mathscr{S}(\R^{d})$ defined as: $S(0)=I$ and for each $t>0$, 
\begin{equation*}
(S(t)\phi)(x) \defeq \frac{1}{(4 \pi t)^{d/2}}  \int_{\R^{d}} e^{-\norm{x-y}^{2}/4t} \phi(y) dy, \quad \forall \phi \in  \mathcal{S}(\R^{d}), \, x \in \R^{d}. 
\end{equation*}
If we let $\mu_{t}(x)= \frac{1}{(4 \pi t)^{d/2}}  e^{-\norm{x}^{2}/4t}$, then $\mu_{t} \in \mathscr{S}(\R^{d})$ and $S(t)\phi=\mu_{t} \ast \phi$ for each $\phi \in \mathscr{S}(\R^{d})$. 

Then Theorem \ref{theoExisUniqCadlagSoluSEE} shows that \eqref{eqHeatSEELevy} has a unique weak solution $(Z_{t}: t \geq 0)$ with continuous paths satisfying for any $t \geq 0$ and $\phi \in   \mathscr{S}(\R^{d})$:
\begin{equation*}\label{eqSoluHeatEqua}
\inner{Z_{t}}{\phi}= \inner{\eta-X_{0}}{\mu_{t} \ast \phi} + \inner{X_{s}}{\phi} +\int_{0}^{t} \inner{X_{s}}{\Delta (\mu_{t-s} \ast \phi)} \, ds. 
\end{equation*}
This example extends to the context of $\mathcal{H}^{1}_{S}$-semimartingales with continuous paths the results obtained by \"{U}st\"{u}nel in \cite{Ustunel:1982-2} on the existence and uniqueness of solutions for the stochastic heat equation in $\mathscr{S}(\R^{d})'$ driven by a Wiener process.
\end{example}

\section{Applications to Stochastic Integration for a Sequence of Real-Valued Semimartingales}\label{sectAppliSequenceSemi}

\subsection{Literature Review}\label{subSecLiteReviewSequeSemi}

In \cite{DeDonnoPratelli:2006}, De Donno and Pratelli introduced a theory of stochastic integration with respect to a sequence of real-valued semimartingales. The main idea behind the construction is to think on a sequence of semimartingales $Z=(Z^{j}: j \in \N)$ as a process taking values in the space of real sequence $\R^{\N}$ (equipped with the product topology). The corresponding class of stochastic integrands is chosen to take values in the set of unbounded  functionals on $\R^{\N}$. This theory of stochastic integration has found applications to mathematical finance, in particular in modelling large markets (see e.g. \cite{DeDonnoGuasoniPratelli:2005, Mostovyi:2018}). In the following paragraphs we describe the main ideas in the construction of the stochastic integral introduced in \cite{DeDonnoPratelli:2006}.

Denote by $\mathcal{U}$ the collection of all the not-necessarily continuous linear functionals on $\R^{\N}$, i.e. $k \in \mathcal{U}$ is a linear functional whose domain $\mbox{Dom}(k)$ is a subspace of $\R^{\N}$ (which can be the trivial set $\{0\}$). A simple integrand (a elementary process in our terminology) is a finite sum of the form $\sum_{k=1}^{n} h_{k} e_{k}$,  
where $h_{k} \in b\mathcal{P}$ for $k=1, \dots, n$ and $(e_{j}:j \in \N)$ is the canonical basis in $\R^{\N}$. For such a simple integrand the stochastic integral is defined as $H \cdot Z= \sum_{k=1}^{n} (h_{k} \cdot Z^{k})$. 

A process $H$ with values in $\mathcal{U}$ is called predictable if 
there exists a sequence $(H^{n}: n \in \N)$ of simple processes such that for each $t \geq 0$, $\omega \in \Omega$, $H(t,\omega)(f) = \lim_{n} H(t,\omega)(f)$  for all $f \in \mbox{Dom}(H(t,\omega))$. 

A $\mathcal{U}$-valued predictable process is said to be integrable with respect to the sequence of semimartingales $Z=(Z^{j}: j \in \N)$ if there is a sequence $(H^{n}: n \in \N)$ of simple integrands, such that $H^{n}$ converges to $H$ a.s. and the sequence of semimartingales $ H^{n} \cdot Z$ converges to a semimartingale $Y$ in $S^{0}$. In this case $H$ is called a generalized integrand and its stochastic integral is defined as  $ H \cdot Z = Y$. 

It is proved in Proposition 1 in  \cite{DeDonnoPratelli:2006} that when the stochastic integral exists, it does not depend on the approximating sequence. This is proven using a result by Memin \cite{Memin:1980} which allows, by mean of an  application  of a change in probability measure, to replace the integral with respect to a sequence of semimartingales with the sum of an integral with respect to a sequence of square integrable martingales and an integral with respect to a sequence of predictable processes with integrable variation. These types of stochastic integrals are developed in \cite{DeDonnoPratelli:2006}. 

Observe that the definition of generalized integrand implicitly assumes continuity of the stochastic integral mapping  $H \mapsto H \cdot Z$. Indeed, it is proven in Theorem 3 in  \cite{DeDonnoPratelli:2006} that the set of that the set of stochastic integrals with respect to $Z$ is closed in $S^{0}$. 

The above definition of stochastic integral for generalized integrands  has its drawbacks. In particular, the stochastic integral lacks some of the standard properties of the stochastic integral in finite dimensions. For instance, the integral is not linear with respect to the integrator process (see Remark 9 and the examples in Section 6 in \cite{DeDonnoPratelli:2006}).  

Motivated by the above, in the next section we apply our recently developed theory of stochastic integration in nuclear spaces to carry out an alternative construction for the stochastic integral with respect to a sequence of real-valued semimartingales. We will consider however locally bounded integrands as in Section \ref{sectExteStocIntegNuclear} instead of the class of generalized integrands of \cite{DeDonnoPratelli:2006}. This way our integral will have at disposal all the tools developed in this article which we hope can be useful for future applications of this theory.  

\subsection{Construction of the Stochastic Integral}\label{subsecIntegSequenSemima}

To define the stochastic integral we first need to settle this problem in our context. Let $\Phi=\oplus_{j \in \N} \R_{j}$ where $\R_{j}\defeq \R$ for every $j \in \N$. The locally convex direct sum $\Phi$ is a complete, barrelled, nuclear space (see \cite{Jarchow}, Propositions 6.6.7 and 11.3.1, and Corollary 21.2.3). Furthermore, we have $\Phi' \simeq \R^{\N} \defeq \prod_{j \in \N} \R_{j}$ with the product topology (see \cite{Schaefer}, Corollary IV.4.3.1) is also a complete, barrelled, nuclear space (see \cite{Schaefer}, Theorems II.5.3. and III.7.4).

Let $(e_{j}:j \in \N)$ be the canonical basis in $\R^{\N}$. Observe that each $\phi \in \Phi$ is of the form 
\begin{equation}\label{eqGeneFormElemDirectSum}
 \phi= \sum_{j \in \N} a_{j} e_{j}, 
\end{equation}
where $a_{j} \in \R$ for all $j \in \N$ and the set $\{ j \in \R: a_{j} \neq 0\}$ is finite.

Let $(Z^{j}: j \in \N)$ be a sequence in $S^{0}$. For $\phi \in \Phi$ of the form \eqref{eqGeneFormElemDirectSum}, define 
\begin{equation}\label{eqDefiCyliSemiSequenceSemi}
X(\phi) \defeq \sum_{j \in \N} a_{j} Z^{j}. 
\end{equation}
Observe that in particular we have $X(e_{j})=Z^{j}$. 
It is clear that $X(\phi) \in S^{0}$ for each $\phi \in \Phi$ and that the mapping $X: \Phi \rightarrow S^{0}$, $\phi \mapsto X(\phi)$, is linear. This way we have that $X$ is a cylindrical semimartingale in $\Phi'$. Moreover, we have:

\begin{theorem}\label{theoCylSemiSequenceSemiContinuous} Let $(Z^{j}: j \in \N)$ be a sequence in $S^{0}$. Let $X$ be the cylindrical semimartingale defined in \eqref{eqDefiCyliSemiSequenceSemi}. Then the mapping  $X: \Phi \rightarrow S^{0}$ is continuous. Furthermore, $X$ has a version that is a $\Phi'$-valued regular, c\`{a}dl\`{a}g semimartingale. 
\end{theorem}
\begin{proof} We first prove the continuity of $X$. 
For each $j \in \N$ denote by $I_{j}: \R_{j} \rightarrow \oplus_{j \in \N} \R_{j}$, $u \mapsto u e_{j}$, the canonical inclusion. From properties of the direct sum topology (see e.g. Sections 4.1 and 4.3 in \cite{Jarchow}) we know that the mapping $X: \Phi \rightarrow S^{0}$ is continuous if and only if for each $j \in \N$ the mapping $X \circ I_{j}$ is continuous from $\R_{j}$ into $S^{0}$. But observe that for each $j \in \N$ we have that $X \circ I_{j}(u)=u Z^{j}$ and since for each $Z \in S^{0}$ the mapping $u \mapsto uZ$ is continuous from $\R$ into $S^{0}$, then we conclude that 
$X: \Phi \rightarrow S^{0}$ is continuous. The final assertion follows from Proposition \ref{propCylSemimContOpeSpaceSemim} and Theorem \ref{theoRegulCylinSemimartigalesNuclear}. 
\end{proof}

Because $\Phi$ is a complete, barrelled, nuclear space,  we can now define the stochastic integral with respect to a sequence of semimartingales via Theorem \ref{theoLocalBoundWeakIntegNucleSpace}.  

\begin{definition}
Let $(Z^{j}: j \in \N)$ be a sequence in $S^{0}$ and let 
$X$ be the cylindrical semimartingale defined by such sequence in \eqref{eqDefiCyliSemiSequenceSemi}. For each $H \in \mathcal{P}_{loc}(\Phi)$, we define the stochastic integral of $H$ with respect to the sequence $(Z^{j}: j \in \N)$  as the stochastic integral $\int \, H \, dX$ of $H$ with respect to $X$ as defined in 
Theorem \ref{theoLocalBoundWeakIntegNucleSpace}. 
\end{definition}

The stochastic integral with respect to a sequence of semimartingales satisfies all the properties listed in the above sections. However, as we will describe below some new properties emerge, in particular in connection with the property of good integrator. 

In effect, the reader might observe that from our construction of the stochastic integral we only require for $(Z^{j}: j \in \N)$ to be a sequence in $S^{0}$. Indeed, by Theorem \ref{theoCylSemiSequenceSemiContinuous} the cylindrical semimartingale $X$ defined in \eqref{eqDefiCyliSemiSequenceSemi} is continuous from $\Phi$ into $S^{0}$. However, as mentioned in Section \ref{sectionCharacWeakIntegransNuclear} the above continuity property is not in general enough to conclude that $X$ is a good integrator, unless $X$ is one of the particular classes of cylindrical semimartingales which we know are good integrators. 

As a new phenomena of the above construction we will show below that under a change of probability measure the cylindrical semimartingale defined by a sequence of real-valued semimartingales is a good integrator.  We will require the following notation: $\mathcal{S}^{0}(\Prob)$ and $\mathcal{M}^{2}_{\infty} \oplus \mathcal{A}(\Prob)$ denote the spaces $\mathcal{S}^{0}$ and $\mathcal{M}^{2}_{\infty} \oplus \mathcal{A}$ defined with respect to the probability measure $\Prob$. 

\begin{proposition}\label{propGoodIntegraChangeInProbab}
Let $(Z^{j}: j \in \N)$ be a sequence in $S^{0}$ and let $X$ be the cylindrical semimartingale defined by such sequence in \eqref{eqDefiCyliSemiSequenceSemi}. There exists a probability measure $\mathbbm{Q}$, equivalent to $\Prob$, with $d\mathbbm{Q}/d\mathbbm{P} \in L^{\infty} (\Prob)$, such that 
\begin{enumerate}
\item $Z^{j}=m^{j}+a^{j} \in \mathcal{M}^{2}_{\infty} \oplus \mathcal{A}(\mathbbm{Q})$ for each $j \in \N$, and $X$ is a $\mathcal{M}^{2}_{\infty} \oplus \mathcal{A}(\mathbbm{Q})$-cylindrical semimartingale such that the mapping $X: \Phi \rightarrow S^{0}(\mathbbm{Q})$ is continuous. 
\item For every $ H \in b\mathcal{P}(\Phi)$,
$$ \int \, H \, dX =  \int \, H \, dM + \int \, H \, dA  \in \mathcal{M}^{2}_{\infty} \oplus \mathcal{A}(\mathbbm{Q}), $$
where $M$ is the $\mathcal{M}^{2}_{\infty}(\mathbbm{Q})$-cylindrical martingale defined as in \eqref{eqDefiCyliSemiSequenceSemi} by the sequence $(m^{j}: j \in \N) \subseteq \mathcal{M}^{2}_{\infty}(\mathbbm{Q})$ and $A$ is the $\mathcal{A}(\mathbbm{Q})$-cylindrical semimartingale defined as in \eqref{eqDefiCyliSemiSequenceSemi} by the sequence $(a^{j}: j \in \N) \subseteq \mathcal{A}(\mathbbm{Q})$. 
\item The stochastic  integral mapping $H \mapsto \int \, H \, dX$ is continuous from $b\mathcal{P}(\Phi)$ into $\mathcal{M}^{2}_{\infty} \oplus \mathcal{A}(\mathbbm{Q})$. In particular, $X$ is a good integrator under the measure $\mathbbm{Q}$.
\end{enumerate}
\end{proposition}
\begin{proof}
Given the sequence $(Z^{j}: j \in \N)$ in $S^{0}$, by a result of Dellacherie (Th\'{e}or\`{e}me 5 in \cite{Dellacherie:1978}; see also Lemme I.3 in \cite{Memin:1980})  there exists a probability measure $\mathbbm{Q}$, equivalent to $\Prob$, with $d\mathbbm{Q}/d\mathbbm{P} \in L^{\infty} (\Prob)$, such that for each $j \in \N$ we have $Z^{j} =m^{j}+a^{j}   \in \mathcal{M}^{2}_{\infty} \oplus \mathcal{A}(\mathbbm{Q})$. Then $X$ defined by \eqref{eqDefiCyliSemiSequenceSemi} is a $\mathcal{M}^{2}_{\infty} \oplus \mathcal{A}(\mathbbm{Q})$-cylindrical semimartingale. 

Now recall from  Theorem \ref{theoCylSemiSequenceSemiContinuous} that $X$ is continuous from $\Phi$ into $S^{0}(\Prob)$. Since $\Prob$ and $\mathbbm{Q}$ are equivalent, the spaces $S^{0}(\Prob)$ and $S^{0}(\mathbbm{Q})$ are homeomorphic (see \cite{DellacherieMeyer}, Remarques VII.100.(c), p.318).  Hence $X$ is continuous from $\Phi$ into $S^{0}(\mathbbm{Q})$. This proves \emph{(1)}. 

The statements \emph{(2)} and \emph{(3)} are a direct consequence of \emph{(1)} and Corollary \ref{coroSpecialSemiIsGoodInteg}. 
\end{proof}

\begin{remark}
The result in Proposition \ref{propGoodIntegraChangeInProbab} has an interesting application. Since both $\Phi$ and $\Phi'$ are complete  barrelled nuclear spaces, under a change in probability measure, for any sequence of real-valued semimartingales the corresponding cylindrical semimartingale is a good integrator, hence the stochastic integral with respect to any such sequence of semimartingales has at disposal all the machinery developed in Sections \ref{sectExteStocIntegNuclear} and \ref{sectStochFubiniSEE}. In particular the Riemann representation, the stochastic integration by parts formula and the stochastic Fubini theorem.   
\end{remark} 

We finalize this section by pointing out the relation of the stochastic integral defined by De Donno and Pratelli in  \cite{DeDonnoPratelli:2006} with the stochastic integral introduced in this section.  To do this, we examine the behaviour of our stochastic integral on the bounded integrands. 

First, from \eqref{eqGeneFormElemDirectSum} we can see that the  elementary processes in $b\mathcal{P}(\Phi)$ are those processes of the form 
\begin{equation}\label{eqElemIntegSequenceSemima}
H_{t}(\omega)=\sum_{k=1}^{n} h_{k}(r,\omega) e_{k}, \quad \forall \, t \geq 0, \, \omega \in \Omega, 
\end{equation}
where $h_{k} \in b\mathcal{P}$ for $k=1, \dots, n$. From Theorem \ref{theoBoundedWeakIntegNucleSpace} we have
\begin{equation}\label{eqIntegSequeSemiElemIntegrands}
 \int \, H \, dX = \sum_{k=1}^{n} h_{k} \cdot X(e_{k})= \sum_{k=1}^{n} \int \, h_{k} \, dZ^{k}. 
\end{equation}
This way our stochastic integral coincide with that in \cite{DeDonnoPratelli:2006} for the elementary processes. 

Now, given any $H \in b\mathcal{P}$ observe that by Lemma \ref{lemmSequeApproxFubini} there exists a sequence of elementary processes $(H_{n}: n \in \N)$ such that $H_{n} \rightarrow H$ in $b\mathcal{P}(\Phi)$. If the cylindrical semimartingale $X$ determined by the sequence $(Z^{j}: j \in \N)$ is a good integrator, then we have 
$\int \, H_{n} \, dX$ converges to $\int \, H \, dX$ in $S^{0}$. Therefore, if $X$ is a good integrator every $H \in b\mathcal{P}(\Phi)$ is a generalized integrand (as defined in Section \ref{subSecLiteReviewSequeSemi}) for the sequence $(Z^{j}: j \in \N)$.  

If the cylindrical semimartingale $X$ determined by the sequence $(Z^{j}: j \in \N)$ is not a good integrator the connection between the two definitions of stochastic integrals is not so clear. However, we 
know from Proposition \ref{propGoodIntegraChangeInProbab} that an equivalent probability measure $\mathbbm{Q}$ exists such that $X$ is a good integrator under the measure $\mathbbm{Q}$.

\section{Final Remarks and Comparison With Other Theories of Stochastic Integration in Locally Convex Spaces}\label{sectRemarkComparisons}

In Section \ref{subSecLocalConvInteg} we introduced the stochastic integral mapping for integrands in the space $\Phi \, \widehat{\otimes}_{\pi} \, b\mathcal{P}$ under the assumption that $\Phi$ is a locally convex space and that $X: \Phi \rightarrow S^{0}$ is a cylindrical semimartingale with a standard continuity property (Assumption \ref{assuCylSemimartingale}). Our construction relies on the good integrator property of real-valued semimartingales and properties of the tensor product of topological vector spaces. 

Under additional assumptions on either $\Phi$ and $X$, we have been able to prove further extensions of our theory of integration. Indeed if either $\Phi$ is metrizable or a complete barrelled nuclear space, we have shown that the integrands $\Phi \, \widehat{\otimes}_{\pi} \, b\mathcal{P}$ can be identified with $\Phi$-valued weakly predictable (locally) bounded processes (Proposition \ref{propInducedBoundPredProcessMetricCase} and Theorem \ref{theoWeakIntegNuclearBoundPrec}). Furthermore, we have shown that for some specific classes of cylindrical semimartingales the stochastic integral mapping is continuous (see Remark \ref{remaLocaConvWeakIntegSpSemima}). In particular if $\Phi$ is a complete, barrelled nuclear space and $X$ is a good integrator, we have shown in Sections \ref{sectExteStocIntegNuclear} to \ref{sectStochFubiniSEE} that the stochastic integral satisfies many of the properties of the stochastic integral in finite dimensions, including a Riemann representation formula, stochastic integration by parts formula, and the stochastic Fubini theorem. 

Other authors have introduced theories of stochastic integration either assuming further properties on the locally convex space or under the assumption that the cylindrical semimartingale is of some specific class. Depending on the above assumptions, it is possible to define the stochastic integral for a class of integrands larger than the weakly predictable (locally) bounded processes.  In the following paragraphs we briefly describe some of these other theories of stochastic integration and their corresponding class of integrands. 

In the papers \cite{MikuleviciusRozovskii:1998, MikuleviciusRozovskii:1999}, Mikulevicius and Rozovskii introduced a theory of stochastic integration with respect to a locally square integrable cylindrical martingale defined on a quasi-complete locally convex space $\Phi$. The covariance structure of such a cylindrical martingale is determined by a predicable family $(Q_{s}: s \geq 0)$ of symmetric non-negative linear forms from $\Phi'$ into $\Phi$ and a predictable increasing process $\lambda_{s}$. Using the Schwartz theory of reproducing kernels, a family of Hilbert spaces $H_{s} \subseteq \Phi$ is associated with the covariance function $Q_{s}$. The class of integrands in    \cite{MikuleviciusRozovskii:1998, MikuleviciusRozovskii:1999} are the $\Phi$-valued weakly predictable processes $h$ satisfying $h(s,\omega)\in H_{s}$ and $\int_{0}^{t} \abs{h(s)}_{H_{s}}^{2} \, d \lambda_{s} < \infty$ $\Prob$-a.e. $\forall t>0$.  

In \cite{FonsecaMora:2018-1} the author of this paper introduced a theory of stochastic integration with respect to a cylindrical martingale-valued measure  defined on a locally convex space $\Phi$. Roughly speaking, a cylindrical martingale-valued measure is a family $M=(M(t,A): t \geq 0, A \in \mathcal{R})$ such that $(M(t,A): t \geq 0)$ is a cylindrical martingale in $\Phi'$ for each $A \in \mathcal{R}$ and $M(t,\cdot)$ is finitely additive on $\mathcal{R}$ for each $t \geq 0$, where $\mathcal{R}$ is a ring of subsets of a topological space $U$. These cylindrical martingale-valued measures are assumed to have (weak) square moments determined by a family of continuous Hilbertian semi-norms $\{ q_{r,u}: r \in \R_{+}, u \in U\}$ on $\Phi$ and two $\sigma$-finite measures $\mu$ and $\lambda$ defined on $(U,\mathcal{B}(U))$ and $(\R_{+},\mathcal{B}(\R_{+}))$ respectively. The class of integrands in \cite{FonsecaMora:2018-1} are families of predictable processes $h$ satisfying $h(r,\omega,u) \in \Phi_{q_{r,u}}$ and $\forall t>0$, $\int_{0}^{t} \, q_{r,u}(h(r,u))^{2} \, \mu(du)  \lambda(dr) < \infty$ $\Prob$-a.e.   

A theory of stochastic integration for weakly predictable locally bounded integrands with respect to semimartingales (in the projective sense) taking values in the dual of a reflexive nuclear space was introduced by \"{U}st\"{u}nel in the papers \cite{Ustunel:1982, Ustunel:1982-1, Ustunel:1982-2}. As mentioned previously in this  article, the theory of stochastic integration developed in Sections \ref{sectStochIntegNuclear} to \ref{sectStochFubiniSEE} generalize the results obtained by \"{U}st\"{u}nel (Remarks \ref{remaSemiDefiUstunel}, \ref{remaUstunelStochInteg} and  \ref{remaUstunelIntegByParts}). 

In \cite{KurtzProtter:1996}, Kurtz and Protter introduced the concept of standard $\mathbbm{H}^{\#}$-semimartingale, which corresponds to a cylindrical semimartingale $Y$ defined on a separable Banach space $\mathbbm{H}$ such that for each $t>0$ the set $ \mathcal{H}_{t}= \left\{ \sup_{s \leq t} \abs{Z_{-} \cdot Y(s)}: Z \in S_{1}  \right\}$ 
is stochastically continuous, where in the above $S_{1}$ denotes the collection of all the $\mathbbm{H}$-valued c\`{a}dl\`{a}g processes of the form $Z(t)=\sum_{k=1}^{m} \xi_{k}(t) h_{k}$ where for $k=1,\dots, m$, $\xi_{k}$ is an adapted real-valued c\`{a}dl\`{a}g process, $h_{k} \in \mathbbm{H}$, $\sup_{s \leq t} \norm{Z(s)} \leq 1$, and $Z_{-} \cdot Y(t) = \sum_{k=1}^{m} \int_{0}^{t} \,  \xi_{k}(s-) dY(h_{k})(s)$. Stochastic integrals are therefore defined for $\mathbbm{H}$-valued predictable processes with locally compact image by a continuity argument. 

The concept of standard $\mathbbm{H}^{\#}$-semimartingale extends the concept of good integrator from finite dimensions to a cylindrical semimartingale defined on a separable Banach space $\mathbbm{H}$. Our concept of good integrator of Definition \ref{defiGoodIntegrators} also extends the concept of good integrator from finite dimension but to a cylindrical semimartingale in the dual $\Phi'$ of a complete barrelled nuclear space $\Phi$. Observe that no intersection of these infinite dimensional versions of the concept of good integration occurs since it is known that only the finite dimensional spaces can be Banach and nuclear at the same time. 

In the context of Banach spaces, there are several theories of (vector-valued) stochastic integration with respect to specific classes of cylindrical semimartingales. For example, in Hilbert spaces we have stochastic integration with respect to cylindrical square integrable martingales and cylindrical L\'{e}vy processes \cite{DaPratoZabczyk, MetivierPellaumail, JakubowskiRiedle:2017}, in separable Banach spaces we have stochastic integration with respect to cylindrical Brownian motion \cite{Kalinichenko:2019}, in Banach spaces of martingale type $p \in [1,2]$ with respect to cylindrical L\'{e}vy processes of order $p$ \cite{KosmalaRiedle:2021}, and in UMD Banach spaces we have stochastic integration with respect to cylindrical Brownian motion \cite{vanNeervenVeraarWeis:2008} and with respect to continuous cylindrical local martingales  \cite{VeraarYaroslavtsev:2016}. 

We are not aware of any theory of stochastic integration in Banach spaces with respect to general cylindrical semimartingales. The stochastic integral mapping of Section \ref{subSecLocalConvInteg} can serve for this purpose but the class of integrands $\Phi \, \widehat{\otimes}_{\pi} \, b\mathcal{P}$ (which we know by Proposition \ref{propInducedBoundPredProcessMetricCase}  that can be identified with $\Phi$-valued weakly predictable bounded processes) might be relatively small and requires a further extension or alternative characterization. In particular, observe that if $\Phi$ has the approximation property (e.g. if $\Phi$ has a Schauder basis) then $\Phi \, \widehat{\otimes}_{\pi} \, b\mathcal{P}$ can be identified with the space of nuclear operators $\mathcal{N}(\Phi',b\mathcal{P})$ (see \cite{KotheII}). We hope we can explore these ideas in the Banach space setting in a latter publication.

\textbf{Acknowledgements} {The author thank an anonymous referee for valuable comments and suggestions that contributed
greatly to improve the presentation of this article.  Thanks also to  Dar\'{i}o Mena-Arias and Ronald Zu\~{n}iga-Rojas for useful discussions.
The author acknowledge The University of Costa Rica for providing financial support through the grant ``B9131-An\'{a}lisis Estoc\'{a}stico con Procesos Cil\'{i}ndricos''.}

\end{document}